\DeclareRobustCommand{\SkipTocEntry}[4]{}
\makeindex

\documentclass[letterpaper,12pt,reqno,latexsym,amsbsy,xypic,mathrsfs,verbatim]{amsart}
\usepackage{amsfonts,amssymb,amsthm}
\usepackage{amsmath,amscd}
\usepackage{pstricks}
\usepackage{pstricks,pst-node}
\usepackage{mathrsfs}
\usepackage[all]{xy}
\usepackage[enableskew,vcentermath]{youngtab}

\renewcommand{\subsection}[1]{\vspace{.18in}\par\noindent\addtocounter{subsection}{1}\setcounter{equation}{0}{\bf\thesubsection.\hspace{5pt}#1}}

\newtheorem{theorem}{Theorem}[section]
\numberwithin{theorem}{section}
\numberwithin{equation}{theorem}

\theoremstyle{definition}
\newtheorem{defn}[theorem]{Definition}%[section]

\newtheorem{rem}[theorem]{Remark}

\theoremstyle{plain}
\newtheorem{prop}[theorem]{Proposition}
\newtheorem{thm}[theorem]{Theorem}
\newtheorem{lem}[theorem]{Lemma}
\newtheorem{cor}[theorem]{Corollary}

\newcommand{\ds}{\displaystyle}

\newcommand{\C}{\mathbb C}
\newcommand{\Z}{\mathbb Z}
\newcommand{\N}{\mathbb N}
\newcommand{\Q}{\mathbb Q}

\newcommand{\mf}{\mathfrak}

\newcommand{\Cl}{{\mathcal C}}
\newcommand{\HC}{\mathcal{H}_r^c}

\newcommand{\g}{\mathfrak{g}}

\newcommand{\al}{\alpha}

\newcommand{\ep}{\epsilon}
\newcommand{\la}{\lambda}
\newcommand{\ga}{\gamma}

\newcommand{\La}{\Lambda}

\newcommand{\ov}{\overline}

\newcommand{\ox}{\bar{x}}
\newcommand{\oX}{{\overline{X}}}

\newcommand{\udj}{\underline{j}}

\newcommand{\mc}{\mathcal}

\newcommand{\ro}{{\rm ro}}
\newcommand{\co}{{\rm co}}

\def\mfq{{\mathfrak q(n)}}
\def\mfg{{\mathfrak g}}
\def\mfn{{\mathfrak n}}
\def\mfh{{\mathfrak h}}
\def\mfb{{\mathfrak b}}

\newcommand{\Qnr}{\mathcal{Q}(n,r)}
\newcommand{\Pnr}{U(n,r)}
\newcommand{\qQnr}{\mathcal{Q}_q(n,r)}
\newcommand{\qPnr}{U_q(n,r)}
\newcommand{\ev}[1]{{#1}_{{0}}}
\newcommand{\od}[1]{{#1}_{{1}}}
\newcommand{\End}{{\rm End}}

\newcommand{\ovHD}{{\overline{h}_D}}

\newcommand{\qUq}{{U_q(\mathfrak q(n))}}
\newcommand{\qUZ}{{U_{q,\mc Z}}}
\newcommand{\bi}{{\bar i}}
\newcommand{\bj}{{\bar j}}
\newcommand{\mcA}{{\mathcal{A}}}
\newcommand{\mcZ}{{\mathcal{Z}}}
\newcommand{\Xije}{{X_{i,j}}}
\newcommand{\Xijo}{{\overline{X}_{i,j}}}
\newcommand{\Xjie}{{X_{j,i}}}
\newcommand{\Xjio}{{\overline{X}_{j,i}}}
\newcommand{\Xkle}{{X_{k,l}}}
\newcommand{\Xklo}{{\overline{X}_{k,l}}}

\newcommand{\Xkje}{{X_{k,j}}}
\newcommand{\Xile}{{X_{i,l}}}

\newcommand{\Xkjo}{{\overline{X}_{k,j}}}
\newcommand{\Xilo}{{\overline{X}_{i,l}}}

\newcommand{\Xijem}{{X^{(m)}_{i,j}}}
\newcommand{\Xkles}{{X^{(s)}_{k,l}}}

\newcommand{\Xilet}{{X^{(t)}_{i,l}}}

\def\sA{{\mathcal A}}

\def\fkm{{\mathfrak m}}\def\fku{{\mathfrak u}}
\def\fkl{{\mathfrak l}}

\def\bsh{{\boldsymbol h}}
\def\bsK{{\boldsymbol K}}
\def\bsb{{\boldsymbol b}}

\def\ddim{{\bf dim}}
\def\vep{{\varepsilon}}
\def\sB{{\mathcal B}}

\def\sBq{{{\mathcal B}_q}}

{\vskip-\lastskip\medskip
  \noindent
  {\em #1.}\enspace
  }%
{\qed\par\medskip
  }

\begin{document}

\title[Presenting queer Schur superalgebras]{Presenting queer Schur superalgebras}
\author[Du and Wan]{Jie Du and Jinkui Wan}
\thanks{Supported by ARC DP-120101436 and NSFC-11101031. The research was carried out while Wan was visiting
the University of New South Wales during the year 2012--2013. The hospitality and support of UNSW are gratefully acknowledged.}

\address{(Du) School of Mathematics,
University of New South Wales,
UNSW Sydney 2052,
Australia.
 }\email{j.du@unsw.edu.au}
\address{(Wan) Department of Mathematics, Beijing Institute of Technology,
Beijing, 100081, P.R. China. } \email{wjk302@gmail.com}

%\address{(Wang) Department of Mathematics, University of Virginia,
%Charlottesville,VA 22904, USA.}
%\email{ww9c@virginia.edu}
\date{\today}
\begin{abstract} Associated to the two types of finite dimensional simple superalgebras, there are the general linear Lie superalgebra and the queer Lie superalgebra. The universal enveloping algebras of these Lie superalgebras act on the tensor spaces of the natural representations and, thus, define certain finite dimensional quotients, the Schur superalgebras and the queer Schur superalgebra. In this paper, we introduce the quantum analogue of the queer Schur superalgebra and investigate the presentation problem for both the queer Schur superalgebra and its quantum analogue.
\end{abstract}

%\keywords{Kostka polynomials, symmetric groups, Schur $Q$-functions,
%Hall-Littlewood functions, $q$-weight multiplicity, Hecke-Clifford
%algebra}

\maketitle

\section{Introduction}

A superspace $V$ is a vector space over a field endowed with
a $\Z_2$-grading (or a parity structure): $V =\ev V\oplus\od V$, where an element in $\ev V$ is called even, while an element in $\od V$ is called odd.
A superalgebra $\sA$ is a $\Z_2$-graded (associative) algebra with 1 over a field. Thus, the underlying space of $\sA$
is a superspace $\sA = \ev \sA\oplus\od \sA$ and the multiplication satisfies
$\sA_i\sA_j\subseteq \sA_{i+j}$, for $i, j \in \Z_2$. It is known (see, e.g., \cite{BK}) that a finite dimensional simple superalgebra over the complex field $\mathbb C$ is either isomorphic to
the (full) matrix superalgebra $\mc M=M_{m+n}(\C)$ with even part $\mc M_0=\bigl\{{A\,0\choose 0\,B}\mid A\in M_m(\mathbb C),B\in M_n(\C)\bigr\}$
and odd part $\mc M_1=\bigl\{{0\,C\choose D\,0}\mid C\in M_{m,n}(\mathbb C),D\in M_{n,m}(\C)\bigr\}$, or isomorphic to the queer matrix superalgebra $\mc Q=\bigl\{{A\,B\choose B\,A}\mid A,B\in M_n(\mathbb C)\bigr\}$ with even part
$\ev{\mc Q}=\bigl\{{A\,0\choose 0\,A}\mid A\in M_n(\mathbb C)\bigr\}$ and odd part $\od{\mc{Q}}=\bigl\{{0\,B\choose B\,0}\mid B\in M_n(\mathbb C)\bigr\}$.

Associated to a superalgebra $\sA$, there is a Lie superalgebra $\sA^-$  equipped
with the super bracket product (or super commutator) defined by
$$
[x,y]:=xy-(-1)^{\hat{x}\cdot\hat{y}}yx,
$$
where $x,y\in \mcA$ are homogeneous elements and $\hat z=i$ if $z\in\sA_i$. Thus, the two type simple superalgebras $\mc M$ and $\mc Q$ give rise to two Lie superalgebras $\mathfrak{gl}({m|n}):=\mc M^-$, the general linear Lie superalgebra,  and  $\mathfrak q(n):=\mc Q^-$, the queer Lie superalgebra.

If $V$ denotes the natural representation of $\mathfrak{gl}({m|n})$ (resp., $\mathfrak q(n)$), then the tensor product $V^{\otimes r}$ is a representation of the universal enveloping algebra $U(\mathfrak{gl}({m|n}))$ (resp., $U(\mathfrak q(n))$). The image of $U(\mathfrak{gl}({m|n}))$ (resp., $U(\mathfrak q(n))$) in End$(V^{\otimes r})$ is called the {\it Schur superalgebra} (resp. {\it queer Schur superalgebra} or {\it Schur superalgebra of type Q}, following \cite{BK}).

%The Schur/queer Schur superalgebras and their representations were introduced and investigated by several authors including Donkin \cite{Do},
%Brundan--Kleshchev \cite{BK}, and Brundan--Kujawa \cite{BKj} almost over ten years ago. Recently, the study of quantum Schur supperalgebras
%has made substantial progress; see \cite{Mi, DR, TK, DGu, DGW}.
%In this paper, we will introduce the quantum analogue of queer Schur superalgebras and will follow the works \cite{DG},  \cite{DP},
%and \cite{TK} to determine a presentation for the queer Schur superalgebra and its quantum analogue.

The Schur superalgebras and their representations were introduced and investigated by several authors including Donkin \cite{Do} and Brundan--Kujawa \cite{BKj} almost over ten years ago. Recently, the study of quantum Schur superalgebras has made substantial progress; see \cite{Mi, DR, TK, DGu, DGW}.
In particular, in \cite{TK}, El Turkey and Kujawa provided a presentation of the Schur superalgebras and their quantum analogues,
which generalizes the work of Doty and Giaquinto \cite{DG} for (quantum) Schur algebras.

It is known that the queer Lie superalgebra $\mfq$ differs drastically from the basic classical Lie superalgebras.
For example, the Cartan subalgebra of $\mfq$ is not purely even  and there is no invariant bilinear form on $\mfq$.
On the other hand, $\mfq$ behaves in many aspects as the Lie algebra $\mathfrak{gl}(n)$.
In particular, there exists a beautiful analogue of the Schur-Weyl duality discovered by Sergeev \cite{Se}, often referred as
Sergeev duality. In \cite{Ol}, Olshanski constructed a quantum deformation $U_q(\mfq)$ of the universal enveloping algebra $U(\mfq)$
and established a quantum analog of the Sergeev duality in the generic case.

The queer Schur superalgebra was introduced and studied by Brundan and Kleshchev \cite{BK},
and then they determined the irreducible projective representations of the symmetric group $\mf S_r$ via Sergeev duality.
In this paper, we will introduce the quantum analogue of queer Schur superalgebras and
will follow the works \cite{DG} and \cite{DP} to determine a presentation for the queer Schur superalgebra (see \cite{BGKJW} for a more general setting involving walled Brauer-Clifford superalgebras) and its quantum analogue,
which was stated as an interesting problem in \cite{TK}.
In particular, in the quantum case we also establish the existence of $\mc Z$-form for the quantum superalgebra $U_q(\mfq)$.
This is based on a lengthy but straightforward calculation of the commutation formulas for the divided powers of root vectors.

We organise the paper as follows. We first investigate a presentation of the queer Schur superalgebra in the first three sections. More precisely, we study in \S2 the basics of the queer Lie superalgebra and its universal enveloping superalgebra, and establish the commutation formulas for divided powers of root vectors and a Kostat $\Z$-form in \S3. A presentation for the queer Schur superalgebra  is given in \S4. From \S5 onwards, we investigate the quantum case. We start in \S5 with the Olshanski presentation (via a certain matrix in $\End(V)^{\otimes 2}$ satisfying the quantum Yang-Baxter equation)
and the Drinfeld--Jimbo type presentation for the quantum queer superalgebra and introduce all quantum root vectors. We compute all commutation formulas for these vectors in \S6 and for those with higher order in \S7. A Lusztig type form for the quantum queer superalgebra is introduced and certain quotients are investigated in \S8. Finally, we solve the presentation problem for the quantum queer Schur superalgebra in the last section.

{\it Throughout the paper, let $\Z_2=\{0,1\}$. We will use a two-fold meaning for $\Z_2$. We will regard $\Z_2$ as an abelian group when we use it to describe a superspace. However,  for a matrix or an $n$-tuple with entries in $\Z_2$, we will regard it as a subset of $\Z$.

%The dimension vector of a superspace V is the tuple
%$\ddim V =(\dim \ev V,\dim\od V)$ and the dimension of $V$ is $\dim V=\dim V_0+\dim V_1$. We define $\hat v=i$ if $v\in V_i$. %Moreover, the notation $\hat{v}$ for $v\in V$ always implicitly assume that $v$ is a homogeneous element.
}

% A module $M$ over a superalgebra $\sA$ is always understood
%in the $\Z_2$-graded sense, that is $M = \ev M\oplus \od M$ such that $\sA_iM_j\subseteq M_{i+j}$, for $i, j \in \Z_2$. %%
%%
\section{The queer Lie superalgebra $\mfq$ and the associated Schur superalgebra $\Qnr$}\label{queerSchur}
The ground field in this section is the field $\Q$ of rational numbers.
It is known that the general linear Lie superalgebra
$\mathfrak{gl}(n|n)$ consists of matrices of the form
\begin{equation}\label{gmatrix}
\begin{pmatrix}
A&B\\
C&D
\end{pmatrix},
\end{equation}
where $A,B,C,D$ are arbitrary $n\times n$ matrices, and
the rows and columns of \eqref{gmatrix} are labelled by the set
$$
I(n|n)=\{1,2,\ldots, n,-1,-2,\ldots,-n\}.
$$
For $ i,j\in I(n|n)$, denote by $E_{i,j}\in\mathfrak{gl}(n|n)$
the matrix unit with 1 at the $(i,j)$ position and 0 elsewhere. The set $\{E_{i,j}\mid i,j\in I(n|n)\}$ is a basis of $\mathfrak{gl}(n|n)$
and the $\Z_2$-grading on $\mathfrak{gl}(n|n)$ is defined via $\widehat{E}_{i,j}=0$ if $ij>0$ and $\widehat{E}_{i,j}=1$ if $ij<0$.
Then the Lie bracket in $\mathfrak{gl}(n|n)$ is given by
\begin{equation}\label{gl-bracket}
[E_{i,j},E_{k,l}]=\delta_{jk}E_{i,l}-(-1)^{\widehat{E}_{i,j}\cdot \widehat{E}_{k,l}}\delta_{il}E_{k,j}.
\end{equation}
The queer Lie superalgebra, denoted by $\mfg=\mfq$, is the
subalgebra of the general linear Lie superalgebra
$\mathfrak{gl}(n|n)$ consisting of matrices of the form
\begin{align}   \label{qmatrix}
\begin{pmatrix}
A&B\\
B&A
\end{pmatrix},
\end{align}
where $A$ and $B$ are arbitrary $n\times n$ matrices.
The even (resp., odd) part $\ev\mfg$ (resp., $\od\mfg$) consists of
those matrices of the form \eqref{qmatrix} with $B=0$
(resp., $A=0$).
We fix $\mfh$ to be the standard Cartan subalgebra of $\mfg$ consisting of
matrices of the form ~(\ref{qmatrix})~ with
$A, B$ being arbitrary diagonal.
Then the algebra $\ev\mfh$ has a basis $\{h_1,\ldots,h_n\}$ and
$\od\mfh$ has a basis $\{h_{\bar 1},\ldots,h_{\bar n}\}$, where
\begin{equation}\label{cartan}
h_i=E_{i,i}+E_{-i,-i}
, \quad
h_{\bar i}=E_{i,-i}+E_{-i,i}.
\end{equation}
Fix the triangular decomposition
$$
\g=\mfn^-\oplus\mfh\oplus\mfn^+,
$$
where $\mfn^+$ (resp.,  $\mfn^-$) is the  subalgebra of $\mfg$
which consists of matrices of the form ~(\ref{qmatrix})~ with
$A, B$ being arbitrary upper triangular (resp.,
lower triangluar) matrices.
Observe that the even subalgebra $\ev \mfh$ of $\mfh$ can be identified with the standard Cartan subalgebra of $\mathfrak{gl}(n)$ via
the natural isomorphism
\begin{equation}\label{qnev}
\ev{\mfq} \cong \mathfrak{gl}(n).
\end{equation}
Let $\{\epsilon_i \mid i=1, \ldots, n\}$ be the basis for $\ev\mfh^*$ dual to the
standard basis $\{h_i \mid i=1, \ldots, n\}$
for $\ev\mfh$ and we define a bilinear form $(\cdot,\cdot)$ on $\ev\mfh^*$
via
\begin{equation}\label{form}
(\epsilon_i,\epsilon_j):=\epsilon_j(h_i)=\delta_{ij}.
\end{equation}
For $\alpha\in\ev\mfh^*$, let $\mfg_\alpha=\{x\in\mfg\mid [h,x]=\alpha(h)x \text{ for all } h\in\ev\mfh\}$. Then we have the root superspace decomposition
$\mfg=\mfh\oplus\bigoplus_{\alpha\in\Phi}\mfg_\alpha$ with the root system
$$
\Phi=\{\alpha_{i,j}:=\epsilon_i-\epsilon_j|1\le i\not=j\le n\}.
$$
The set of
positive roots corresponding to the Borel subalgebra $\mfb=\mfh\oplus\mfn^+$ is
$$
\Phi^+=\{\alpha_{i,j}\mid1\le i<j\le n\}.
$$
Observe that each root superspace $\mfg_{\al}$ has dimension vector\footnote{The dimension vector of a superspace V is the tuple
$\ddim V =(\dim \ev V,\dim\od V)$.}  $(1,1)$ for $\al\in\Phi$.
Let $\al_i=\al_{i,i+1}$ for $1\leq i\leq n-1$.
Then the root space $\mfg_{\al_i}$ is spanned by $\{e_i,e_{\bi}\}$ with
\begin{equation}\label{simplerootp}
e_i=E_{i,i+1}+E_{-i,-i-1}, \quad
e_{\bar i}=E_{i,-i-1}+E_{-i,i+1},
\end{equation}
while the root space $\mfg_{-\al_i}$ is spanned by $\{f_i,f_{\bi}\}$ with
\begin{equation}\label{simplerootn}
f_i=E_{i+1,i}+E_{-i-1,-i}, \quad
f_{\bar i}=E_{i+1,-i}+E_{-i-1,i}.
\end{equation}
Moreover, $\al_{i,j}=\al_i+\cdots+\al_{j-1}$ for all $1\leq i<j\leq n$.
Let
$${\mc P}:=\oplus^n_{i=1}\Z\ep_i,\quad(\text{resp.}, {\mc P}_{\geq0}=\oplus^n_{i=1}\N\ep_i)$$
be the weight lattice (reps., positive weight lattice) of $\mfq$.

The universal enveloping superalgebra $U=U(\mfq)$ is obtained from the tensor algebra $T(\mfq)$
by factoring out the ideal generated by the elements
$[u,v]-u\otimes v+(-1)^{\hat{u}\cdot\hat{v}}v\otimes u$ for $u\in\mfg_i,v\in\mfg_j$ with $i,j\in\Z_2$,
where $[u,v]$ denotes the Lie bracket of $u,v$ in $\mfq$.
It inherits  a $\Z_2$-grading from $\mfq$.

\begin{prop}[{\cite[Proposition 1.1]{GJKK}, cf.(\cite{LS})}]\label{Uqn}
The universal enveloping superalgebra $U(\mfq)$ is the associative superalgebra over $\Q$
generated by even generators
$ h_i, e_j, f_j,$
and odd generators $h_{\bar i}, e_{\bar j}, f_{\bar j},$
with $1\leq i\leq n$ and $1\leq j\leq n-1$
subject to the following relations:
\begin{itemize}
 \item[(QS1)]
 $[h_i,h_j]=0,\quad [h_i,h_{\bar j}]=0,\quad [h_{\bar i},h_{\bar j}]=\delta_{ij}2h_i$;

\vspace{0.1in}

 \item[(QS2)]
 $[h_i,e_j]=(\ep_i,\alpha_j)e_j,~ [h_i,e_{\bar j}]=(\ep_i,\alpha_j)e_{\bar j},$~
$[h_i,f_j]=-(\ep_i,\alpha_j)f_j,~ [h_i,f_{\bar j}]=-(\ep_i,\alpha_j)f_{\bar j}$;

\vspace{0.1in}

 \item[(QS3)]
 $[h_{\bar i},e_j]=(\ep_i,\alpha_j)e_{\bar j},\quad [h_{\bar i},f_j]=-(\ep_i,\alpha_j)f_{\bar j},$

\noindent  $
 [h_{\bar i},e_{\bar j}]=\left\{
 \begin{array}{ll}
 e_j,\quad \text{if }i=j\text{ or }j+1,\\
 0,\quad\text{ otherwise},
 \end{array}
 \right.
 $
 $
 [h_{\bar i},f_{\bar j}]=\left\{
 \begin{array}{ll}
 f_j,\quad \text{if }i=j\text{ or }j+1,\\
 0,\quad\text{ otherwise};
 \end{array}
 \right.
 $

\vspace{0.1in}

 \item[(QS4)]
 $[e_i,f_j]=\delta_{ij}(h_i-h_{i+1}),\quad [e_{\bar i},f_{\bar j}]=\delta_{ij}(h_i+h_{i+1}),$

\noindent  $[e_{\bar i},f_j]=\delta_{ij}(h_{\bar i}-h_{\overline{i+1}}),\quad [e_{i},f_{\bar j}]=\delta_{ij}(h_{\bar i}-h_{\overline{i+1}})$;

\vspace{0.1in}

 \item[(QS5)]
 $[e_i,e_{\bar j}]=[e_{\bar i},e_{\bar j}]=[f_i,f_{\bar j}]=[f_{\bar i},f_{\bar j}]=0$ for $|i-j|\neq 1$,

\noindent  $[e_i,e_j]=[f_i,f_j]=0$ for $|i-j|>1$,

\noindent $[e_i,e_{i+1}]=[e_{\bar i},e_{\overline{i+1}}],~ [e_i,e_{\overline{i+1}}]=[e_{\bar i},e_{i+1}],$~
  $[f_{i+1},f_i]=[f_{\overline{i+1}},f_{\bar i}],~ [f_{i+1},f_{\bar i}]=[f_{\overline{i+1}},f_{i}]$;

\vspace{0.1in}

  \item[(QS6)]
 $[e_i,[e_i,e_j]]=[e_{\bar i},[e_i,e_j]]=0$,~ $[f_i,[f_i,e_j]]=[f_{\bar i},[f_i,f_j]]=0$ for $|i-j|=1$.
\end{itemize}
\end{prop}

Let $U^+$ (resp. $U^0$ and $U^-$) be the subalgebra of $U=U(\mfq)$ generated by the elements
$e_i, e_{\bi}$ (resp. $h_j, h_{\bj}$, and $f_i,f_{\bi}$), where $1\leq i\leq n-1,1\leq j\leq n$.
Then, similar to Lie algebras, we have the Poincar\'{e}-Birkhoff-Witt (PBW) Theorem and triangular decomposition as follows; see  \cite[Theorem 1.32]{CW}, \cite[(1.3)]{GJKK}, \cite[Theorem 2.1]{Ro}.

\begin{prop} \label{PBW-qn} \begin{enumerate}
\item Suppose that $\{z_1,\ldots,z_p\}$ is a homogeneous basis for $\mfq$.
Then the set
$$
\{z_1^{a_1}z_2^{a_2}\cdots z_p^{a_p}\mid a_1,\ldots,a_p\in\N, a_i\in\Z_2\text{ if }z_i\text{ is odd }, 1\leq i\leq p\}
$$
is a basis for $U(\mfq)$.

\item The algebra $U(\mfq)$ has the triangular decomposition
$
U(\mfq)\cong U^-\otimes U^0\otimes U^+.
$
\end{enumerate}
\end{prop}
Let $V$ be the vector superspace over $\Q$ with dimension vector $\ddim V =(n,n)$ .
Fix a basis $\{v_1,\ldots, v_n\}$ for $\ev V$ and a basis $\{v_{-1},\ldots,v_{-n}\}$ for  $\od V$, respectively.
Then there is a natural action of the algebra $\mathfrak{gl}(n|n)$ on $V$ given by left multiplication,
that is,
\begin{equation}\label{Eij}
E_{i,j}v_k=\delta_{jk}v_i
\end{equation}
for $i,j,k\in I(n|n)$.
The restriction to the algebra $\mfq$ implies that $V$ naturally affords a representation of $U=U(\mfq)$.
As $U(\mfq)$ admits a comultiplication $U(\mfq)\rightarrow U(\mfq)\otimes U(\mfq)$ given on elements of $\mfq$
by $x\mapsto x\otimes 1+1\otimes x$, we have that the $r$-fold tensor product $V^{\otimes r}$ of the natural module $V$
also affords a $U(\mfq)$-module.
Let $\phi_r$ denote the corresponding superalgebra homomorphism:
$$
\phi_r:U(\mfq)\rightarrow \End_\Q(V^{\otimes r}).
$$
Define the {\em queer Schur superalgebra } (also known as Schur superalgebra of type $Q$, cf. \cite{BK}) to be
\begin{equation}
\Qnr=\phi_r(U(\mfq)),\label{Qr1}
\end{equation}
that is, the image of $\phi_r$.
Therefore, $\Qnr$ can be viewed as a quotient of $U(\mfq)$.

Similar to Schur algebras associated to $\mathfrak{gl}(n)$,
there is another way to define the queer Schur superalgebra $\Qnr$
via the known Schur-Sergeev duality as follows.
Denote by $\Cl_r$ the
Clifford superalgebra generated by odd elements $c_1,\ldots,c_r$
subject to the relations
\begin{equation}  \label{Cl}
c_i^2=-1,~c_ic_j=-c_jc_i, \quad 1\leq i\neq j\leq r.
\end{equation}
Denote by $\mf{H}^c_r=\mathcal{C}_r\rtimes \mf S_r$ the so-called {\em Sergeev superalgebra}, which is
generated by the even elements $s_1,\ldots,s_{r-1}$ and the odd
elements $c_1,\ldots,c_r$  subject to \eqref{Cl},
and the additional relations:
\begin{align*}
s_i^2&=1, ~s_is_j=s_js_i,\quad 1\leq i,j\leq r-1,~ |i-j|>1,\\
s_is_{i+1}s_i&=s_{i+1}s_is_{i+1},\quad 1\leq i\leq r-2,\\
s_ic_i&=c_{i+1}s_i,~ s_ic_j=c_js_i,\quad 1\leq i\leq r-1, ~1\leq j\leq r, ~j\neq i,i+1.
\end{align*}
Then by~\cite[Lemma 2]{Se}, we have
a representation $( \psi_r,V^{\otimes r})$ of $\mf H^c_r$ on $V^{\otimes r}$
defined by
\begin{align*}
\psi_r(s_k)(v_{j_1}\otimes  \cdots\otimes v_{j_k}\otimes v_{j_{k+1}}\otimes\cdots \otimes v_{j_r}) =
(-1)^{\hat{v}_{j_k}\cdot\hat{v}_{j_{k+1}}}v_{j_1}\otimes
 \cdots\otimes v_{j_{k+1}}\otimes v_{j_k}\otimes\cdots \otimes v_{j_r},\\
\psi_r(c_l)(v_{j_1}\otimes  \cdots\otimes v_{j_l}\otimes \cdots \otimes v_{j_r}) =
 (-1)^{(\hat{v}_{j_1} +\ldots +\hat{v}_{j_{l-1}})} v_{j_1}\otimes
   \ldots \otimes v_{j_{l-1}} \otimes J_V(v_{j_l}) \otimes\ldots \otimes v_{j_r},
\end{align*}
for all $j_1,\ldots,j_r\in I(n|n), 1\leq k\leq r-1, 1\leq l\leq r$, where $J_V\in\End(V)$ satisfies
$J_V(v_a)=v_{-a}$ and $J_V(v_{-a})=-v_a$ for $1\leq a\leq n$.
Then a classical result \cite[Theorem 3]{Se} of Sergeev says
\begin{equation}
\Qnr=\End_{\mf H^c_r}(V^{\otimes r}),\label{Qr2}
\end{equation}
which was introduced and studied in \cite{BK}.
In particular, we note that $\Qnr$ is naturally a subsuperalgebra of the Schur superalgebra
associated to the Lie superalgebra $\mf{gl}(n|n)$.
\section{Commutation formulas for root vectors and Kostant $\Z$-form}
For $\al_{i,j}=\ep_i-\ep_j\in\Phi$ with $1\leq i\neq j\leq n$, we introduce
the root vectors in $\mfq$ as follows:
\begin{equation}\label{root}
\aligned
x_{i,j}\equiv x_{\al_{i,j}}=E_{i,j}+E_{-i,-j}, \quad
\ox_{i,j}\equiv\ox_{\al_{i,j}}=E_{i,-j}+E_{-i,j}.
\endaligned
\end{equation}
Clearly, $x_{\al}$ is even and $\ox_\al$ is odd for $\al\in\Phi$.
Observe that the even element $x_{\al}$ for $\al\in\Phi$ correspond to the usual root vectors in $\mathfrak{gl}(n)$
under the identification~\eqref{qnev}. Moreover the set
$\{x_{\al},\ox_{\al}\mid\al\in\Phi\}\cup\{h_i,h_{\bar{i}}\mid 1\leq i\leq n\}$
is a homogeneous basis for $\mfq$.

By~\eqref{root} and \eqref{gl-bracket}, a direct calculation gives rise to the following
super commutator formulas.%shows that the following holds for $1\leq i\neq j,k\neq l\leq n$:
\begin{lem}\label{root-comm} For $\al_{i,j},\al_{k,l}\in\Phi$ satisfying $\al_{i,j}+\al_{k,l}\in\Phi$, let $\vep=\vep_{i,j;k,l}=\bigg\{\displaystyle{{1,\;\;\;\;\text{ if }j=k;}\atop{-1,\;\text{ if }i=l.}}$ Then we have in $\mathfrak{q}(n)$
\begin{enumerate}
\item $[x_{i,j},x_{k,l}]
=\left\{
\begin{array}{ll}
h_i-h_j,& \text{ if }\al_{i,j}+\al_{k,l}=0;\\
\vep x_\beta,&\text{ if }\beta=\al_{i,j}+\al_{k,l}\in\Phi;\\
0,&\text{ otherwise},
\end{array}
\right.$
\item $[x_{i,j},\ox_{k,l}]=\left\{
\begin{array}{ll}
h_{\bar i}-h_{\bar j},&  \text{ if }\al_{i,j}+\al_{k,l}=0;\\
\vep\ox_{\beta},&\text{ if }\beta=\al_{i,j}+\al_{k,l}\in\Phi;\\
0,&\text{ otherwise},
\end{array}
\right.$
\item $[\ox_{i,j},\ox_{k,l}]
=\left\{
\begin{array}{ll}
h_{i}+h_{ j},&  \text{ if }\al_{i,j}+\al_{k,l}=0;\\
x_{\beta},&\text{ if }\beta=\al_{i,j}+\al_{k,l}\in\Phi;\\
0,&\text{ otherwise},
\end{array}
\right.$
\item $[h_k,x_{i,j}]=(\ep_k,\al_{i,j})x_{i,j},\quad [h_k,\ox_{i,j}]=(\ep_k,\al_{i,j})\ox_{i,j},$\\
$[h_{\bar k},x_{i,j}]=(\ep_k,\al_{i,j})\ox_{i,j},\quad [h_{\bar k},\ox_{i,j}]=|(\ep_k,\al_{i,j})|x_{i,j}.$
\end{enumerate}
\end{lem}
By Lemma \ref{root-comm}(3) and (QS1), we have for $\al\in\Phi$ and $1\leq i\leq n$:
\begin{equation}\label{odd-square}
\ox^2_{\al}=0,\quad h^2_{\bi}=h_i.
\end{equation}
For $\al\in \Phi$ and $1\leq i\leq n, k\in\N$, we introduce the following elements:
$$
x_{\al}^{(k)}=\frac{x_\al^k}{k!},\quad \begin{pmatrix}
h_i\\
0
\end{pmatrix}
=1,\,
\text{ and }
\begin{pmatrix}\,
h_i\\
k
\end{pmatrix}
=\frac{h_i(h_i-1)\cdots (h_i-k+1)}{k!}\; (k\geq 1).
$$

It is known that $\mfq$ can naturally be viewed as a subspace of the universal enveloping algebra $U(\mfq)$ and
moreover, for any homogeneous $u,v\in\mfq$, we have $uv-(-1)^{\hat{u}\cdot\hat{v}}vu=[u,v]$ in $U(\mfq)$, where $[u,v]$ means the Lie bracket
of $u,v$ in $\mfq$. Then by Lemma~\ref{root-comm}, we have the following.
\begin{prop}\label{div-root}
Maintain the notation defined in Lemma~\ref{root-comm} and let $m,s\in\N$.
Then the following holds in $U(\mfq)$:
\begin{align*}
(1)\;x^{(m)}_{i,j}x^{(s)}_{k,l}&=\left\{
\begin{array}{ll}
x^{(s)}_{k,l}x^{(m)}_{i,j}+\!\!\displaystyle\sum^{\min(m,s)}_{t=1}\!\!x^{(s-t)}_{k,l}\begin{pmatrix}h_{i}-h_j-m-s+2t\\ t\end{pmatrix}
x^{(m-t)}_{i,j}, &\text{ if }\al_{i,j}+\al_{k,l}=0;\\
x^{(s)}_{k,l}x^{(m)}_{i,j}+\displaystyle\sum^{\min(m,s)}_{t=1}\vep^t x^{(s-t)}_{k,l}x^{(t)}_{\beta}
x^{(m-t)}_{i,j},&\!\!\!\!\!\!\!\!\!\!\!\!\! \text{ if }\beta=\al_{i,j}+\al_{k,l}\in\Phi;\\
x^{(s)}_{k,l}x^{(m)}_{i,j},&\text{ otherwise}.
\end{array}
\right.\\
(2)\;x^{(m)}_{i,j}\ox_{k,l}&=\left\{
\begin{array}{ll}
\ox_{k,l}x^{(m)}_{i,j}+(h_{\bar i}-h_{\bar j})x^{(m-1)}_{i,j}-\ox_{i,j}x^{(m-2)}_{i,j}, &\text{ if }\al_{i,j}+\al_{k,l}=0;\\
\ox_{k,l}x^{(m)}_{i,j}+\vep \ox_{\beta}x^{(m-1)}_{i,j}, &\text{ if }\beta=\al_{i,j}+\al_{k,l}\in\Phi;\\
\ox_{k,l}x^{(m)}_{i,j},&\text{ otherwise}.
\end{array}
\right.\\
(3)\;\;\,\ox_{i,j}\ox_{k,l}&=\left\{
\begin{array}{ll}
-\ox_{k,l}\ox_{i,j}+(h_{i}+h_{j}), &\text{ if }\al_{i,j}+\al_{k,l}=0;\\
-\ox_{k,l}\ox_{i,j}+x_{\beta}, &\text{ if }\beta=\al_{i,j}+\al_{k,l}\in\Phi;\\
-\ox_{k,l}\ox_{i,j},&\text{ otherwise}.
\end{array}
\right.\\
(4)\;\;\,x^{(m)}_{i,j}h_{\bar k}&=h_{\bar k}x^{(m)}_{i,j}-(\ep_k,\al_{i,j})\ov{x}_{i,j}x^{(m-1)}_{i,j},\quad
\ox_{i,j}h_{\bar k}=-h_{\bar k}\ox_{i,j}-|(\ep_k,\al_{i,j})|x_{i,j}.\\
(5)\;\;\,x^{(m)}_{i,j}h_k&=(h_k-m(\ep_k,\al_{i,j}))x^{(m)}_{i,j}, \quad
\ox_{i,j}h_k=(h_k-(\ep_k,\al_{i,j}))\ox_{i,j}.
\end{align*}
\end{prop}
\begin{proof}
Since the even root vectors $x_\al$ can be identified with the usual root vectors in $\mathfrak{gl}(n)$
under the identification~\eqref{qnev}, part (1) follows from
the classical case (cf. \cite[(5.11a-c)]{DG}).
Now suppose $\al_{i,j}+\al_{k,l}=0$. Then $i=l,j=k$.
By Lemma \ref{root-comm}(2), the following holds:
\begin{align*}
x^m_{i,j} \ox_{j,i}
=&\ox_{j,i} x^m_{i,j}+\sum^{m-1}_{d=0}x^d_{i,j} \cdot (h_{\bi}-h_{\bj})\cdot x^{m-1-d}_{i,j}\\
=&\ox_{j,i} x^m_{i,j}+\sum^{m-1}_{d=0} \big((h_{\bi}-h_{\bj}) x^d_{i,j}-\sum^{d-1}_{t=0}x^t_{i,j}\cdot 2\ov{x}_{i,j}\cdot x^{d-1-t}_{i,j} \big)x^{m-1-d}_{i,j}\\
=&\ox_{j,i} x^m_{i,j}+\sum^{m-1}_{d=0} \big((h_{\bi}-h_{\bj}) x^d_{i,j}-2d\ov{x}_{i,j} x^{d-1}_{i,j} \big)x^{m-1-d}_{i,j}\\
=&\ox_{j,i} x^m_{i,j}+m(h_{\bi}-h_{\bj}) x^{m-1}_{i,j}-m(m-1)\ov{x}_{i,j} x^{m-2}_{i,j},
\end{align*}
where the second and third equalities are due to the facts
$x_{i,j} (h_{\bi}-h_{\bj})=(h_{\bi}-h_{\bj}) x_{i,j}-2\ox_{i,j}$ and $x_{i,j}\ov{x}_{i,j}=\ov{x}_{i,j} x_{i,j}$ by Lemma \ref{root-comm}.
Hence, the first case of part (2) holds.
Similarly, the other cases  can be verified.
\end{proof}

Define the {\em Kostant $\Z$-form} $U_\Z$ (cf. \cite[Section 4]{BK2}) to be the $\Z$-subalgebra of $U$ generated by
$$
\big\{e_i^{(k)},f_i^{(k)},e_{\bar i},f_{\bar i}\mid 1\leq i\leq n-1, k\in\N\big\}
\bigcup\bigg\{\begin{pmatrix}h_i\\k\end{pmatrix}, h_{\bar i}\,\bigg|\, 1\leq i\leq n, k\in\N\bigg\}.
$$
Denote by $U_\Z^+$ (resp. $U_\Z^-$) the $\Z$-subalgebra of $U$ generated by
$e_i^{(k)},e_{\bi}$ (resp. $f_i^{(k)},f_{\bi}$), where $1\leq i\leq n-1$ and $k\in\N$.
Let $U_\Z^0$ be the $\Z$-subalgebra of $U$ generated by $\begin{pmatrix}h_i\\k\end{pmatrix}, h_{\bi}~ (1\leq i\leq n,k\in\N)$.

%For simplicity, we also need to introduce the following set
%\begin{equation}\label{Upn}
%\Upsilon(n)=\big\{A=(A^0_\al,A^1_\al)_{\al\in\Phi^+}\mid A^0_\al\in\N, A^1_\al\in\Z_2\big\}.
%\end{equation}
%For $A=(A^0_\al,A^1_\al)_{\al\in\Phi^+}\in\Upn$, we set $|A|=\sum_{\al\in\Phi^+}(A^0_{\al}+A^1_{\al})$ and define
%\begin{align}\label{eAfA}
%e_{A^+}=\prod_{\al\in\Phi^+}x_\al^{(A^0_\al)}\prod_{\al\in\Phi^+}\ox^{A^1_\al}_{\al},\quad
%f_{A^-}=\prod_{\al\in\Phi^+}x_{-\al}^{(A^0_{\al})}\prod_{\al\in\Phi^+}\ox^{A^1_{\al}}_{-\al},
%\end{align}
%where the products are taken in any fixed order in $\Phi^+$ and $A^0=(a^0_{ij}), A^1=(a^1_{ij})$.

%Let $$M_n(\N|\Z_2):=M_n(\N)\times M_n(\Z_2).$$
%For an $n\times n$ matrix $X$, let $X=X^-+X^0+X^+$ be the decomposition of $X$ into lower triangular, diagonal, and upper triangular parts of $X$, and
%for each $A=(A_0,A_1)\in M_n(\N|\Z_2)$, let $A^\vep=(A_0^\vep,A_1^\vep)$ for every $\vep\in\{+,0,-\}$. We may identify $A^+$ (or $A^-$) as an elements of $\Upsilon(n)$ and $A^0$ as an element of
%$\N^n\times\Z_2^n$. Thus,
%there is a bijection:
%\begin{equation}\label{jmath}
%\aligned
%\jmath:M_n(\N|\Z_2)&\longrightarrow \Upsilon(n)\times \N^n\times\Z_2^n\times\Upsilon(n);\\
%                                       A&\longmapsto(A^+,A_0^0,A_1^0,A^-),
%                                        \endaligned
%                                        \end{equation}

For $\bsb\in\N^n$, set $|\bsb|=b_1+\cdots+b_n$ and define
\begin{equation}\label{HB}
{\bsh\choose \bsb}=\prod^n_{i=1}\begin{pmatrix}h_i\\ b_i\end{pmatrix}.
\end{equation}
For $D=(D_1,\ldots,D_n)\in\Z_2^n$, set
$$
\ovHD=h^{D_1}_{\bar 1}\cdots h^{D_n}_{\bar n}.
$$
For the subset $\Z_2$ of $\N$,
let $$M_n(\N|\Z_2):=M_n(\N)\times M_n(\Z_2).$$
For an $n\times n$ matrix $X$, let $X=X^-+X^0+X^+$ be the decomposition of $X$ into lower triangular, diagonal, and upper triangular parts of $X$, and
for each $A=(A_0,A_1)\in M_n(\N|\Z_2)$, let $A^\vep=(A_0^\vep,A_1^\vep)$ for every $\vep\in\{+,0,-\}$ and define
\begin{align}\label{eAfA}
e_{A^+}=\prod_{1\leq i<j\leq n}(x_{i,j}^{(a^0_{ij})}~\ox^{a^1_{ij}}_{{i,j}}),\quad
f_{A^-}=\prod_{1\leq i<j\leq n}(x_{j,i}^{(a^0_{ji})}~\ox^{a^1_{ji}}_{j,i}),
\end{align}
where $A_0=(a^0_{ij}), A_1=(a^1_{ij})$, and the order in the  products are defined as follows (cf. \cite{DP}).
For the $j$th row (reading to the right) $a_{j,j+1},\ldots, a_{jn}$ of $A^+$, put
$$
\pi_j(A^+)=(x^{(a^0_{j,j+1})}_{j,j+1}\ox^{a^1_{j,j+1}}_{j,j+1})\cdots (x^{(a^0_{jn})}_{j,n}\ox^{a^1_{jn}}_{j,n})
$$
and let
\begin{equation}\label{eA+}
e_{A^+}=\pi_{n-1}(A^+)\cdots\pi_1(A^+).
\end{equation}
Similarly for the $j$th column (reading upwards) $a_{nj},\ldots,a_{j+1,j}$ for $A^-$, put
$$
\pi_j(A^-)=(x^{(a^0_{nj})}_{n,j}\ox^{a^1_{nj}}_{n,j})\cdots (x^{(a^0_{j+1,j})}_{j+1,j}\ox^{a^1_{j+1,j}}_{j+1,j})
$$
and let
\begin{equation}\label{fA-}
f_{A^-}=\pi_{1}(A^-)\cdots\pi_{n-1}(A^-).
\end{equation}
Then we have the following.
\begin{prop}\label{intPBW}
\begin{enumerate}
\item As abelian groups, we have
$
U_\Z\cong U_\Z^-\otimes U_\Z^0\otimes U_\Z^+.
$

\item
The superalgebra $U_\Z$ is a free $\Z$-module with basis given by the set
%$$
%\prod_{\al\in\Phi^+}x^{(A_\al)}_{\al}\prod_{\al\in\Phi^+}x^{A_{\oal}}_{\oal}\prod^n_{i=1}\begin{pmatrix}h_i\\B_i\end{pmatrix}
%\prod^n_{i=1}H^{D_i}_{\bar i}
%\prod_{\al\in\Phi^-}x^{(C_\al)}_{\al}\prod_{\al\in\Phi^-}x^{C_{\oal}}_{\oal}
%$$
%for all $A_\al, C_\al, B_i\in\N$ and $A_{\oal}, C_{\oal}, D_i\in\Z_2$, where the product is taken in any fixed order.
\begin{equation}\label{PBW-UZ}
\big\{\fkm_A:=f_{A^-}{\bsh\choose A_0^0}\ov{h}_{A_1^0} e_{A^+}\mid A\in M_n(\N|\Z_2)\big\}.
\end{equation}
\end{enumerate}
\end{prop}
\begin{proof} Applying the commutation formulas in Proposition~\ref{div-root} yields the inclusion
$U_\Z\subseteq U_\Z^- U_\Z^0 U_\Z^+.$ Thus, the required isomorphism in part (1) follows from
the restriction to $U_\Z$ of the canonical isomorphism $U(\mfq)\rightarrow U^-\otimes U^0\otimes U^+$ in Proposition~\ref{PBW-qn}(2).

Clearly by Propositions~\ref{PBW-qn} and \ref{div-root} (and noting \eqref{odd-square}), the subalgebra
$U^+_\Z$ (resp. $U^-_\Z$) is spanned by the set $\{e_{A^+}\mid A\in M_n(\N|\Z_2), A^-=A^0=0\}$
(resp. $\{f_{A^-}\mid A\in M_n(\N|\Z_2), A^+=A^0=0\}$).
Furthermore, by (QS1) and \eqref{odd-square} (together with a result for $U_\Z({\mathfrak q}(n)_0)$, we have that $U^0_\Z$ is spanned
by the set $\{{\bsh\choose A_0^0}\ov{h}_{A_1^0}\mid A\in M_n(\N|\Z_2), A^+=A^-=0\}$.
Putting together we obtain that $U_{\Z}$ is spanned by the set~\eqref{PBW-UZ}.
Meanwhile by Proposition~\ref{PBW-qn}(1) it is easy to check that the set~\eqref{PBW-UZ}
is linearly independent.
Hence the proposition is proved.
\end{proof}

We define the {\em degree} of the generators $x^{(s)}_{i,j},\ox_{i,j},h_{\bar i}, {h_i\choose s}$ for $U_\Z$ via
\begin{align}\label{eq:degree}
{\rm deg}(x^{(s)}_{i,j})=s|j-i|,\quad {\rm deg}(\ox_{i,j})=|j-i|,\quad{\rm deg}(h_{\bi})=1,\quad
\deg{h_i\choose s}=0,
\end{align}
for $1\leq i\neq j\leq n, s\in\N$.
Then the degree of the element $\fkm_A$ for $A=(A_0,A_1)\in M_n(\N|\Z_2)$ is given by
\begin{equation}\label{degree-mA}
{\rm deg}(\fkm_A)={\rm deg}(A):=(a^1_{11}+\cdots+a^1_{nn})+\sum_{1\leq i\neq j\leq n}(a_{ij}+a_{ji})|j-i|,
\end{equation}
where $A_0=(a^0_{ij}), A_1=(a^1_{ij})$ and $a_{ij}=a^0_{ij}+a^1_{ij}$.
By Proposition~\ref{intPBW}(2), each non-zero element $\fkl$ in $U_\Z$ can be written as a linear combination
of $\fkm_A$ with $A\in M_n(\N|\Z_2)$ and we define ${\rm deg}(\fkl)$ to be the highest degree of
the terms $\fkm_A$ appearing in $\fkl$.
In particular, if $\frak M\subseteq U_\Z$ denotes the set of monomials $\fkm$ in
$x^{(s)}_{i,j},\ox_{i,j},h_{\bar i},  {h_i\choose s}(1\leq i\neq j\leq n, s\in\N)$,
then the degree of a non-zero monomial $\fkm\in\mf M$ is well-defined.
Clearly by Proposition~\ref{div-root}, $U_\Z$ is a filtered algebra with respect to the degree defined in \eqref{eq:degree}
and ${\rm deg}(\fkm_A\fkm_B)\leq {\rm deg}(\fkm_A)+{\rm deg}(\fkm_B)={\rm deg}(A)+{\rm deg}(B)$ for $A,B\in M_n(\N|\Z_2)$.
In particular, given a non-zero monomial $\fkm\in\mf M$ in which $x^{(s)}_{i,j},\ox_{i,j},h_{\bar i}$ appear $a^s_{i,j}, b_{i,j}$ and $c_i$
times, respectively, for $1\leq i\neq j\leq n, s\in\N$, then
\begin{equation}\label{degree-m-ineq}
{\rm deg}(\fkm)\leq \sum_{s}\sum_{1\leq i\neq j\leq n}s|i-j|~a^s_{i,j}~ +\sum_{1\leq i\neq j\leq n}|i-j|~b_{i,j} +\sum^n_{i=1} c_i.
\end{equation}

\begin{rem}\label{degree} Let
$$\mc G=\{x^{(s)}_{i,j},\ox_{i,j},h_{\bar i}\mid s\in\N, 1\leq i\neq j\leq n\}.$$
By \eqref{eq:degree} and \eqref{degree-m-ineq}, every commutator $ab-(-1)^{\hat{a}\cdot\hat{b}}ba$ in Proposition~\ref{div-root}(1)-(4),
where $a,b\in \mc G$
belong to different triangular parts,  has degree strictly less than ${\rm deg}(ab)$.
The fact will be useful below.
\end{rem}
\section{Presenting the queer Schur superalgebra $\Qnr$}
Denote by $\La(n,r)$ the set of compositions of $r$ into $n$ parts, or equivalently we can view $\La(n,r)$
as a subset of ${\mc P}_{\geq0}$ in the following way:
$$
\La(n,r)=\{\la=\la_1\ep_1+\cdots+\la_n\ep_n\in{\mc P}_{\geq0}\mid \la_1+\cdots+\la_n=r\}.
$$
For $\la\in\La(n,r)$, denote by $\ell(\la)$ the number of nonzero parts in $\la$, that is,
$\ell(\la)=\sharp\{i\mid \la_i\neq 0, 1\leq i\leq n\}$.
Recall that $\{v_1,\ldots,v_n\}$ and $\{v_{-1},\ldots,v_{-n}\}$ are bases for $\ev V$ and $\od V$.
For a $r$-tuple $\udj=(j_1,\ldots,j_r)\in I(n|n)^r$, we set $v_{\udj}=v_{j_1}\otimes\cdots\otimes v_{j_r}$ and define
${\rm wt}(\udj)=(\mu_1,\ldots,\mu_n)$ via
\begin{equation}\label{wt}
\mu_i=\sharp\{k\mid~ j_k=\pm i, 1\leq k \leq r\},\text{ for all } 1\leq i\leq n.
\end{equation}
Then ${\rm wt}(\udj)\in\La(n,r)$ and the set $\{v_{\udj}\mid \udj\in I(n|n)^r\}$ is a basis for $V^{\otimes r}$.

\begin{lem} \label{phir-cartan}
Suppose $\udj\in I(n|n)^r$ and ${\rm wt}(\udj)=\mu$.
We have, for $1\leq i\leq n$,
\begin{enumerate}
\item $\phi_r(h_i)(v_{\udj})=\mu_i v_{\udj}$;
\item $\phi_r(h_{\bi})(v_{\udj})=0$ if $\mu_i=0$.
\end{enumerate}
\end{lem}
\begin{proof}
Firstly, by the definition of the action of $\mfq$ on $V$, \eqref{cartan} and \eqref{Eij}, we obtain
\begin{align*}
\phi_1(h_i)(v_j)=(E_{i,i}+E_{-i,-i})(v_j)=\delta_{i,|j|}v_j,\quad
\phi_1(h_{\bi})(v_j)=(E_{i,-i}+E_{-i,i})(v_j)=\delta_{i,|j|}v_{-j}.
\end{align*}
for $1\leq i\leq n$ and $j\in I(n|n)$.
Hence, by the definition of the homomorphism $\phi_r$ and the comultiplication on $U(\mfq)$ and noting
$\hat{h}_i=0$ and $\hat{h}_{\bi}=1$, we have
\begin{align*}
\phi_r(h_i)(v_{\udj})
=&\big(\phi_1(h_i)\otimes 1\otimes\cdots\otimes 1+\cdots+
1\otimes\cdots\otimes 1\otimes \phi_1(h_i)\big)(v_{j_1}\otimes\cdots\otimes v_{j_r})\\
=&\sum^r_{k=1}\delta_{i,|j_k|}v_{\udj},\\
\phi_r(h_{\bi})(v_{\udj})=&\big(\phi_1(h_{\bi})\otimes 1\otimes\cdots\otimes 1+\cdots+
1\otimes\cdots\otimes 1\otimes \phi_1(h_{\bi})\big)(v_{j_1}\otimes\cdots\otimes v_{j_r})\\
=&\sum^r_{k=1}(-1)^{\hat{v}_{j_1}+\cdots+\hat{v}_{j_{k-1}}}
\delta_{i,|j_k|}v_{j_1}\otimes\cdots\otimes v_{j_{k-1}}\otimes v_{-j_k}\otimes v_{j_{k+1}}\otimes\cdots\otimes v_{j_r},
\end{align*}
for $1\leq i\leq n$.
Then by \eqref{wt}, the lemma follows.
\end{proof}

Let $I$ be the ideal of the universal enveloping algebra $U=U(\mfq)$ given by
\begin{align}\label{ideal-I}
I=\langle h_1+\cdots+h_n-r,~
h_{\bar 1}(h_1-1)\cdots(h_1-r),\ldots, h_{\bar{n}}(h_n-1)\cdots(h_n-r)\rangle.
\end{align}
Then by \eqref{odd-square} we obtain for $1\leq i\leq n$
\begin{equation}\label{old-Hiproduct}
h_i(h_i-1)\cdots(h_i-r)=h^2_{\bar i}(h_i-1)\cdots(h_i-r)\in I.
\end{equation}
Define
\begin{equation}\label{Unr}
\aligned
\Pnr=U/I, \quad U_\Z(n,r)=U_\Z/I\cap U_\Z\\
U^0(n,r)=U^0/I\cap U^0,\quad U^0_\Z(n,r)=U^0_\Z/I\cap U^0_\Z.
\endaligned
\end{equation}
%Moreover, let $\mc{Q}_{\Z}(n,r)$ (resp. $\mc{Q}^0_{\Z}(n,r)$) be the image of $U_\Z$ (resp. $U^0_\Z$) under the map $\phi_r$.
{\it For notational simplicity, we will denote, by abuse of notation,  the images of $e_i,f_i,h_i$, $x_\al,e_{A^+}$ etc. in $U(n,r)$ by the same letters.}

%As both $\Pnr$ and $\Qnr$ are homomorphic image of $U=U(\mfq)$, any relation between generators
%holding in $U$ will automatically be satisfied in $\Pnr$ and $\Qnr$.
%For simplicity, we will not distinguish notationally between the generators nor root vectors
%for $U,\Pnr,\Qnr$.

\begin{prop}\label{surjective}
The homomorphism $\phi_r:U\rightarrow \End_{\Q}(V^{\otimes r})$
satisfies $I\subseteq \ker\phi_r$. Hence there exists a natural surjective
homomorphism
\begin{equation}\label{phibar}
\ov{\phi}_r: \Pnr\twoheadrightarrow \Qnr.
\end{equation}
\end{prop}
\begin{proof}
Fix an arbitrary $\udj\in I(n|n)^r$ and write ${\rm wt}(\udj)=\mu=(\mu_1,\ldots,\mu_n)\in\La(n,r)$.
By Lemma~\ref{phir-cartan}(1), we have
$
\phi_r(h_1+\cdots+h_n)(v_{\udj})=(\mu_1+\cdots+\mu_n)v_{\udj}=rv_{\udj}.
$
This implies
\begin{align}\label{ker-phir-Hi}
\phi_r(h_1+\cdots+h_n-r)=0
\end{align}
since the set $\{v_{\udj}\mid \udj\in I(n|n)^r\}$
is a basis for $V^{\otimes r}$.
On the other hand, again by Lemma~\ref{phir-cartan}(1), we have
\begin{align}
\phi_r(h_{\bar{i}}(h_i-1)\cdots (h_i-r))(v_{\udj})=&(\mu_i-1)\cdots(\mu_i-r)\big(\phi_r(h_{\bar i})\big)(v_{\udj}).\label{phir-HioddHi}
\end{align}
Observe that $0\leq \mu_i\leq r$.
If $\mu_i=0$, then we have $\phi_r(h_{\bar i})(v_{\udj})=0$ by Lemma~\ref{phir-cartan}(2); otherwise,
we have $1\leq \mu_i\leq r$, which implies $(\mu_i-1)\cdots(\mu_i-r)=0$.
Therefore, by \eqref{phir-HioddHi}, we deduce that $\phi_r(h_{\bar{i}}(h_i-1)\cdots (h_i-r))(v_{\udj})=0$.
This means
\begin{align}\label{ker-phir-Hiodd}
\phi_r(h_{\bar{i}}(h_i-1)\cdots (h_i-r))=0
\end{align}
for all $1\leq i\leq n$.
Therefore by~\eqref{ideal-I}, \eqref{ker-phir-Hi} and \eqref{ker-phir-Hiodd},
the ideal $I$ is contained in $\ker\phi_r$ and hence the proposition is verified.
\end{proof}

For $\la\in\La(n,r)$, write
\begin{align*}
1_\la={\bsh\choose\la}\in U_\Z^0(n,r).
\end{align*}

\begin{prop} \label{idempotent}
The following hold in $U_\Z^0(n,r)$:
\begin{enumerate}
\item The set $\{1_\la\mid \la\in\La(n,r)\}$ is a set of pairwise orthogonal
central idempotents in $U_\Z^0(n,r)$ and
$\sum_{\la\in\La(n,r)}1_{\la}=1.
$

\item ${\bsh\choose \bsb}=0$ for any $\bsb\in\N^n$ with $|\bsb|>r$.

\item For $1\leq i\leq n, \la\in\La(n,r)$ and $\bsb\in\N^n$, we have $h_i1_\la=\la_i1_\la$,
\begin{align*}
%\begin{pmatrix}h_i\\ b\end{pmatrix}1_\la=\begin{pmatrix}\la_i\\ b\end{pmatrix}1_\la,~
{\bsh\choose \bsb}1_\la={\la\choose \bsb}1_\la,\text{ and }
{\bsh\choose \bsb}=\sum_{\la\in\La(n,r)}{\la\choose\bsb}1_\la,
\end{align*}
where ${\la\choose \bsb}=\prod^n_{i=1}\begin{pmatrix}\la_i\\ b_i\end{pmatrix}$.

\item Suppose $1\leq i\leq n$ and $\la\in\La(n,r)$. If $\la_i=0$, then
$
h_{\bi}1_{\la}=0.
$

\item The $\Z$-algebra $U^0_\Z(n,r)$ is spanned by the set
$$
\{1_\la \ovHD\mid \la\in\La(n,r),D\in\Z_2^n, D_i\leq\la_i \text{ for }1\leq i\leq n \}.
$$
\end{enumerate}
\end{prop}
\begin{proof}
By \eqref{ideal-I} and \eqref{old-Hiproduct}, we know the elements $h_1+\cdots+h_n-r$
and $h_1(h_1-1)\cdots(h_1-r),\ldots,h_n(h_n-1)\cdots(h_n-r)$ are contained in the ideal $I$.
Then under the identification~\eqref{qnev}, there is a natural algebra homomorphism from the algebra $T^0$ introduced
in \cite[(4.1)]{DG} for $\mathfrak{gl}(n)$ to $U^0(n,r)$, which sends the element $H_i$ in \cite{DG}
to $h_i$ for $1\leq i\leq n$.  Therefore, parts (1)-(3) follow from \cite[Proposition 4.2]{DG}.

To prove part (4), we observe that the following holds for $\la\in\La(n,r)$ and $1\leq i\leq n$
\begin{align}\label{Hi-idempotent}
(h_i-1)(h_i-2)\cdots (h_i-r)1_\la=(\la_i-1)(\la_i-2)\cdots(\la_i-r)1_\la
\end{align}
by part (3).
Now suppose $\la_i=0$. Then by \eqref{Hi-idempotent} we obtain
$$
(h_i-1)(h_i-2)\cdots (h_i-r)1_\la=(-1)^rr!1_\la.
$$
This implies
$$
(-1)^rr!h_{\bi}1_\la=h_{\bi}(h_i-1)(h_i-2)\cdots (h_i-r)1_\la=0
$$
since $h_{\bi}(h_i-1)(h_i-2)\cdots (h_i-r)\in I\cap U^0$.
Therefore $h_{\bi}1_\la=0$.

Finally, %by the relation (QS1) and part (3)
the fact that the elements ${\bsh\choose \bsb}\ovHD$ with $\bsb\in\N^n, D\in\Z_2^n$ span $U_\Z^0$ implies that $U_\Z^0(n,r)$ is spanned by the set
$\{ 1_\la \ovHD\mid \la\in\La(n,r),D\in\Z_2^n\}$.
By part (4), we have $ 1_\la \ovHD=0$ if there exists $1\leq i\leq n$ such that $D_i>\la_i$.
Therefore part (5) is proved.
\end{proof}

\begin{prop}\label{root-idem}
Suppose $1\leq i\leq n, \al\in\Phi,\la\in\La(n,r)$.
Then the following commutation formulas hold in $U_\Z(n,r)$:
\begin{align*}
x_\al1_\la&=\left\{
\begin{array}{ll}
1_{\la+\al}x_\al,&\text{ if }\la+\al\in\La(n,r),\\
0,&\text{ otherwise},
\end{array}
\right.
\quad
\ox_{\al}1_\la=\left\{
\begin{array}{ll}
1_{\la+\al}\ox_{\al},&\text{ if }\la+\al\in\La(n,r),\\
0,&\text{ otherwise},
\end{array}
\right.\\
1_\la x_\al&=\left\{
\begin{array}{ll}
x_\al1_{\la-\al},&\text{ if }\la-\al\in\La(n,r),\\
0,&\text{ otherwise},
\end{array}
\right.
\quad
1_\la \ox_{\al}=\left\{
\begin{array}{ll}
\ox_{\al}1_{\la-\al},&\text{ if }\la-\al\in\La(n,r),\\
0,&\text{ otherwise},
\end{array}
\right.\\
h_{\bar i}1_\la&=1_\la h_{\bar i}.
\end{align*}
\end{prop}
\begin{proof}
The last equality follows from the relation (QS1).
The proof of the remaining equalities
is parallel to that of \cite[Proposition 4.5]{DG} (see Lemma~\ref{root-comm}). Let us illustrate by checking in detail the second formula.
Suppose $\al=\al_{i,j}=\ep_i-\ep_j$ with $i\neq j$. Then by Propositions \ref{idempotent}(3) and ~\ref{div-root}(5) we obtain
\begin{equation}\label{xla-1la}
\aligned
\ox_\al1_\la=& \ox_\al \frac{h_i+1}{\la_i+1}1_\la
=\ox_\al \frac{h_i+1}{\la_i+1}\begin{pmatrix} h_1\\ \la_1\end{pmatrix}\begin{pmatrix} h_2\\ \la_2\end{pmatrix}\cdots\begin{pmatrix} h_n\\\la_n\end{pmatrix}\\
=&\frac{h_i}{\la_i+1}\Bigg(\begin{pmatrix} h_i-1\\ \la_i\end{pmatrix}\begin{pmatrix} h_j+1\\ \la_j\end{pmatrix}\prod_{l\neq i,j}\begin{pmatrix} h_l\\ \la_l\end{pmatrix}\Bigg)\ox_\al \\
=&\Bigg(\begin{pmatrix} h_i\\ \la_i+1\end{pmatrix}\begin{pmatrix} h_j+1\\ \la_j\end{pmatrix}\prod_{l\neq i,j}\begin{pmatrix} h_l\\ \la_l\end{pmatrix}\Bigg)\ox_\al.
\endaligned
\end{equation}
If $\la+\al_{i,j}\not\in \La(n,r)$, then $\la_j=0$ and \eqref{xla-1la} becomes
$$
\ox_\al1_\la= \Bigg(\begin{pmatrix} h_i\\ \la_i+1\end{pmatrix}\prod_{l\neq i}\begin{pmatrix} h_l\\ \la_l\end{pmatrix}\Bigg)\ox_\al= {\bsh\choose{\la+\ep_i}}\ox_\al=0,
$$
%Since $|\la+\ep_i|=r+1>r$, we obtain $\ox_\al1_\la=$ due to the fact ${\bsh\choose{\la+\ep_i}}=0$
by Proposition~\ref{idempotent}(2). If $\la+\al_{i,j}\in \La(n,r)$, then $\la_j\geq 1$ and $\begin{pmatrix} h_j+1\\ \la_j\end{pmatrix}=\begin{pmatrix} h_j\\ \la_j\end{pmatrix}+\begin{pmatrix} h_j\\ \la_j-1\end{pmatrix}$.
Hence, \eqref{xla-1la} becomes
\begin{align*}
\ox_\al1_\la= &\Bigg(\begin{pmatrix} h_i\\ \la_i+1\end{pmatrix}\begin{pmatrix} h_j\\ \la_j\end{pmatrix}\prod_{l\neq i,j}\begin{pmatrix} h_l\\ \la_l\end{pmatrix}\Bigg)\ox_\al+\Bigg(\begin{pmatrix} h_i\\ \la_i+1\end{pmatrix}\begin{pmatrix} h_j\\ \la_j-1\end{pmatrix}\prod_{l\neq i,j}\begin{pmatrix} h_l\\ \la_l\end{pmatrix}\Bigg)\ox_\al\\
=&{\bsh\choose{\la+\ep_i}}\ox_\al+{\bsh\choose{\la+\al}}\ox_\al=1_{\la+\al}\ox_\al,
\end{align*}
as desired.
\end{proof}

For $A=(A_0,A_1)\in M_n(\N|\Z_2)$ with $A_0=(a^0_{ij}), A_1=(a^1_{ij})$ and $a_{ij}=a^0_{ij}+a^1_{ij}$, define
\begin{equation}\label{degree-content}
\aligned
\ro(A)&:=(\sum^n_j a_{1j},\ldots,\sum^n_ja_{nj}),\\
\co(A)&:=(\sum^n_ja_{j1},\ldots,\sum^n_ja_{jn}),\\
\chi(A)&:=(a_{11}+\sum^n_{j=2}(a_{1j}+a_{j1}),a_{22}+\sum^n_{j=3}(a_{2j}+a_{j2}),\ldots, a_{nn})\\
&=\co(A^0_0)+\co(A^0_1)+\ro(A^+)+\co(A^-).
\endaligned
\end{equation}

Let
$$
M_n(\N|\Z_2)'=\{C\in M_n(\N|\Z_2)\mid C^0_0=0\}.
$$
Similar to \eqref{PBW-UZ}, we define
$$
u_C=f_{C^-}\ov{h}_{C^0_1}e_{C^+}\in U_\Z\quad\text{ and }\quad \fku_{(C,\la)}=f_{C^-}1_\la\ov{h}_{C^0_1}e_{C^+}\in U_\Z(n,r)
$$ for $C\in M_n(\N|\Z_2)'$ and $\la\in\La(n,r)$. Clearly, by definition, the degree function defined in \eqref{degree-mA} satisfies $
{\rm deg}(u_C)={\rm deg}(C).
$

By Proposition~\ref{root-idem}, we have in $U_\Z(n,r)$
\begin{equation}\label{uclan}
\fku_{(C,\la)}=1_{\la'}u_C=u_C1_{\la''} \text{ if } \fku_{(C,\la)}\neq 0,
\end{equation}
where $\la'=\la+\ro(C^-)-\co(C^-)$ and $\la''=\la-\ro(C^+)+\co(C^+)$.
Recall that $\mf M$ consists of the monomials in $x^{(s)}_{i,j},\ox_{i,j},h_{\bar i},  {h_i\choose s}(1\leq i\neq j\leq n, s\in\N)$.

\begin{lem}\label{express-m} $(1)$
Let $0\neq \fkm\in\mf M$. Then, in $U_\Z(n,r)$,  $\fkm$ can be written as a linear combination of $\fku_{(C,\la)}$
for $C\in M_n(\N|\Z_2)', \la\in\La(n,r)$ such that ${\rm deg}(C)\leq {\rm deg}(\fkm)$.

$(2)$ Suppose $C\in M_n(\N|\Z_2)'$. Then  we have in $U_\Z(n,r)$
$$
f_{C^-}\ov{h}_{C^0_1}e_{C^+}=\pm e_{C^+}\ov{h}_{C^0_1}f_{C^-}+\sum \zeta^{C}_{(G,\la)} \fku_{(G,\la)},
$$
for some $\zeta^{C}_{(G,\la)}\in\Z$, where the summation is over $G\in M_n(\N|\Z_2)'$ and $\la\in\La(n,r)$
such that ${\rm deg}(G)<{\rm deg}(C)$.
\end{lem}
\begin{proof}
By Proposition \ref{intPBW}, we see that in $U_\Z$ the monomial $\fkm$ can be written as
\begin{equation}\label{express-m-UZ}
\fkm=\sum_{A\in M_{n}(\N|\Z_2)}\xi^{\fkm}_A ~\fkm_A=\sum_{A\in M_{n}(\N|\Z_2)}\xi^{\fkm}_A ~f_{A^-}{\bsh\choose A_0^0}\ov{h}_{A_1^0} e_{A^+}
\end{equation}
where $\xi^{\fkm}_{A}\in \Z$ and
 $\xi^{\fkm}_{A}=0$ unless ${\rm deg}(\fkm_A)\leq {\rm deg }(\fkm)$.
Since
${\bsh\choose A_0^0}=\sum_{\la\in\La(n,r)} {\la\choose A_0^0}1_{\la}$ in $U_\Z(n,r)$ by Proposition~\ref{idempotent},
we have
$$
\fkm_A=\sum_{\la\in\La(n,r)} {\la\choose A_0^0}f_{A^-}1_{\la}\ov{h}_{A_1^0} e_{A^+}
=\sum_{\la\in\La(n,r)} {\la\choose A_0^0} \fku_{(A',\la)},
$$
where $A'\in M_n(\N|\Z_2)'$ with $(A')^+=A^+,(A')^-=A^-, (A')^0_1=A^0_1$ for each $A\in M_{n}(\N|\Z_2)$.
Clearly ${\rm deg}(A')={\rm deg}(A)={\rm deg}(\fkm_A)$.
This together with \eqref{express-m-UZ} proves (1).

By applying a sequence of commutation formulas for generators from different triangular parts,
we have in $U_\Z$
$$
f_{C^-}\ov{h}_{C^0_1}e_{C^+}=\pm e_{C^+}\ov{h}_{C^0_1}f_{C^-}+g,
$$
where $g$ is a linear combination of monomials $\fkm$ in the generators for $U_\Z$ and, by Remark~\ref{degree}, ${\rm deg}(\fkm)<{\rm deg}(e_{C^+}\ov{h}_{C^0_1}f_{C^-})$. Now part (2) follows from part (1).
\end{proof}

Given two elements $\la=\sum^n_i\la_i\ep_i,\mu=\sum^n_i\mu_i\ep_i\in{\mc P}_+$, we define
\begin{equation}\label{orderprec}
\la\preceq\mu \Leftrightarrow \la_i\leq \mu_i, \text{ for }1\leq i\leq n.
\end{equation}
Set
\begin{align*}
\sB=\{\fku_{(C,\la)}\mid
C\in M_n(\N|\Z_2)',
\la\in\La(n,r),\chi(C)\preceq \la\}.
\end{align*}
Let $M_n(\N|\Z_2)_r:=\{A=(A_0,A_1)=((a^0_{ij}), (a^1_{ij}))\in M_n(\N|\Z_2)\mid \sum_{i,j}(a^0_{ij}+a^1_{ij})=r\}$.
Then, one can also check the following holds
\begin{align*}
\sB=\{\fku_A:=f_{A^-}1_{\chi(A)} \bar h_{A_1^0} e_{A^+}\mid A\in M_n(\N|\Z_2)_r\}.
\end{align*}

\begin{prop}\label{Y-span}
The algebra $U_\Z(n,r)$ is spanned by the set $\sB$.
\end{prop}
\begin{proof}
Let $U'_\Z(n,r)$ be the $\Z$-submodule of $U_\Z(n,r)$ spanned by $\sB$.
Clearly by Proposition~\ref{intPBW} and \eqref{Unr},
the algebra $U_\Z(n,r)$ is spanned by monomials $\fkm\in\mf M$.
Then by Lemma~\ref{express-m} and Proposition~\ref{root-idem}, it suffices to show that
$\fku_{(C,\la)}\in U'_\Z(n,r)$ for all $C\in M_n(\N|\Z_2)'$ and $\la\in\La(n,r)$.
Now fix a $C\in M_n(\N|\Z_2)'$ and $\la\in\La(n,r)$.
Write $C_0=(c^0_{ij}), C_1=(c^1_{ij})$ and let $c_{ij}=c^0_{ij}+c^1_{ij}$.
If $\chi(C)\preceq\la$, then $\fku_{(C,\la)}\in\sB$ and we are done.
Now assume $\chi(C)=(\chi_1(C),\ldots,\chi_n(C))\not\preceq\la$.
Then by \eqref{orderprec}, there exists $1\leq i\leq n$ such that $\la_i<\chi_i(C)$.
We proceed on ${\rm deg}(C)$. If ${\rm deg}(C)=0$, then the result holds as $\fku_{(C,\la)}=1_{\la}\in\sB$.
Now assume that ${\rm deg}(C)\geq 1$ and let $i$ be the biggest $i$ such that $\la_i<\chi_i(C)$.
Let $G$ be the submatrix $((c_{k,l}^0)|(c_{k,l}^1))_{i\leq k,l\leq n}$ of $C$ at the bottom right corner. Then by \eqref{eA+} and \eqref{fA-} we can
write $e_{C^+}=e_{G^+}\cdot\fkm_1$,
$f_{C^-}=\fkm'_1\cdot f_{G^-}$ and $\ov{h}_{C^0_1}=\ov{h}'\cdot\ov{h}_{G^0_1}$, where $\ov{h}'=\prod^{i-1}_{j=1}h^{c^1_{jj}}_{\ov{j}}$.
Observe that $f_{G^-}1_{\la}\ov{h}'=\pm\ov{h}'f_{G^-}1_{\la}$ by Proposition~\ref{div-root}(4) and Proposition~\ref{root-idem}.
Then we have
$$
\fku_{(C,\la)}=f_{C^-}1_{\la}\ov{h}_{C^0_1} e_{C^+}=(\fkm'_1f_{G^-})\cdot(1_{\la}\ov{h}'\cdot\ov{h}_{G^0_1})\cdot (e_{G^+}\fkm_1)=\pm\fkm'_1\ov{h}'\cdot(f_{G^-}1_{\la}\ov{h}_{G^0_1} e_{G^+})\cdot\fkm_1.
$$
We can assume $f_{G^-}1_{\la}\neq0$. Otherwise, we are done.
Then by \eqref{uclan} we have $f_{G^-}1_{\la}=1_{\la'}f_{G^-}$ and hence
$$
\fku_{(C,\la)}
=\pm\fkm'_1\ov{h}'1_{\la'}\big(f_{G^-}\ov{h}_{G^0_1}e_{G^+}\big)\fkm_1,
$$
where $\la'=\la+\ro(G^-)-\co(G^-)$.
Then applying Lemma~\ref{express-m}(2) to $f_{G^-}\ov{h}_{G^0_1}e_{G^+}$, we have
\begin{align*}
\fku_{(C,\la)}
=&\pm\fkm'_1\ov{h}'1_{\la'}\big(\pm e_{G^+} \ov{h}_{G^0_1}f_{G^-}+\sum\zeta^{G}_{(J,\ga)}\fku_{(J,\ga)})\fkm_1\\
=&\pm\fkm'_1\ov{h}'1_{\la'}e_{G^+} \ov{h}_{G^0_1}f_{G^-}\fkm_1\pm\sum\zeta^{G}_{(J,\ga)}\fkm'_11_{\la'}\ov{h}'\fku_{(J,\ga)}\fkm_1,
\end{align*}
where $\zeta^{G}_{(J,\ga)}=0$ unless ${\rm deg}(J)<{\rm deg}(G)$.
Observe that ${\rm deg}(G)+{\rm deg}(\fkm'_1\ov{h}')+{\rm deg}(\fkm_1)={\rm deg}(C)$.
Note that by \eqref{uclan}, we can apply Lemma~\ref{express-m} to each term $\fkm'_11_{\la'}\ov{h}'\fku_{(J,\ga)}\fkm_1$ to obtain the following:
\begin{align*}
\fku_{(C,\la)}=\pm\fkm'_1\ov{h}'1_{\la'}e_{G^+} \ov{h}_{G^0_1}f_{G^-}\fkm_1+\fkl,
\end{align*}
where $\fkl$ is a linear combination of $\fku_{(C',\mu)}$ with ${\rm deg}(C')<{\rm deg}(C)$.
By induction, $\fkl\in U'_\Z(n,r)$. Meanwhile we claim that $\fkm'_1\ov{h}'1_{\la'}e_{G^+} \ov{h}_{G^0_1}f_{G^-}\fkm_1=0$.
Indeed, we can assume $1_{\la'}e_{G^+}\neq 0$\footnote{In the non super case, we have $1_{\la'}e_{G^+}=0$ as $\la_i'<\chi_i(G^+)$.
However, this inequality may not be true in the super case as the $i$-th entry of $C_1^0$ may not be zero.}. Otherwise we are done.
Then
$
1_{\la'}e_{G^+}=e_{G^+}1_{\la''}
$ by \eqref{uclan},
where $\la''=\la'-\ro(G^+)+\co(G^+)=\la+\ro(G^-)-\co(G^-)-\ro(G^+)+\co(G^+)$. Moreover $\la''\in\La(n,r)$.
Since $G$ is the submatrix of $C$ consisting of the last $n-i+1$ rows and columns, we have
$$\la''_i=\la_i-(\sum^n_{j=i+1}c_{ji})-(\sum^n_{j=i+1}c_{ij})=\la_i-(\chi_i(C)-c^1_{ii})=\la_i-\chi_i(C)+c^1_{ii}$$
by \eqref{degree-content}.
This means $\la''_i<c^1_{ii}$ since $\chi_i(C)>\la_i$, forcing $c_{i,i}^1=1$ and $\la_i''=0$ since $\la''\in\La(n,r)$. Hence, $1_{\la''}\ov{h}_{G^0_1}=0$ by Proposition~\ref{idempotent}(4), proving the claim.
In conclusion, $\fku_{(C,\la)}\in U_\Z'(n,r)$.
\end{proof}

We remark that the proof above follows \cite{DP}. However, it is possible to modify the proof of \cite{DG} to give an alternative proof.

\begin{theorem}\label{presentQr}
The homomorphism $\ov{\phi}_r : U(n,r)\rightarrow \Qnr$ in \eqref{phibar} is an isomorphism.
In other words, the Schur superalgebra $\Qnr$ is the associative superalgebra
generated by even generators
$ h_i, e_j, f_j,$
and odd generators $h_{\bar i}, e_{\bar j}, f_{\bar j},$
with $1\leq i\leq n$ and $1\leq j\leq n-1$
subject to the relations {\rm (QS1)-(QS6)} together with the following extra relations:
\begin{itemize}

  \item[(QS7)] $h_1+h_2+\cdots+h_n=r$;

 %\item[(QS8)]$h_{i}(h_i-1)\cdots(h_i-r)=0$ for $1\leq i\leq n$.

  \item[(QS8)]$h_{\bi}(h_i-1)\cdots(h_i-r)=0$ for $1\leq i\leq n$.
\end{itemize}
\end{theorem}
\begin{proof}
By~\cite[\S 4]{BK}, we know that the dimension of the algebra $\Qnr$
is equal to the number of monomials of total degree $r$ in the free supercommutative algebra in $n^2$ even
variables and $n^2$ odd variables.
Thus,
\begin{equation}\label{dimQnrX}
\dim\Qnr=|M_n(\N|\Z_2)_r|.
\end{equation} Hence, by~\eqref{dimQnrX},
Proposition~\ref{surjective} and Proposition~\ref{Y-span} we have
$$
\dim U(n,r)\leq |\sB|=\dim \Qnr\leq \dim U(n,r),
$$
which implies $\dim U(n,r)=|\sB|=\dim \Qnr$. This forces that the surjective homomorphism $\ov{\phi}_r$ is an isomorphism.
\end{proof}

%By Theorem~\ref{presentQr}, we have the following.
\begin{cor}
The set
$$
\sB=\big\{\fku_A\mid
A\in M_n(\N|\Z_2)_r\big\}
$$
is a $\Z$-basis for $U_\Z(n,r)$. In particular, the set $\{1_\la\ovHD\mid \la\in\La(n,r),D\in\Z_2^n, D_i\leq \la_i,1\leq i\leq n\}$
is  $\Z$-basis for $U^0_\Z(n,r)$.
\end{cor}

\begin{cor}\label{dim formulas} We have the following dimension formulas:
\begin{enumerate}
\item $
\dim_\Q \Qnr=\sum^r_{k=0}\begin{pmatrix}n^2+k-1\\ k\end{pmatrix}\begin{pmatrix}n^2\\ r-k\end{pmatrix};
$
\item $\dim_\Q \mc{Q}^0(n,r)=\sum_{\la\in\La(n,r)}2^{\ell(\la)}.
$
\end{enumerate}
\end{cor}

Using a similar argument in the proof of \cite[Theorem 2.4]{DG}, we obtain the following.
\begin{thm}\label{presentQr-idemp}
The queer Schur superalgebra $\Qnr$ is the unitary associative superalgebra generated by
the even elements $1_\la, e_j, f_j $ and odd elements
$h_{\bi}, e_{\bj}, f_{\bj}$ for $\la\in\La(n,r),1\leq i\leq n, 1\leq j\leq n-1$
subject to {\rm (QS3)} and {\rm (QS5)-(QS6)} as well as the following relations:
\begin{enumerate}
\item[(QS1${}^\prime$)]
 $1_\la1_\mu=\delta_{\la,\mu}1_\la,~ \sum_{\la\in\La(n,r)}1_\la=1$,~
 $h_{\bi}1_\la=1_\la h_{\bi},$~
 $h_{\bar i}h_{\bar j}+h_{\bar j}h_{\bar i}=\delta_{ij}\sum_{\la\in\La(n,r)}2\la_i1_\la$,

\noindent $h_{\bi}1_\la=0$  if $\la_i=0$;

\vspace{0.1in}

 \item[(QS2${}^\prime$)]
$
e_j1_\la=\left\{
\begin{array}{ll}
1_{\la+\al_j}e_j,&\text{ if }\la+\al_j\in\La(n,r),\\
0,&\text{ otherwise},
\end{array}
\right.
$
$
e_{\bj}1_\la=\left\{
\begin{array}{ll}
1_{\la+\al_j}e_{\bj},&\text{ if }\la+\al_j\in\La(n,r),\\
0,&\text{ otherwise},
\end{array}
\right.
$
$
f_j1_\la=\left\{
\begin{array}{ll}
1_{\la-\al_j}f_j,&\text{ if }\la-\al_j\in\La(n,r),\\
0,&\text{ otherwise},
\end{array}
\right.
$
$
f_{\bj}1_\la=\left\{
\begin{array}{ll}
1_{\la-\al_j}f_{\bj},&\text{ if }\la-\al_j\in\La(n,r),\\
0,&\text{ otherwise},
\end{array}
\right.
$
$
1_\la e_j=\left\{
\begin{array}{ll}
e_j1_{\la-\al_j},&\text{ if }\la-\al_j\in\La(n,r),\\
0,&\text{ otherwise},
\end{array}
\right.
$
$
1_\la e_{\bj}=\left\{
\begin{array}{ll}
e_{\bj}1_{\la-\al_j},&\text{ if }\la-\al_j\in\La(n,r),\\
0,&\text{ otherwise},
\end{array}
\right.
$
$
1_\la f_j=\left\{
\begin{array}{ll}
f_j1_{\la+\al_j},&\text{ if }\la+\al_j\in\La(n,r),\\
0,&\text{ otherwise},
\end{array}
\right.
$
$
1_\la f_{\bj}=\left\{
\begin{array}{ll}
f_{\bj}1_{\la+\al_j},&\text{ if }\la+\al_j\in\La(n,r),\\
0,&\text{ otherwise};
\end{array}
\right.
$

\vspace{0.1in}

 \item[(QS4${}^\prime$)]
 $e_if_j-f_je_i=\delta_{ij}\sum_{\la\in\La(n,r)}(\la_i-\la_{i+1})1_\la$,\quad
 $e_{\bi}f_{\bj}+f_{\bj}e_{\bi}=\delta_{ij}\sum_{\la\in\La(n,r)}(\la_i+\la_{i+1})1_\la$;

\noindent $e_if_{\bj}-f_{\bj}e_i=\delta_{ij}(h_{\bi}-h_{\ov{i+1}})$,\quad
 $e_{\bi}f_j-f_je_{\bi}=\delta_{ij}(h_{\bi}-h_{\ov{i+1}})$.
\end{enumerate}
\end{thm}

\section{The quantum queer superalgebra $\qUq$ and its root vectors}
The quantum queer superalgebra is more subtle than the enveloping algebra. In the next four sections, we will investigate the quantum root vectors and their commutation formulas. We then generalize Theorems \ref{presentQr} and \ref{presentQr-idemp} to quantum queer Schur algebras in the last section.

In \cite{Ol}, Olshanski introduced the quantum deformation $U_q=\qUq$ of the universal enveloping algebra
$U(\mfq)$ of $\mfq$ as follows.
Let the symbol $\{\cdots\}$, where the dots standard for some inequalities, equal 1 if all these inequalities are satisfied
and 0 otherwise.
For $i,j,k\in I(n|n)$, put $\varphi(i,j)=\delta_{|i|,|j|}{\rm sgn}(j),$
\begin{align*}
p(i,j)=\left\{
\begin{array}{ll}
0,&\text{ if }ij>0,\\
1,&\text{ if }ij<0,
\end{array}
\right. ~
\;\;\text{ and }\;\;
\theta(i,j,k)={\rm sgn}({\rm sgn}(i)+{\rm sgn}(j)+{\rm sgn}(k)),
\end{align*}
where ${\rm sgn}(a)=1$ if $a>0$ and ${\rm sgn}(a)=-1$ if $a<0$ for an arbitrary nonzero integer $a$.
\begin{defn} The quantum queer superalgebra $\qUq$ is the associative superalgebra over $\Q(q)$ generated by $L_{i,j}$ for $i, j\in I(n|n)$
with $i\leq j$
subject to the following relations
\begin{equation}\label{Ol-reln}
\aligned
&L_{i,i}L_{-i,-i}=L_{-i,-i}L_{i,i}=1,\\
&(-1)^{p(i,j)p(k,l)}q^{\varphi(j,l)}L_{i,j}L_{k,l}
+\{k\leq j<l\}\theta(i,j,k)(q-q^{-1})L_{i,l}L_{k,j}\\
&\qquad +\{i\leq-l<j\leq -k\}\theta(-i,-j,k)(q-q^{-1})L_{i,-l}L_{k,-j}\\
&=q^{\varphi(i,k)}L_{k,l}L_{i,j}+\{k<i\leq l\}\theta(i,j,k)(q-q^{-1})L_{i,l}L_{k,j}\\
&\qquad +\{-l\leq i<-k\leq j\}\theta(-i,-j,k)(q-q^{-1})L_{-i,l}L_{-k,j}.
\endaligned
\end{equation}
The $\Z_2$-grading on $\qUq$ is defined via $\hat{L}_{i,j}=p(i,j)$ for $i,j\in I(n|n)$
with $i\leq j$.
\end{defn}

Following \cite[Remark 7.3]{Ol}, we introduce the following set of generators of $\qUq$:
\begin{equation}\label{q-generator}
\aligned
K_i:=&L_{i,i},\quad K_i^{-1}:=L_{-i,-i},\quad K_{\bi}:=-\frac{1}{q-q^{-1}}L_{-i,i},\\
E_j:=&-\frac{1}{q-q^{-1}}K_{j+1}L_{-j-1,-j},\quad E_{\bj}:=-\frac{1}{q-q^{-1}}K_{j+1}L_{-j-1,j},\\
F_j:=&\frac{1}{q-q^{-1}}L_{j,j+1}K_{j+1}^{-1},\quad F_{\bj}:=-\frac{1}{q-q^{-1}}L_{-j,j+1}K_{j+1}^{-1}.
\endaligned
\end{equation}
for $1\leq i\leq n$ and $1\leq j\leq n-1$.
Then the algebra $\qUq$ can be defined via a quantum analogue of the relations (QS1)-(QS6)
using~~\eqref{Ol-reln} as follows.

\begin{prop}[{\cite[Theorem 2.1]{GJKK} cf. (\cite{GJKKK})}]\label{presentqUq}
The quantum superalgebra $\qUq$ is isomorphic to the unital associative superalgebra over $\Q(q)$ generated
by even generators $K^{\pm1}_i, E_j, F_j$ and odd generators $K_{\bi},E_{\bj},F_{\bj},$ for $1\leq i\leq n, 1\leq j\leq n-1$,
satisfying the following relations:
\begin{itemize}
 \item[(QQ1)]
 $K_iK^{-1}_i=1=K^{-1}_iK_i$,~
 $K_iK_j=K_jK_i,~ K_iK_{\bar j}=K_{\bj}K_i,$

\noindent $K_{\bar i}K_{\bar j}+K_{\bar j}K_{\bar i}=\ds2\delta_{ij}\frac{K_i^2-K_i^{-2}}{q^2-q^{-2}}$;

\vspace{0.1in}

 \item[(QQ2)]
 $K_iE_j=q^{(\ep_i,\alpha_j)}E_jK_i,\quad K_iE_{\bar j}=q^{(\ep_i,\alpha_j)}E_{\bar j}K_i,$

\noindent  $K_iF_j=q^{-(\ep_i,\alpha_j)}F_jK_i,\quad K_iF_{\bar j}=q^{-(\ep_i,\alpha_j)}F_{\bar j}K_i$;

\vspace{0.1in}

 \item[(QQ3)]
 $K_{\bi}E_i-qE_iK_{\bi}=E_{\bi}K_i^{-1},\quad qK_{\bi}E_{i-1}-E_{i-1}K_{\bi}=-K^{-1}_iE_{\ov{i-1}},$

 \noindent  $K_{\bi}F_i-qF_iK_{\bi}=-F_{\bi}K_i,\quad qK_{\bi}F_{i-1}-F_{i-1}K_{\bi}=K_iF_{\ov{i-1}},$

 \noindent  $K_{\bi}E_{\bi}+qE_{\bi}K_{\bi}=E_iK_i^{-1},\quad qK_{\bi}E_{\ov{i-1}}+E_{\ov{i-1}}K_{\bi}=K^{-1}_iE_{i-1},$

 \noindent  $K_{\bi}F_{\bi}+qF_{\bi}K_{\bi}=F_iK_i,\quad qK_{\bi}F_{\ov{i-1}}+F_{\ov{i-1}}K_{\bi}=K_iF_{i-1},$

\noindent  $K_{\bi}E_j-E_jK_{\bi}=K_{\bi}F_j-F_jK_{\bi}=K_{\bi}E_{\bj}+E_{\bj}K_{\bi}=K_{\bi}F_{\bj}+F_{\bj}K_{\bi}=0$ for $j\neq i, i-1$;

\vspace{0.1in}

 \item[(QQ4)]
 $E_iF_j-F_jE_i=\ds\delta_{ij}\frac{K_iK^{-1}_{i+1}-K^{-1}_{i}K_{i+1}}{q-q^{-1}}$,

 \noindent $E_{\bi}F_{\bj}+F_{\bj}E_{\bi}=\ds\delta_{ij}\Big(\frac{K_iK_{i+1}-K^{-1}_{i}K^{-1}_{i+1}}{q-q^{-1}}+(q-q^{-1})K_{\bi}K_{\ov{i+1}}\Big)$,

 %$qE_{i+1}F_i-F_iE_{i+1}=E_iF_{i+1}-qF_{i+1}E_i=E_iF_j-F_jE_i=0,$
 %for $|i-j|>1$;

 %$qE_{\ov{i+1}}F_{\bi}+F_{\bi}E_{\ov{i+1}}=E_{\bi}F_{\ov{i+1}}+qF_{\ov{i+1}}E_{\bi}=E_{\bi}F_{\bj}+F_{\bj}E_{\bi}=0,$
 %for $|i-j|>1$;

\noindent  $E_iF_{\bj}-F_{\bj}E_i=\ds\delta_{ij}(K^{-1}_{i+1}K_{\bi}-K_{\ov{i+1}}K^{-1}_i)$,

% $qE_{i+1}F_{\bi}-F_{\bi}E_{i+1}=E_iF_{\ov{i+1}}-qF_{\ov{i+1}}E_i=E_iF_{\bj}-F_{\bj}E_i=0$
% for $|i-j|>1$;

\noindent  $E_{\bi}F_j-F_jE_{\bi}=\delta_{ij}(K_{i+1}K_{\bi}-K_{\ov{i+1}}K_i)$;

 %$qE_{\ov{i+1}}F_i-F_iE_{\ov{i+1}}=E_{\bi}F_{i+1}-qF_{i+1}E_{\bi}=E_{\bi}F_j-F_jE_{\bi}=0$
% for $|i-j|>1$;

\vspace{0.1in}

 \item[(QQ5)]
  $E^2_{\bi}=\ds-\frac{q-q^{-1}}{q+q^{-1}}E^2_{i},\quad F^2_{\bi}=\ds \frac{q-q^{-1}}{q+q^{-1}}F^2_{i}$,

 \noindent $E_iE_{\bar j}-E_{\bar j}E_i=F_iF_{\bar j}-F_{\bar j}F_i=0$
 for  $|i-j|\neq 1$,

\noindent  $E_iE_j-E_jE_i=F_iF_j-F_jF_i=E_{\bi}E_{\bj}+E_{\bj}E_{\bi}=F_{\bi}F_{\bj}+F_{\bj}F_{\bi}=0$
for  $|i-j|>1$,

\noindent $E_iE_{i+1}-qE_{i+1}E_i=E_{\bi}E_{\ov{i+1}}+qE_{\ov{i+1}}E_{\bi},$
 $E_iE_{\ov{i+1}}-qE_{\ov{i+1}}E_i=E_{\bi}E_{i+1}-qE_{i+1}E_{\bi},$

\noindent  $qF_{i+1}F_i-F_iF_{i+1}=qF_{\ov{i+1}}F_{\bar i}+F_{\bar i}F_{\ov{i+1}},$
 $qF_{\ov{i+1}}F_i-F_iF_{\ov{i+1}}=qF_{i+1}F_{\bi}-F_{\bi}F_{i+1}$;

 \vspace{0.1in}

  \item[(QQ6)]
 $E_i^2E_{j}-(q+q^{-1})E_iE_{j}E_i+E_{j}E^2_i=0$,
 $F_i^2F_j-(q+q^{-1})F_iF_jF_i+F_jF^2_i=0$,

\noindent $E_i^2E_{\ov{j}}-(q+q^{-1})E_iE_{\ov{j}}E_i+E_{\ov{j}}E^2_i=0$,
$F_i^2F_{\ov{j}}-(q+q^{-1})F_iF_{\ov{j}}F_i+F_{\ov{j}}F^2_i=0$,

 \noindent where $|i-j|=1$.
\end{itemize}
\end{prop}

\begin{rem}\label{twogenerators}
The generators in \eqref{q-generator} are different from those in \cite[Theorem 2.1]{GJKK}.
Actually, setting
\begin{equation}\label{old generators}
E'_j=K_{j+1}^{-1}E_j,\quad E'_{\bj}=K_{j+1}^{-1}E_{\bj},\quad F'_j=F_jK_{j+1},\quad F'_{\bj}=F_{\bj}K_{j+1}
\end{equation}
for $1\leq j\leq n-1$,
then by~\cite[Remark 1.2]{GJKKK} the elements $K_i^{\pm1}, K_{\bi}, E'_j, F'_j, E'_{\bj}, F'_{\bj}$ can be identified with $q^{\pm k_i}, k_{\bi},
e_j, f_j, e_{\bj}, f_{\bj}$ in \cite[Theorem 2.1]{GJKK}, respectively.
The relations in Proposition~\ref{presentqUq} are obtained by rewriting the whole defining relations
among $q^{\pm k_i}, k_{\bi},
e_j, f_j, e_{\bj}, f_{\bj}$ given by \cite[Theorem 2.1]{GJKK} in terms of the elements $K_i^{\pm1}, K_{\bi}, E_j, F_j, E_{\bj}, F_{\bj}$ by using \eqref{old generators}.
\end{rem}

For convenience, let us write down some of the relations
in terms of the generators $E'_j, F'_j, E'_{\bj}, F'_{\bj}$ as follows, which will be useful later on.
\begin{lem}[{\cite[Theorem 2.1]{GJKK}}]\label{old-reln}
The following holds in $U_q(\mathfrak{q}(n))$: for $1\leq k\leq n-2, 1\leq a\leq n$ and $1\leq i,j\leq n-1$ with $|i-j|>1$,
\begin{align*}
&E'_{k+1}E'_{\bar{k}}-E'_{\bar{k}}E'_{k+1}=E'_{\ov{k+1}}E'_{k}-E'_{k}E'_{\ov{k+1}}, \quad
F'_{k+1}F'_{\bar{k}}-F'_{\bar{k}}F'_{k+1}=F'_{\ov{k+1}}F'_{k}-F'_{k}F'_{\ov{k+1}}, \\
&K_aE'_i=q^{(\ep_a,\al_i)}E'_iK_a,~ K_aF'_i=q^{-(\ep_a,\al_i)}F'_iK_a,~
K_aE'_{\bi}=q^{(\ep_a,\al_i)}E'_{\bi}K_a,~ K_aF'_{\bi}=q^{-(\ep_a,\al_i)}F'_{\bi}K_a, \\
&E'_iE'_j-E'_jE'_i= F'_iF'_j-F'_jF'_i=0.
\end{align*}
\end{lem}

By~\cite[\S 4]{Ol} (cf. \cite[(2.9)]{GJKK}), the comultiplication $\Delta:\qUq\rightarrow\qUq\otimes\qUq$ is defined by
\begin{equation}\label{comult}
\Delta(L_{i,j})=\sum^j_{k=i}L_{i,k}\otimes L_{k,j},
\end{equation}
for $i,j\in I(n|n)$ with $i\leq j$.
Then by \eqref{q-generator} we have
\begin{equation}\label{q-comult}
\aligned
\Delta(K_i)&=K_i\otimes K_i,\\
\Delta(E_j)&=1\otimes E_j+E_j\otimes K_j^{-1}K_{j+1},\quad
\Delta(F_j)=K_jK_{j+1}^{-1}\otimes F_j+F_j\otimes 1.
\endaligned
\end{equation}
for $1\leq i\leq n$ and $1\leq j\leq n-1$.

By Proposition~\ref{presentqUq}, it is routine to check that there is an anti-involution $\Omega:\qUq\rightarrow\qUq$ given by
\begin{equation}\label{Omega}
\aligned
\Omega(q)&=q^{-1},\quad \Omega(K_i)=K_i^{-1},\quad \Omega(K_{\bi})=K_{\bi},\\
\Omega(E_j)&=F_j,\quad \Omega(F_j)=E_j, \quad \Omega(E_{\bj})=F_{\bj},\quad \Omega(F_{\bj})=E_{\bj}.
\endaligned
\end{equation}
for $1\leq i\leq n$ and $1\leq j\leq n-1$.

As for Lie algebras, we have the following PBW Theorem for $\qUq$ due to Olshanski.
%\begin{lem}\cite[Theorem 6.2]{Ol}\label{qUq-Cartan}
%The set $\{K_1^{b_1}\cdots K_n^{b_n}K^{\tau_1}_{\bar i}\cdots K^{\tau_n}_{\bar n}\mid b_i\in\Z,\tau_i\in\Z_2\}$
%is a basis of $U_q^0$.
%\end{lem}

\begin{prop}[{\cite[Theorem 6.2]{Ol}}]\label{Ol-PBW}
Fix an order on $L_{i,j}$ for $i,j\in I(n|n)$ with $i<j$. Then the set
\begin{equation}\label{PBW-basis-L}
\Big\{L_{1,1}^{m_1}\cdots L_{n,n}^{m_n}\prod_{i<j\in I(n|n)}L_{i,j}^{m_{ij}}\mid m_1,\ldots,m_n\in\Z, m_{ij}\in\N, m_{ij}\in\Z_2 \text{ if }\hat{L}_{i,j}=1\Big\},
\end{equation}
is a $\Q(q)$-basis of the algebra $\qUq$.
\end{prop}

\begin{rem} In \cite{Ol}, a particular order on the elements $L_{i,j}$ for $i,j\in I(n|n)$ with $i<j$ was chosen
to prove the PBW Theorem. Actually, it is easy to check that the arguments in the proof of \cite[Theorem 6.2]{Ol}
do not depend on the choice of the order.
%Following \cite[\S 6]{Ol}, let us introduce a partial order $\preceq$ on the set of generators $L_{ij}$ for $i\geq j\in I(n|n)$
%as follows:
%\begin{align}
%L_{ii}\preceq L_{kl}, \text{ for any }i,k<l\in I(n|n),\notag\\
%L_{ij}\preceq L_{kl} \text{ if } i>k \text{ or if } i=k\text{ and }j>l, \text{ for }i<j,k<l\in I(n|n).\label{partialorder}
%\end{align}

%Any word in the letters $L_{ij}, i\leq j\in I(n|n)$ is called a monomial.
%A monomial is said to be {\em reduced} if it has the form
%$$
%(L_{11})^{m_1}\cdots L_{nn}^{m_n}(L_{i_1j_1})^{p_1}\cdots (L_{i_sj_s})^{p_s}
%$$
%where $L_{i_1j_1}\preceq\cdots\preceq L_{i_sj_s}$, $m_1,\ldots,m_n$ are integers, $s$ is a nonnegative integer and $p_1,\ldots,p_s$ are positive integers such that $p_a=1$ if $L_{i_aj_a}$ is odd for $1\leq a\leq s$.
%Then we have the following due to Olshanski.
\end{rem}

For $\al_{i,j}=\ep_i-\ep_j\in\Phi$ with $1\leq i\neq j\leq n$, we introduce inductively
the root vectors as follows.
For $1\leq i\leq n-1$, we set
$$
X_{i,i+1}=E_i,\quad X_{i+1,i}=F_i, \quad \oX_{i,i+1}=E_{\bi},\quad  \oX_{i+1,i}=F_{\bi}.
$$
For $|j-i|>1$, we define
\begin{equation}\label{q-root}
\aligned
X_{\al_{i,j}}\equiv \Xije&:=\left\{
\begin{array}{ll}
X_{i,k}X_{k,j}-q X_{k,j}X_{i,k},&\text{ if }i<j,\\
X_{i,k} X_{k,j}-q^{-1} X_{k,j}X_{i,k},&\text{ if }i>j,
\end{array}
\right.\\
\oX_{\al_{i,j}}\equiv\Xijo &:=\left\{
\begin{array}{ll}
X_{i,k} \oX_{k,j}-q \oX_{k,j}X_{i,k},&\text{ if }i<j,\\
\oX_{i,k} X_{k,j}-q^{-1}X_{k,j} \oX_{i,k},&\text{ if }i>j,
\end{array}
\right.
\endaligned
\end{equation}
where $k$ is strictly between $i$ and $j$.
It is straightforward to check that $X_{i,j},\oX_{i,j}$ are independent of the choice of $k$.
Clearly $X_{i,j}$ are even elements while $\oX_{i,j}$ are odd elements.
By~\eqref{q-generator} and (QQ2) in Proposition~\ref{presentqUq}, one can check that
for $\al\in\Phi, 1\leq i\leq n$
\begin{align}\label{KX}
K_iX_{\al}K_i^{-1}=q^{(\ep_i,\al)}X_{\al},\quad
K_i\oX_{\al}K_i^{-1}=q^{(\ep_i,\al)}\oX_{\al}.
\end{align}
Therefore the elements $X_{\al}$ and $\oX_{\al}$ can be viewed as the quantum analog of the root vectors
$x_{\al}$ and $\ox_{\al}$ introduced in~\eqref{root}.
Meanwhile, by~\eqref{Omega}, we obtain
\begin{equation}\label{Omega-root}
\aligned
\Omega(\Xije)=\Xjie,\quad
\Omega(\Xijo)=\Xjio.
\endaligned
\end{equation}
for $1\leq i\neq j\leq n$.

By~\eqref{q-generator} and~\eqref{Ol-reln}, a direct calculation shows that
\begin{align}
K_aL_{i,j}=\left\{
\begin{array}{ll}
q^{-1}L_{i,j}K_a,&\text{ if }|i|=a,\\
qL_{i,j}K_a,&\text{ if }|j|=a,\\
L_{i,j}K_a,&\text{ otherwise},
\end{array}
\right.
K_a^{-1}L_{i,j}=\left\{
\begin{array}{ll}
q L_{i,j}K_a^{-1},&\text{ if }|i|=a,\\
q^{-1}L_{i,j}K_a^{-1},&\text{ if }|j|=a,\\
L_{i,j}K_a^{-1},&\text{ otherwise}.
\end{array}
\right.\label{KL}
\end{align}
for $1\leq a\leq n$ and $i,j\in I(n|n)$ with $i\leq j$ and $|i|\neq |j|$.

\begin{lem}\label{XL}
The following holds for $1\leq i<j\leq n$:
\begin{align*}
\Xije&=
\ds \frac{-1}{q-q^{-1}}K_jL_{-j,-i},\quad
\Xijo=
\ds\frac{-1}{q-q^{-1}}K_jL_{-j,i},\\
\Xjie&=
\ds \frac{1}{q-q^{-1}}L_{i,j}K_j^{-1}, \quad
\Xjio=
\ds\frac{-1}{q-q^{-1}}L_{-i,j}K_j^{-1}.
\end{align*}
\end{lem}
\begin{proof}
Suppose $1\leq i<j\leq n$.
Recall the generators~\eqref{old generators} used in~\cite[Theorem 2.1]{GJKK}.
Then by~\cite[(2.4)]{GJKK} and Remark~\ref{twogenerators}, we have
\begin{equation}\label{generator-L}
\aligned
L_{-j,-i}=&
(-1)^{j-i}(q-q^{-1})\Big(\prod^{j-1}_{a=i+1}K_{a}\Big)\cdot\prod^{j-1}_{a=i+1}{\rm ad}E'_{a}(E'_i),\\
L_{-j,i}=&
(-1)^{j-i}(q-q^{-1})\Big(\prod^{j-1}_{a=i+1}K_{a}\Big)\cdot\prod^{j-1}_{a=i+1}{\rm ad}E'_{a}(E'_{\bi}),\\
L_{i,j}=&
(q-q^{-1})\Big(\prod^{j-1}_{a=i+1}K^{-1}_{a}\Big)\cdot\prod^{j-1}_{a=i+1}{\rm ad}F'_{a}(F'_i),\\
L_{-i,j}=&
-(q-q^{-1})\Big(\prod^{j-1}_{a=i+1}K^{-1}_{a}\Big)\cdot\prod^{j-1}_{a=i+1}{\rm ad}F'_{a}(F'_{\bi}),
\endaligned
\end{equation}
where ${\rm ad}G_a(G_i):=G_aG_i-G_iG_a, \prod_{a=i+1}^{j-1}{\rm ad}G_a(G_i):={\rm ad}G_{j-1}\cdots{\rm ad}G_{i+1}(G_i)$
if $j\geq i+2$
and $\prod_{a=i+1}^{j-1}{\rm ad}G_a(G_i)=G_i$ if $j=i+1$, for $G_i=E'_i,E'_{\bi},F'_i,F'_{\bi}$ and $G_a=E'_a,F'_a$ with $i+1\leq a\leq j-1$.
The first formula in~\eqref{generator-L} and Lemma~\ref{old-reln} imply
$$\aligned
L_{-j,-i}=&(-1)^{j-i}(q-q^{-1})\Big(\prod^{j-1}_{a=i+1}K_{a}\Big)\cdot\prod^{j-1}_{a=i+2}{\rm ad}E'_{a}(-{\rm ad}E'_i(E'_{i+1}))\\
=&(-1)^{j-i-1}(q-q^{-1})\Big(\prod^{j-1}_{a=i+1}K_{a}\Big)\cdot{\rm ad}E'_i\Big(\prod^{j-1}_{a=i+2}{\rm ad}E'_{a}(E'_{i+1})\Big)\\
=&K_{i+1}\cdot{\rm ad}E'_i(L_{-j,-i-1}),\\
\endaligned$$
then, by \eqref{KL} and \eqref{old generators},
\begin{equation}\label{Lji-ind}
\aligned
L_{-j,-i}=&K_{i+1}E_i'L_{-j,-i-1}-K_{i+1}L_{-j,-i-1}E_i'\\
=& K_{i+1}E_i'L_{-j,-i-1}-qL_{-j,-i-1}K_{i+1}E_i'\\
=& E_iL_{-j,-i-1}-qL_{-j,-i-1}E_i.
\endaligned
\end{equation}
%where the second equality is due to \eqref{KL} and the last equality follows from~\eqref{old generators}.
Similarly, by \eqref{generator-L} and the corresponding equalities in Lemma~\ref{old-reln}, one can prove
\begin{equation}\label{L-ind}
\aligned
L_{-j,i}=& E_iL_{-j,i+1}-qL_{-j,i+1}E_i,\\
L_{i,j}=& L_{i+1,j}F_i-q^{-1}F_iL_{i+1,j},\quad
L_{-i,j}= L_{-i-1,j}F_i-q^{-1}F_iL_{-i-1,j}.
\endaligned
\end{equation}

We now prove the first formula in the lemma by induction on $j-i$.
Indeed, if $j=i+1$, then
$$
\Xije=E_{j-1}=\frac{-1}{q-q^{-1}}K_{j}L_{-j,-j+1}
$$
by~\eqref{q-generator}.
Now assume $j\geq i+2$, then by induction and~\eqref{q-root} we have
\begin{align*}
\Xije&=X_{i,i+1}X_{i+1,j}-qX_{i+1,j}X_{i,i+1}\\
&=E_i \frac{-1}{q-q^{-1}}K_{j}L_{-j,-i-1}-q\frac{-1}{q-q^{-1}}K_{j}L_{-j,-i-1} E_i\\
&=\frac{-1}{q-q^{-1}}K_{j}(E_iL_{-j,-i-1}-qL_{-j,-i-1} E_i)\quad(\text{by \eqref{KL} as $j\geq i+2$})\\
&=\frac{-1}{q-q^{-1}}K_{j}L_{-j,-i}\quad (\text{by \eqref{Lji-ind}}),
\end{align*}
as desired. By a parallel argument, we can prove the remaining three formulas.
\end{proof}
Fix an order in $\Phi^+$, or equivalently in the set $\{(i,j)\mid 1\leq i<j\leq n\}$.
Recall that $M_n(\N|\Z_2)'$ is the set consisting of $C\in M_n(\N|\Z_2)$ such that $C^0_0=0$.
Given $C=(C_0,C_1)\in M_n(\N|\Z_2)'$,
we can introduce the elements
\begin{align*}
%X^+_A&=\oX^{A^1_{12}}_{12}X^{A^0_{12}}_{12}\oX^{A^1_{23}}_{23}\cdots\oX^{A^1_{n-1,n}}_{n-1,n}\cdots X^{A^0_{1n}}_{1n}X^{A^0_{1n}}_{1n}\cdots X^{A^0_{n-1,n}}_{n-1,n},\\
%X^-_A&=X^{A^0_{n-1,n}}_{n,n-1} X^{A^0_{n-2,n}}_{n,n-2}\cdots X^{A^0_{1n}}_{n1} \cdots X^{A^0_{12}}_{21}
%\oX^{A^1_{1n}}_{n1}\cdots \oX^{A^1_{12}}_{21}\cdots \oX^{A^1_{n-2,n-1}}_{n-1,n-2} \oX^{A^1_{n-1,n}}_{n,n-1}.
X_{C^+}=\prod_{1\leq i<j\leq n}\big(X_{i,j}^{c^0_{ij}}~\oX_{i,j}^{c^1_{ij}}\big), \quad
X_{C^-}=\prod_{1\leq i<j\leq n}\big(X_{j,i}^{c^0_{ji}}~\oX_{j,i}^{c^1_{ji}}\big),\quad
\ov{K}_{C^0_1}=K^{c^1_{11}}_{\bar 1}\cdots K^{c^1_{nn}}_{\bar n},
\end{align*}
where $C_0=(c^0_{ij}), C_1=(c^1_{ij})$.
For $\sigma=(\sigma_1,\ldots,\sigma_n)\in\Z^n$, let
$
K_{\sigma}=K_1^{\sigma_1}\cdots K_n^{\sigma_n}.
$

\begin{prop}\label{PBW-X}
%The sets $\{X^+_A\mid A\in\N^{\Phi^+}\times \Z_2^{\Phi^+}\}$
%and $\{X^-_A\mid A\in\N^{\Phi^+}\times \Z_2^{\Phi^+}\}$ are $\Q(q)$-bases of $U^+_q$ and $U^-_q$, respectively.
The set
\begin{align}
\big\{X_{C^-}K_{\sigma}\ov{K}_{C^0_1} X_{C^+}
\mid C\in M_n(\N|\Z_2)',  \sigma \in\Z^n \big\}\label{PBW-basis-X}
\end{align}
is a $\Q(q)$-basis of $\qUq$. In particular, if $U_q(\ev{\mfq})$ denotes the subalgebra generated by $E_1,\ldots, E_{n-1}, F_1,\ldots, F_{n-1}, K^{\pm1}_1,\ldots, K^{\pm1}_n$, then
\begin{equation}\label{qUqev}
U_q(\ev{\mfq})\cong U_q(\mathfrak{gl}(n)),
\end{equation}
which is a Hopf algebra isomorphism.
%Moreover, by~\eqref{q-comult} we obtain that~\eqref{qUqev} is actually a Hopf algebra isomorphism.
\end{prop}
\begin{proof}
Suppose $ C\in M_n(\N|\Z_2)'$ and $C_0=(c^0_{ij}), C_1=(c^1_{ij})$.
By Lemma~\ref{XL} and \eqref{KL}, the following holds
\begin{equation}\label{twobasis}
\aligned
X_{C^-}K_{\sigma}\ov{K}_{C^0_1} X_{C^+}
=&g(C,\sigma) L_{1,1}^{m_1}\cdots L_{n,n}^{m_n}\cdot\prod_{1\leq i<j\leq n}\big( L_{i,j}^{c^0_{ji}}~
L_{-i,j}^{c^1_{ji}} \big)\cdot\\
&L^{c^1_{11}}_{-1,1}\cdots L_{-n,n}^{c^1_{nn}}\cdot \prod_{1\leq i<j\leq n}\big(L_{-j,-i}^{c^0_{ij}}~
L_{-j,i}^{c^1_{ij}}\big),
\endaligned
\end{equation}
where $g(C,\sigma)\in\Q(q)$ is nonzero, $m_j=\sigma_j+\sum^j_{i=1}(c^0_{ij}+c^1_{ij}-c^0_{ji}-c^1_{ji})$
for $1\leq j\leq n$, and the product is taken with to the fixed order on $\{(i,j)\mid 1\leq i<j\leq n\}$.
This implies, up to nonzero scalars, the set \eqref{PBW-basis-X} actually coincides with
the set \eqref{PBW-basis-L} where the order on $L_{i,j}$ for $i,j\in I(n|n)$ with $i<j$ is taken to be compatible
with the product on the right hand side of \eqref{twobasis}.
Then by Proposition~\ref{Ol-PBW}, the first assertion is verified.
%By \eqref{q-root-pe} and \eqref{q-root-po}, we have $X_{ij},\oX_{ij}\in U_q^+$ and hence $U_q^+$ is also generated by
%$X_{ij},\oX_{ij}$ for $1\leq i<j\leq n$.
%Then by Lemma~\ref{q-ppee}, Lemma~\ref{q-ppeo} and Lemma~\ref{q-ppoo}, it is easy to deduce that
%$U^+_q$ is spanned by the set $\{X^+_A\mid A\in\N^{\Phi^+}\times \Z_2^{\Phi^+}\}$.
%Furthermore, by \eqref{Kbarsq}, \eqref{KL}, \eqref{XijLji} and \eqref{oXijLji} we obtain that $X_A^+$ can be written as
%$$
%X_A^+=g_AL_{11}^{m_1}\cdots L_{nn}^{m_n}L^{A1_{12}}_{-2,1}L^{A^0_{12}}_{-2,-1}L^{A^1_{23}}_{-3,2}\cdots L^{A^0_{n-1,n}}_{-n,-n+1}
%$$
%for some nonzero element $g_A\in\Q(q)$ and $m_1,\ldots,m_n\in\Z$.
%Hence up to a scalar $X_A^+$ for $A\in\N^{\Phi^+}\times \Z_2^{\Phi^+}$ are pairwise distinct reduced monomial in the elements $L_{-j,-i}, L_{-j,i}$ for $1\leq i<j\leq n$.
%This together with Proposition~\ref{Ol-PBW} implies that the set $\{X^+_A\mid A\in\N^{\Phi^+}\times \Z_2^{\Phi^+}\}$ is linearly independent and therefore it forms a basis of $U^+_q$.
%Similarly, one can prove that the set $\{X^-_A\mid A\in\N^{\Phi^+}\times \Z_2^{\Phi^+}\}$ are $\Q(q)$-bases $U^-_q$.
%Then the finally statement of the lemma follows from Lemma~\ref{qUq-Cartan} and Proposition~\ref{GJKK-PBW}.
For the second assertion,  we observe that the relations in Proposition~\ref{presentqUq} involving $E_i, F_i, K_j$
for $1\leq i\leq n-1$ and $1\leq j\leq n$ are the same as the standard relations for $U_q(\mathfrak{gl}(n))$. This gives a homomorphism from  $U_q(\mathfrak{gl}(n))$ to $U_q(\ev{\mfq})$ which is an isomorphism by the first assertion. The last assertion is clear.
\end{proof}
\begin{rem}
Under the identification~\eqref{qUqev}, the elements $X_{\al}$ and $X_{-\al}$
up to scalar multiplication coincides with the root vectors defined in \cite[(3.1)]{DG}.
More precisely, let $X'_\al=q^{i-j+1}X_\al$ and $X'_{-\al}=q^{j-i-1}X_{-\al}$ for $\al=\ep_i-\ep_j\in\Phi^+$.
Then $X'_{\al}$ and $X'_{-\al}$ correspond to the root vectors $\mathrm{X}_{\al}$ and $\mathrm{X}_{-\al}$ in \cite[(3.1)]{DG} (cf. \cite{Xi}).
\end{rem}

Let $U_q^0$ be the subalgebra of $\qUq$ generated by $K^{\pm1}_i,K_{\bi}$ for
$1\leq i\leq n$, and let $U_q^+$ (resp. $U_q^-$) be the subalgebra of $\qUq$ generated by the elements $E_j,E_{\bj}$ (resp. $F_j,F_{\bj}$)
for $1\leq j\leq n-1$. We have reproduced the following.

\begin{prop}[{\cite[Theorem 2.3]{GJKK}}]\label{GJKK-PBW}
There is a $\Q(q)$-linear isomorphism
$$
\qUq\cong U_q^{-}\otimes U_q^0\otimes U_q^+.
$$
\end{prop}

\section{Commutation formulas for quantum root vectors}
We divide the commutation formulas into four groups which will be discussed in four cases below.
Each of the first three cases consists of two lemmas, dealing with the case of two positive (or negative) roots and the case of one positive and one negative roots. The three cases are divided according to whether the pair of root vectors are even-even, even-odd, or odd-odd. Moreover, by the anti-automorphism $\Omega$ given in \eqref{Omega-root}, it suffices to look at the commutation formulas of positive root vector $X_{i,j},\Xijo$ $(i<j)$ with others.

\vspace{.3cm}
\noindent
{\bf Case 1}---Commutation formulas for two even-even root vectors $\Xije\Xkle$ ($i<j,k\neq l$). %which will be discussed in Lemmas \ref{q-ppee} and \ref{q-pnee}.
\begin{lem}\label{q-ppee}
The following holds  for $1\leq i,j,k,l\leq n$ satisfying $i<j, k<l$:
$$
\Xije\Xkle=\left\{
\begin{array}{ll}
\Xkle\Xije,& (i<j<k<l\text{ or }i<k<l<j),\\
q^{-1}\Xkle\Xije,&(i<k<j=l\text{ or }i=k<j<l),\\
q\Xkle\Xije+\Xile,&(i<k=j<l),\\
\Xkle\Xije-(q-q^{-1})\Xkje\Xile,&(i<k<j<l).
\end{array}
\right.
$$
%\begin{align}
%\Xije\Xkle=&\Xkle\Xije,\quad  (i<j<k<l\text{ or }i<k<l<j),\label{q-ppee1}\\
%\Xije\Xkle=&q^{-1}\Xkle\Xije,\quad  (i<k<j=l\text{ or }i=k<j<l),\label{q-ppee2}\\
%\Xije\Xkle=&q\Xkle\Xije+\Xile ,\quad  (i<k=j<l),\label{q-ppee3}\\
%\Xije\Xkle=&\Xkle\Xije-(q-q^{-1})\Xkje\Xile,\quad  (i<k<j<l).\label{q-ppee4}
%\end{align}
%\begin{align}
%\Xije\Xkle&=\left\{
%\begin{array}{ll}
%\Xkle\Xije,&\text{ if }i<j<k<l\text{ or }i<k<l<j, \\
%q^{-1}\Xkle\Xije,&\text{ if }i<k<j=l\text{ or }i=k<j<l,\\
%\Xkle\Xije-\Xile K_j^{-1},&\text{ if }i<j=k<l,\\
%\Xkle\Xije-(q-q^{-1})\Xkje\Xile,&\text{ if }i<k<j<l.
%\end{array}
%\right.
%\label{q-comm-ppee}\\
%\end{align}

%\begin{subequation}
%\begin{align*}

%\end{align*}
%\end{subequation}
\end{lem}
\begin{proof}
Suppose $1\leq i,j,k,l\leq n$ and $i<j, k<l$.
Using the formula~\eqref{Ol-reln}, a straightforward calculation shows that
$$
L_{-j,-i}L_{-l,-k}=\left\{
\begin{array}{ll}
L_{-l,-k}L_{-j,-i},&(i<j<k<l\text{ or }i<k<l<j),\\
q^{-1}L_{-l,-k}L_{-j,-i},&(i<k<j=l\text{ or }i=k<j<l),\\
L_{-l,-k}L_{-j,-i}-(q-q^{-1})L_{-j,-j}L_{-l,-i},&(i<k=j<l),\\
L_{-l,-k}L_{-j,-i}-(q-q^{-1})L_{-j,-k}L_{-l,-i},&(i<k<j<l).\\
\end{array}
\right.
$$
Then the lemma is proved case-by-case using Lemma~\ref{XL}.
Let us illustrate by checking in detail the case when $i<k=j<l$.
In this case, by \eqref{q-generator}, Lemma~\ref{XL} and the above formula,  we have
\begin{align*}
\Xije\Xkle
=&\frac{-1}{q-q^{-1}}K_jL_{-j,-i}\cdot \frac{-1}{q-q^{-1}}K_lL_{-l,-k}=\frac{1}{(q-q^{-1})^2}K_jK_lL_{-j,-i}L_{-l,-k}\\
=&\frac{1}{(q-q^{-1})^2}K_jK_l\big(L_{-l,-k}L_{-j,-i}-(q-q^{-1})L_{-j,-j}L_{-l,-i}\big)\\
=&\frac{1}{(q-q^{-1})^2}K_lL_{-l,-k}\cdot qK_jL_{-j,-i}-\frac{1}{q-q^{-1}}K_jL_{-j,-j}K_lL_{-l,-i}\\
=&q\Xkle\Xije+\Xile,
\end{align*}
where the second equality and fourth equality are due to \eqref{KL}.
The remaining cases can be verified similarly and we skip the detail.
\end{proof}
Observe that we can derive another set of commutation formulas for
$\Xije$ and $\Xkle$ by solving for $\Xkle\Xije$
in Lemma~\ref{q-ppee} and then interchanging $(i,j)$ and $(k,l)$.
Namely, we have
\begin{equation}\label{q-ppee-extra}
\Xije\Xkle=\left\{
\begin{array}{ll}
\Xkle\Xije,& (k<l<i<j\text{ or }k<i<j<l),\\
q\Xkle\Xije,& (k<i<l=j\text{ or }k=i<l<j),\\
q^{-1}\Xkle\Xije-q^{-1}\Xkje ,& (k<i=l<j),\\
\Xkle\Xije+(q-q^{-1})\Xkje\Xile,& (k<i<l<j).
\end{array}
\right.
\end{equation}
This together with Lemma~\ref{q-ppee} gives a complete commutation formulas for even positive root vectors
$\Xije$ and $\Xkle$ with $i<j,k<l$.

\begin{lem}\label{q-pnee}
Assume that $1\leq i,j,k,l\leq n$ satisfy $i<j, k>l$.
Then
\begin{align*}
\Xije\Xkle&=\left\{
\begin{array}{ll}
\Xkle\Xije+\ds\frac{K_iK_j^{-1}-K_i^{-1}K_j}{q-q^{-1}},&(i=l,j=k),\\
\Xkle\Xije,&(i<j\leq l<k\text{ or }i<l<k<j) ,\\
\Xkle\Xije-K_l^{-1}K_j\Xkje,&(i=l<j<k),\\
\Xkle\Xije+K_l^{-1}K_j\Xile,&(i<l<j=k),\\
\Xkle\Xije+(q-q^{-1})K_l^{-1}K_j\Xkje\Xile,&(i<l<j<k),\\
\end{array}
\right.
\end{align*}
\end{lem}
\begin{proof}
Suppose $1\leq i,j,k,l\leq n$ and $i<j, k>l$.
By ~\eqref{Ol-reln}, it is easy to check that
$$
L_{-j,-i}L_{l,k}=L_{l,k}L_{-j,-i}, \text{ if } i<j<l<k \text{ or }i<l<k<j.
$$
Then in this situation,  by Lemma~\ref{XL} and \eqref{KL}, we have $\Xije\Xkle=\Xkle\Xije$.
Similarly, in other cases, again by the formula~\eqref{Ol-reln}, we obtain
$$
L_{-j,-i}L_{l,k}=\left\{
\begin{array}{ll}
L_{l,k}L_{-j,-i}+(q-q^{-1})(L_{j,j}L_{-i,-i}-L_{-j,-j}L_{i,i}),~  &(i=l,j=k),\\
L_{l,k}L_{-j,-i}+(q-q^{-1})L_{j,k}L_{-i,-i},~ &(i=l<j<k),\\
qL_{l,k}L_{-j,-i},~ &(i<l=j<k),\\
L_{l,k}L_{-j,-i}+(q-q^{-1})L_{j,j}L_{-l,-i},~ &(i<l<j=k),\\
L_{l,k}L_{-j,-i}+(q-q^{-1})L_{j,k}L_{-l,-i},~ &(i<l<j<k).
\end{array}
\right.
$$
As before, the lemma is proved case-by-case using Lemma~\ref{XL}.
Let us illustrate by checking in detail the case when $i<l<j=k$.
In this case, by \eqref{q-generator}, Lemma~\ref{XL} and the above formulas we have
\begin{align*}
\Xije\Xkle
=&\frac{-1}{q-q^{-1}}K_jL_{-j,-i}\cdot\frac{1}{q-q^{-1}}L_{l,k}K_k^{-1}\\
=&\frac{-1}{(q-q^{-1})^2}K_j\big(L_{l,k}L_{-j,-i}+(q-q^{-1})L_{j,j}L_{-l,-i}\big)K_k^{-1}\\
=&\frac{-1}{(q-q^{-1})^2}L_{l,k}K_jq\cdot q^{-1}K_k^{-1}L_{-j,-i}-\frac{1}{q-q^{-1}}K_jL_{j,j}K_k^{-1}L_{-l,-i}\\
=&\frac{1}{q-q^{-1}}L_{l,k}K_k^{-1}\cdot\frac{-1}{q-q^{-1}}K_jL_{-j,-i}+K_jK_l^{-1}\frac{-1}{q-q^{-1}}K_lL_{-l,-i}\\
=&\Xkle\Xije+K_jK^{-1}_{l}\Xile,
\end{align*}
where the third equality is due to \eqref{KL} and the assumption $i<l<j=k$.
The remaining cases can be verified similarly and we skip the detail.
\end{proof}

Observe that applying the anti-automorphism $\Omega$ to the formulas in Lemma~\ref{q-pnee} and interchanging $(i,j)$ and $(k,l)$,
one can obtain another set of commutation formulas for $\Xije$ and $\Xkle$ as follows:
$$
\Xije\Xkle=\left\{
\begin{array}{ll}
\Xkle\Xije,&  (l<k\leq i<j\text{ or }l<i<j<k) ,\\
\Xkle\Xije-\Xkje K_iK_k^{-1},& (l=i<k<j),\\
\Xkle\Xije+\Xile K_iK^{-1}_k,& (l<i<k=j),\\
\Xkle\Xije-(q-q^{-1})\Xkje\Xile K_iK^{-1}_k,& (l<i<k<j).
\end{array}
\right.
$$

%\begin{align}
%\Xije\Xlke&=\left\{
%\begin{array}{ll}
%\Xlke\Xije,&\text{ if }k<l<i<j\text{ or }k<i<j<l ,\\
%q^{-1}\Xlke\Xije,&\text{ if }k<l=i<j,\\
%\Xlke\Xije-(-1)^{l-i}\Xike K^{-1}_j,&\text{ if }k<i<l=j,\\
%\Xlke\Xije-(-1)^{l-i}\Xlje K_i,&\text{ if }k=i<l<j,\\
%\Xlke\Xije-(-1)^{l-i}(q-q^{-1})\Xlje\Xike,&\text{ if }k<i<l<j,\\
%\end{array}
%\right.\label{q-comm-pnee-1}
%\end{align}
This together with Lemma~\ref{q-pnee} gives a complete commutation formulas for even positive root vectors $\Xije$ and
the even negative root vectors $\Xkle$ with $i<j,k>l$.

\vspace{.3cm}
\noindent
{\bf Case 2}---Commutation formulas for two even-odd root vectors $\Xije\Xklo$ ($i<j,k\neq l$).% which will be discussed in Lemmas \ref{q-ppeo} and \ref{q-pneo}.
\begin{lem}\label{q-ppeo}
Suppose that $1\leq i,j,k,l\leq n$ with $i<j$ and $k<l$.
\begin{enumerate}
\item If $i=k,j=l$ or $i<j<k<l\text{ or }k<l<i<j \text{ or }k<i<j<l$, then
$$
\Xije\Xklo=\Xklo\Xije.
$$
\item In other cases, the following formulas hold:
$$
\Xije\Xklo=\left\{
\begin{array}{ll}
\Xklo\Xije+(q-q^{-1})(\Xkjo\Xile-\Xkje\Xilo),& (i<k<l<j), \\
q^{-1}\Xklo\Xije,&  (i=k<j<l), \\
q\Xklo\Xije+\Xilo,&  (i<k=j<l), \\
q\Xklo\Xije-(q-q^{-1})\Xkje\Xilo,&  (i<k<j=l), \\
\Xklo\Xije-(q-q^{-1})\Xkje\Xilo,&  (i<k<j<l), \\
q^{-1}\Xklo\Xije+q^{-1}(q-q^{-1})\Xkjo\Xile,& (k=i<l<j),\\
q^{-1}\Xklo\Xije-q^{-1}\Xkjo,&  (k<i=l<j),\\
q\Xklo\Xije,&  (k<i<l=j),\\
\Xklo\Xije+(q-q^{-1})\Xkjo\Xile,&  (k<i<l<j).
\end{array}
\right.
$$
\end{enumerate}
\end{lem}
\begin{proof}
Let $i,j,k,l\in\{1,2,\ldots,n\}$ satisfy $i<j$ and $k<l$.
As before by ~\eqref{Ol-reln}, one can check that if $i=k,j=l$ or $i<j<k<l\text{ or }k<l<i<j \text{ or }k<i<j<l$, then
$L_{-j,-i}L_{-l,k}=L_{-l,k}L_{-j,-i}$. Then part (1) is proved by Lemma~\ref{XL} and \eqref{KL}.
To prove part (2), again by~\eqref{Ol-reln} we obtain
$$
L_{-j,-i}L_{-l,k}=\left\{
\begin{array}{ll}
L_{-l,k}L_{-j,-i}+(q-q^{-1})(L_{-j,k}L_{-l,-i}-L_{-j,-k}L_{-l,i}),\quad &(i<k<l<j), \\
q^{-1}L_{-l,k}L_{-j,-i},\quad &(i=k<j<l), \\
L_{-l,k}L_{-j,-i}-(q-q^{-1})L_{-j,-j}L_{-l,i},\quad &(i<k=j<l), \\
qL_{-l,k}L_{-j,-i}-(q-q^{-1})L_{-j,-k}L_{-l,i},\quad &(i<k<j=l), \\
L_{-l,k}L_{-j,-i}-(q-q^{-1})L_{-j,-k}L_{-l,i},\quad &(i<k<j<l), \\
q^{-1}L_{-l,k}L_{-j,-i}+q^{-1}(q-q^{-1})L_{-j,k}L_{-l,-i},\quad &(k=i<l<j),\\
L_{-l,k}L_{-j,-i}+(q-q^{-1})L_{-j,k}L_{-i,-i},\quad &(k<i=l<j),\\
qL_{-l,k}L_{-j,-i},\quad & (k<i<l=j),\\
L_{-l,k}L_{-j,-i}+(q-q^{-1})L_{-j,k}L_{-l,-i},\quad &(k<i<l<j). \\
\end{array}
\right.
$$
As before, part (2) of the lemma is proved case-by-case using Lemma~\ref{XL}.
Let us illustrate by checking in detail the case when $k=i<l<j$.
In this case, by the above formula and Lemma~\ref{XL} we have
\begin{align*}
\Xije\Xklo
=&\frac{-1}{q-q^{-1}}K_jL_{-j,-i}\cdot\frac{-1}{q-q^{-1}}K_lL_{-l,k}=\frac{1}{(q-q^{-1})^2}K_jK_lL_{-j,-i}L_{-l,k}\\
=&\frac{1}{(q-q^{-1})^2}K_jK_l\big(q^{-1}L_{-l,k}L_{-j,-i}+q^{-1}(q-q^{-1})L_{-j,k}L_{-l,-i}\big)\\
=&\frac{1}{(q-q^{-1})^2}\big(q^{-1}K_lL_{-l,k}K_jL_{-j,-i}+q^{-1}(q-q^{-1})K_jL_{-j,k}K_lL_{-l,-i}\big)\\
=&q^{-1}\Xklo\Xije+q^{-1}(q-q^{-1})\Xkjo\Xile,
\end{align*}
where the second and fourth equalities are due to \eqref{KL}.
The remaining cases can be verified similarly and we skip the detail.
\end{proof}

\begin{lem}\label{q-pneo}
Suppose that $1\leq i,j,k,l\leq n$ with $i<j$ and $k>l$.
Then
\begin{enumerate}
\item If $i<j\leq l<k\text{ or }l<k\leq i<j \text{ or }l<i<j<k$, then
$$
\Xije\Xklo=\Xklo\Xije.
$$

\item In other cases, the following formula holds:
$$
\ds\Xije\Xklo=\left\{
\begin{array}{ll}
\Xklo\Xije-(K_{\bj}K_i^{-1}-K_j^{-1}K_{\bi}),&(i=l,j=k),\\
\Xklo\Xije+q(q-q^{-1})K_l^{-1}K_k^{-1}(\Xkjo\Xile-\Xkje\Xilo),&(i<l<k<j), \\
\Xklo\Xije-K_i^{-1}K_j\Xkjo ,&(i=l<j<k),\\
\Xklo\Xije+(q-q^{-1})K_l^{-1}K_{\bj}\Xile+K_l^{-1}K_j^{-1}\Xilo,&(i<l<j=k), \\
\Xklo\Xije+(q-q^{-1})K_l^{-1}K_j\Xkjo\Xile,&(i<l<j<k),\\
\Xklo\Xije-\Xkjo K_i^{-1}K_k^{-1}-(q-q^{-1})\Xkje K_{\bi}K_k^{-1},&(l=i<k<j),\\
\Xklo\Xije+\Xilo K_iK_k^{-1},&(l<i<k=j),\\
\Xklo\Xije-(q-q^{-1})K_iK_k^{-1}\Xkje\Xilo,&(l<i<k<j).
\end{array}
\right.
$$
\end{enumerate}
\end{lem}
\begin{proof}
Assume $1\leq i<j\leq n, 1\leq l<k\leq n$.  By~\eqref{Ol-reln} one can prove
$$
L_{-j,-i}L_{-l,k}=\left\{
\begin{array}{ll}
L_{-l,k}L_{-j,-i}, &\text{ if }i<j<l<k \text{ or }l<k<i<j \text{ or }l<i<j<k,\\
qL_{-l,k}L_{-j,-i},\quad &\text{ if }i<l=j<k, \\
q^{-1}L_{-l,k}L_{-j,-i},\quad &\text{ if }l<i=k<j.
\end{array}
\right.
$$
Then part (1) can be proved by \eqref{KL} and Lemma~\ref{XL}.
Otherwise, again by~\eqref{Ol-reln} it is straightforward to check that the following holds
$$
L_{-j,-i}L_{-l,k}=\left\{
\begin{array}{ll}
L_{-l,k}L_{-j,-i}+(q-q^{-1})(L_{-j,j}L_{-i,-i}-L_{-j,-j}L_{-i,i}),\quad & (i=l,j=k),\\
L_{-l,k}L_{-j,-i}+(q-q^{-1})(L_{-j,k}L_{-l,-i}-L_{-j,-k}L_{-l,i}),\quad &(i<l<k<j), \\
L_{-l,k}L_{-j,-i}+(q-q^{-1})L_{-j,k}L_{-i,-i},\quad &(i=l<j<k), \\
L_{-l,k}L_{-j,-i}+(q-q^{-1})(L_{-j,j}L_{-l,-i}-L_{-j,-j}L_{-l,i}),\quad &(i<l<j=k), \\
L_{-l,k}L_{-j,-i}+(q-q^{-1})L_{-j,k}L_{-l,-i},\quad &(i<l<j<k), \\
L_{-l,k}L_{-j,-i}+(q-q^{-1})(L_{-j,k}L_{-i,-i}-L_{-j,-k}L_{-i,i}),\quad &(l=i<k<j),\\
L_{-l,k}L_{-j,-i}-(q-q^{-1})L_{-j,-j}L_{-l,i},\quad &(l<i<k=j),\\
L_{-l,k}L_{-j,-i}-(q-q^{-1})L_{-j,-k}L_{-l,i},\quad &(l<i<k<j). \\
\end{array}
\right.
$$
This together with Lemma~\ref{XL} and \eqref{KL} gives rise to part (2).
Let us explain in detail the case when $i<l<k<j$.
In this case, by the above formula and Lemma~\ref{XL} we have
\begin{align*}
\Xije\Xklo
=&\frac{-1}{q-q^{-1}}K_jL_{-j,-i}\cdot \frac{-1}{q-q^{-1}}L_{-l,k}K_k^{-1}\\
=&\frac{1}{(q-q^{-1})^2}K_j\big(L_{-l,k}L_{-j,-i}+(q-q^{-1})(L_{-j,k}L_{-l,-i}-L_{-j,-k}L_{-l,i})\big)K_k^{-1}\\
%=&\frac{1}{(q-q^{-1})^2}L_{-l,k}K_k^{-1}K_jL_{-j,-i}+\frac{1}{q-q^{-1}}qK_jK_k^{-1}(L_{-j,k}L_{-l,-i}-L_{-j,-k}L_{-l,i})\\
=&\frac{1}{(q-q^{-1})^2}L_{-l,k}K_k^{-1}K_jL_{-j,-i}\\
&+\frac{1}{q-q^{-1}}qK_k^{-1}K_l^{-1}(K_jL_{-j,k}K_lL_{-l,-i}-K_jL_{-j,-k}K_lL_{-l,i})\\
=&\Xklo\Xije+q(q-q^{-1})K_l^{-1}K_k^{-1}(\Xkjo\Xile-\Xkje\Xilo),
\end{align*}
where the third and fourth equalities are due to \eqref{KL}.
%Other cases can be checked similarly.
\end{proof}

\vspace{.3cm}
\noindent
{\bf Case 3}---Commutation formulas for two odd-odd root vectors $\Xijo\Xklo$ ($i< j$, $k\neq l$). %which are discussed in Lemmas \ref{q-ppoo} and \ref{q-pnoo}.
\begin{lem}\label{q-ppoo}
Let $i,j,k,l\in\{1,2,\ldots,n\}$ satisfy $i<j$ and $k<l$.
Then we have
$$
\displaystyle\Xijo\Xklo=\left\{
\begin{array}{ll}
\ds-\frac{q-q^{-1}}{q+q^{-1}}X_{i,j}^2,& (i=k,j=l),\\
-\Xklo\Xijo,&(i<j<k<l),\\
-\Xklo\Xijo-(q-q^{-1})(\Xkjo\Xilo+\Xkje\Xile),&(i<k<l<j),\\
-q\Xklo\Xijo-q(q-q^{-1})\Xkje\Xile,&(i=k<j<l),\\
-q\Xklo\Xijo+\Xile,& (i<k=j<l),\\
-q\Xklo\Xijo-(q-q^{-1})\Xkje\Xile,&(i<k<j=l),\\
-\Xklo\Xijo-(q-q^{-1})\Xkje\Xile,&(i<k<j<l).
\end{array}
\right.
$$
%\end{subequations}
%\begin{align}
%\Xijo\Xklo&=\left\{
%\begin{array}{ll}
%\ds-\frac{q-q^{-1}}{q+q^{-1}}\Xije^2,\quad\text{ if }i=k,j=l,\\
%-\Xklo\Xijo,\quad\text{ if }i<j<k<l\text{ or }k<l<i<j,\\
%-\Xklo\Xijo-(q-q^{-1})(\Xkjo\Xilo+\Xkje\Xile),\quad\text{ if }i<k<l<j,\\
%-\Xklo\Xijo+(q-q^{-1})(\Xkjo\Xilo-\Xkje\Xile),\quad\text{ if }k<i<j<l,\\
%-q\Xklo\Xijo-q(q-q^{-1})\Xkje\Xile,\quad\text{ if }i=k<j<l,\\
%-\Xklo\Xijo-K_j^{-1}\Xile,\quad\text{ if }i<k=j<l,\\
%-q\Xklo\Xijo-(q-q^{-1})\Xkje\Xile,\quad\text{ if }i<k<j=l,\\
%-\Xklo\Xijo-(q-q^{-1})\Xkje\Xile,\quad\text{ if }i<k<j<l,\\
%-q^{-1}\Xklo\Xijo-q^{-1}(q-q^{-1})\Xkje\Xile,\quad\text{ if }k=i<l<j,\\
%-\Xklo\Xijo-\Xkje K_i^{-1},\quad\text{ if }k<i=l<j,\\
%-q^{-1}\Xklo\Xijo-(q-q^{-1})\Xkje\Xile,\quad\text{ if }k<i<l=j,\\
%-\Xklo\Xijo-(q-q^{-1})\Xkje\Xile,\quad\text{ if }k<i<l<j.
%\end{array}
%\right.
%\end{align}
\end{lem}
\begin{proof}
%Suppose that $i,j,k,l\in\{1,2,\ldots,n\}$ satisfy $i<j$ and $k<l$.
%Then by~\eqref{Ol-reln}, we get
%$$
%L_{-j,i}L_{-l,k}=-L_{-l,k}L_{-j,i}, \text{ if }i<j<k<l \text{ or }k<l<i<j.
%$$
%Otherwise, the following holds:
%$$
%L_{-j,i}L_{-l,k}=\left\{
%\begin{array}{ll}
%\ds-\frac{q-q^{-1}}{q+q^{-1}}L_{-j,-i}^2,\quad &\text{ if } i=k,j=l,\\
%-L_{-l,k}L_{-j,i}-(q-q^{-1})(L_{-j,k}L_{-l,i}+L_{-j,-k}L_{-l,-i}),\quad &\text{ if }i<k<l<j,\\
%-L_{-l,k}L_{-j,i}+(q-q^{-1})(L_{-j,k}L_{-l,i}-L_{-j,-k}L_{-l,-i}),\quad &\text{ if }k<i<j<l,\\
%-q L_{-l,k}L_{-j,i}-q(q-q^{-1})L_{-j,-k}L_{-l,-i},\quad &\text{ if }i=k<j<l,\\
%- L_{-l,k}L_{-j,i}-(q-q^{-1})L_{-j,-j}L_{-l,-i},\quad &\text{ if }i<k=j<l,\\
%-q L_{-l,k}L_{-j,i}-(q-q^{-1})L_{-j,-k}L_{-l,-i},\quad &\text{ if }i<k<j=l,\\
%- L_{-l,k}L_{-j,i}-(q-q^{-1})L_{-j,-k}L_{-l,-i},\quad &\text{ if }i<k<j<l,\\
%-q^{-1}L_{-l,k}L_{-j,i}-q^{-1}(q-q^{-1})L_{-j,-k}L_{-l,-i},\quad &\text{ if }k=i<l<j,\\
%- L_{-l,k}L_{-j,i}-(q-q^{-1})L_{-j,-k}L_{-i,-i},\quad &\text{ if }k<i=l<j,\\
%-q^{-1}L_{-l,k}L_{-j,i}-(q-q^{-1})L_{-j,-k}L_{-l,-i},\quad &\text{ if }k<i<l=j,\\
%- L_{-l,k}L_{-j,i}-(q-q^{-1})L_{-j,-k}L_{-l,-i},\quad &\text{ if }k<i<l<j.
%\end{array}
%\right.
%$$
Suppose that $i,j,k,l\in\{1,2,\ldots,n\}$ satisfy $i<j$ and $k<l$.
As before by~\eqref{Ol-reln}, we get
$$
L_{-j,i}L_{-l,k}=\left\{
\begin{array}{ll}
\ds-\frac{q-q^{-1}}{q+q^{-1}}L_{-j,-i}^2,\quad &(i=k,j=l),\\
-L_{-l,k}L_{-j,i},\quad &(i<j<k<l),\\
-L_{-l,k}L_{-j,i}-(q-q^{-1})(L_{-j,k}L_{-l,i}+L_{-j,-k}L_{-l,-i}),\quad &(i<k<l<j),\\
-q L_{-l,k}L_{-j,i}-q(q-q^{-1})L_{-j,-k}L_{-l,-i},\quad &(i=k<j<l),\\
- L_{-l,k}L_{-j,i}-(q-q^{-1})L_{-j,-j}L_{-l,-i},\quad &(i<k=j<l),\\
-q L_{-l,k}L_{-j,i}-(q-q^{-1})L_{-j,-k}L_{-l,-i},\quad &(i<k<j=l),\\
- L_{-l,k}L_{-j,i}-(q-q^{-1})L_{-j,-k}L_{-l,-i},\quad &(i<k<j<l).
\end{array}
\right.
$$
Then the lemma is proved case-by-case as before.
We will illustrate by checking in detail the case when $i<k=j<l$.
In this case, by the above formulas, \eqref{q-generator} and Lemma~\ref{XL} we have
\begin{align*}
\Xijo\Xklo
=&\frac{-1}{q-q^{-1}}K_jL_{-j,i}\cdot \frac{-1}{q-q^{-1}}K_lL_{-l,k}\\
=&\frac{1}{(q-q^{-1})^2}K_jK_lL_{-j,i}L_{-l,k}\\
=&\frac{1}{(q-q^{-1})^2}K_jK_l\big(-L_{-l,k}L_{-j,i}-(q-q^{-1})L_{-j,-j}L_{-l,-i}\big)\\
=&-\frac{1}{(q-q^{-1})^2}qK_lL_{-l,k}K_jL_{-j,i}-\frac{1}{q-q^{-1}}K_jL_{-j,-j}K_lL_{-l,-i}\\
=&-q\Xklo\Xijo+\Xile,
\end{align*}
where the second and fourth equalities are due to \eqref{KL}.
The remaining cases can be verified similarly, and we omit the detail.
\end{proof}

As before, by solving for $\oX_{k,l}\oX_{i,j}$
in Lemma~\ref{q-ppoo} and then interchanging $(i,j)$ and $(k,l)$, we obtain
$$
\displaystyle\Xijo\Xklo=\left\{
\begin{array}{ll}
-\Xklo\Xijo, &(k<l<i<j),\\
-\Xklo\Xijo+(q-q^{-1})(\Xkjo\Xilo-\Xkje\Xile),&(k<i<j<l),\\
-q^{-1}\Xklo\Xijo-q^{-1}(q-q^{-1})\Xkje\Xile,&(k=i<l<j),\\
-q^{-1}\Xklo\Xijo+q^{-1}\Xkje,&(k<i=l<j),\\
-q^{-1}\Xklo\Xijo-(q-q^{-1})\Xkje\Xile,&(k<i<l=j),\\
-\Xklo\Xijo-(q-q^{-1})\Xkje\Xile,&(k<i<l<j).
\end{array}
\right.
$$
This together with Lemma~\ref{q-ppoo} gives a complete commutation formula for two odd positive root vectors.
\begin{lem}\label{q-pnoo}
The following holds for $i,j,k,l\in\{1,2,\ldots,n\}$ satisfying $i<j$ and $k>l$:
\begin{enumerate}
\item If $i=l,j=k$, then
$$
\Xijo\Xjio=\ds-\Xjio\Xijo+\frac{K_iK_j-K_i^{-1}K_j^{-1}}{q-q^{-1}}
+(q-q^{-1})K_{\bi}K_{\bj}.
$$

%\item If $i<j\leq l<k\text{ or }l<k\leq i<j$, then
%$$
%\Xijo\Xklo=-\Xklo\Xijo.
%$$

%\item In other cases, the following formulas hold:
%$$
%\Xijo\Xklo=\left\{
%\begin{array}{ll}
%-\Xklo\Xijo-q(q-q^{-1})K_l^{-1}K_k^{-1}(\Xkjo\Xilo+\Xkje\Xile),&(i<l<k<j),\\
%-\Xklo\Xijo+q^{-1}(q-q^{-1})(\Xkje\Xile-\Xkjo\Xilo)K_iK_j,&(l<i<j<k),\\
%-\Xklo\Xijo+q^{-1}\Xkje K_iK_j-q^{-1}(q-q^{-1})\Xkjo K_{\bi}K_j,&(i=l<j<k),\\
%-\Xklo\Xijo-(q-q^{-1})K_l^{-1}K_{\bj}\Xilo+K_l^{-1}K_j^{-1}\Xile,&(i<l<j=k),\\
%-\Xklo\Xijo-(q-q^{-1})K_l^{-1}K_j\Xkjo\Xilo,&(i<l<j<k),\\
%-\Xklo\Xijo-(q-q^{-1})\Xkjo K_{\bi}K_k^{-1}+\Xkje K^{-1}_iK_k^{-1},&(l=i<k<j),\\
%-\Xklo\Xijo+q^{-1}K_iK_j \Xile-q^{-1}(q-q^{-1})K_iK_{\bj}\Xilo,&(l<i<k=j),\\
%-\Xklo\Xijo-(q-q^{-1})\Xkjo\Xilo K_iK_k^{-1},&(l<i<k<j).
%\end{array}
%\right.
%$$
\item In other cases, the following formulas hold:
$$
\Xijo\Xklo=\left\{
\begin{array}{ll}
-\Xklo\Xijo,&(i<j\leq l<k),\\
-\Xklo\Xijo-q(q-q^{-1})K_l^{-1}K_k^{-1}(\Xkjo\Xilo+\Xkje\Xile),&(i<l<k<j),\\
-\Xklo\Xijo+q^{-1}\Xkje K_iK_j-q^{-1}(q-q^{-1})\Xkjo K_{\bi}K_j,&(i=l<j<k),\\
-\Xklo\Xijo-(q-q^{-1})K_l^{-1}K_{\bj}\Xilo+K_l^{-1}K_j^{-1}\Xile,&(i<l<j=k),\\
-\Xklo\Xijo-(q-q^{-1})K_l^{-1}K_j\Xkjo\Xilo,&(i<l<j<k).
\end{array}
\right.
$$
\end{enumerate}
\end{lem}
\begin{proof}Suppose that $i,j,k,l\in\{1,2,\ldots,n\}$ satisfy $i<j$ and $k>l$.
If $i=l,j=k$, then by~\eqref{Ol-reln} we have
\begin{equation}\label{Lpn-oo1}
L_{-j,i}L_{-i,j}=-L_{-i,j}L_{-j,i}+(q-q^{-1})(L_{j,j}L_{i,i}-L_{-j,-j}L_{-i,-i})-(q-q^{-1})L_{-j,j}L_{-i,i}.
\end{equation}
%Similarly, if $i<j<l<k \text{ or }l<k<i<j$, then by~\eqref{Ol-reln} we obtain
%$$
%L_{-j,i}L_{-l,k}=-L_{-l,k}L_{-j,i}.
%$$
%Otherwise still by~\eqref{Ol-reln}, we get
%\begin{align*}
%L_{-j,i}L_{-l,k}=\left\{
%\begin{array}{ll}
%-L_{-l,k}L_{-j,i}-(q-q^{-1})(L_{-j,k}L_{-l,i}+L_{-j,-k}L_{-l,-i}),\quad &\text{ if }i<l<k<j,\\
%-L_{-l,k}L_{-j,i}+(q-q^{-1})(L_{j,k}L_{l,i}-L_{-j,k}L_{-l,i}),\quad &\text{ if }l<i<j<k,\\
%-L_{-l,k}L_{-j,i}+(q-q^{-1})(L_{j,k}L_{i,i}-L_{-j,k}L_{-i,i}),\quad &\text{ if }i=l<j<k,\\
%-qL_{-l,k}L_{-j,i},\quad &\text{ if }i<l=j<k,\\
%-L_{-l,k}L_{-j,i}-(q-q^{-1})(L_{-j,j}L_{-l,i}+L_{-j,-j}L_{-l,-i}),\quad &\text{ if }i<l<j=k,\\
%-L_{-l,k}L_{-j,i}-(q-q^{-1})L_{-j,k}L_{-l,i},\quad &\text{ if }i<l<j<k,\\
%-L_{-l,k}L_{-j,i}-(q-q^{-1})(L_{-j,k}L_{-i,i}+L_{-j,-k}L_{-i,-i}),\quad &\text{ if }l=i<k<j,\\
%-q^{-1}L_{-l,k}L_{-j,i},\quad &\text{ if }l<i=k<j,\\
%-L_{-l,k}L_{-j,i}+(q-q^{-1})(L_{j,j}L_{l,i}-L_{-j,j}L_{-l,i}),\quad &\text{ if }l<i<k=j,\\
%-L_{-l,k}L_{-j,i}-(q-q^{-1})L_{-j,k}L_{-l,i},\quad &\text{ if }l<i<k<j.
%\end{array}
%\right.
%\end{align*}
Otherwise still by~\eqref{Ol-reln}, we get
\begin{align*}
L_{-j,i}L_{-l,k}=\left\{
\begin{array}{ll}
-L_{-l,k}L_{-j,i},\quad &(i<j<l<k),\\
-L_{-l,k}L_{-j,i}-(q-q^{-1})(L_{-j,k}L_{-l,i}+L_{-j,-k}L_{-l,-i}),\quad &(i<l<k<j),\\
-L_{-l,k}L_{-j,i}+(q-q^{-1})(L_{j,k}L_{i,i}-L_{-j,k}L_{-i,i}),\quad &(i=l<j<k),\\
-qL_{-l,k}L_{-j,i},\quad &(i<l=j<k),\\
-L_{-l,k}L_{-j,i}-(q-q^{-1})(L_{-j,j}L_{-l,i}+L_{-j,-j}L_{-l,-i}),\quad &(i<l<j=k),\\
-L_{-l,k}L_{-j,i}-(q-q^{-1})L_{-j,k}L_{-l,i},\quad &(i<l<j<k).
\end{array}
\right.
\end{align*}
Then as before the lemma is proved using Lemma~\ref{XL}.
We leave the detail to the reader.
%Let us illustrate by checking in detail the case when $k=i<l<j$.
%In this case, by \eqref{Ljikloo}, \eqref{q-generator},\eqref{KL} and Lemma~\ref{XL} we obtain
%\begin{align*}
%\Xijo\Xlko
%=&\frac{-1}{q-q^{-1}}K_jL_{-j,i}\cdot \frac{-1}{q-q^{-1}}L_{-k,l}K_l^{-1}\\
%=&\frac{1}{(q-q^{-1})^2}K_j\big(-L_{-k,l}L_{-j,i}-(q-q^{-1})(L_{-j,l}L_{-i,i}+L_{-j,-l}L_{-i,-i})\big)K_l^{-1}\\
%=&-\frac{1}{(q-q^{-1})^2}L_{-k,l})K_l^{-1}K_jL_{-j,i}-\frac{1}{q-q^{-1}}(K_jL_{-j,l}L_{-i,i}+K_jL_{-j,-l}L_{-i,-i})K_l^{-1}\\
%=&-\Xlko\Xijo-(q-q^{-1})\Xljo K_{\bi}K_l^{-1}+\Xlje K_iK_l^{-1}.
%\end{align*}
%This means that \eqref{q-pnoo8} holds.
%The remaining cases can be proved similarly and we skip the detail.
\end{proof}

As before, by applying the anti-automorphism $\Omega$ given in \eqref{Omega} and \eqref{Omega-root} to the formulas in
Lemma~\ref{q-pnoo}(2) and interchanging $(i,j)$ and $(k,l)$, the following holds:
$$
\Xijo\Xklo=\left\{
\begin{array}{ll}
-\Xklo\Xijo,&(l<k\leq i<j),\\
-\Xklo\Xijo+q^{-1}(q-q^{-1})(\Xkje\Xile-\Xkjo\Xilo)K_iK_j,&(l<i<j<k),\\
-\Xklo\Xijo-(q-q^{-1})\Xkjo K_{\bi}K_k^{-1}+\Xkje K^{-1}_iK_k^{-1},&(l=i<k<j),\\
-\Xklo\Xijo+q^{-1}K_iK_j \Xile-q^{-1}(q-q^{-1})K_iK_{\bj}\Xilo,&(l<i<k=j),\\
-\Xklo\Xijo-(q-q^{-1})\Xkjo\Xilo K_iK_k^{-1},&(l<i<k<j).
\end{array}
\right.
$$
This together with Lemma~\ref{q-pnoo} gives a complete commutation formula between odd positive root vectors and
odd negative root vectors.

\vspace{.3cm}
\noindent
{\bf Case 4}---Commutation formulas between $\Xije$ and $K_{\bar a}$ and between $\Xijo$  and $K_{\bar a}$ where $1\leq i<j\leq n$ and $1\leq a\leq n$.
\begin{lem}\label{q-XKbar}
Suppose $i,j,a\in\{1,2,\ldots,n\}$ satisfy $i<j$.
Then
\begin{align*}
\Xije K_{\bar {a}}=&\left\{
\begin{array}{ll}
K_{\bar {a}}\Xije,&(a<i \text{ or }a>j),\\
q^{-1}K_{\bar {i}}\Xije-q^{-1}\Xijo K_i^{-1},&(a=i),\\
qK_{\bar {j}}\Xije+q\Xijo K_j^{-1}, &(a=j),\\
K_{\bar{a}}\Xije+q(q-q^{-1})(\oX_{a,j}X_{i,a}-X_{a,j}\oX_{i,a})K_a^{-1},
&(i<a<j),
\end{array}
\right.\\
\Xijo K_{\bar {a}}=&\left\{
\begin{array}{ll}
-K_{\bar {a}}\Xijo,&(a<i \text{ or }a>j),\\
-q^{-1}K_{\bar {i}}\Xijo+q^{-1}\Xije K_i^{-1},&(a=i),\\
-qK_{\bar {j}}\Xijo+q\Xije K_j^{-1}, &(a=j),\\
-K_{\bar{a}}\Xijo-q(q-q^{-1})(\oX_{a,j}\oX_{i,a}+X_{a,j}X_{i,a})K_a^{-1},
&(i<a<j).
\end{array}
\right.
\end{align*}
\begin{proof}
Let $i,j,a\in\{1,2,\ldots,n\}$ and suppose $i<j$.
Then by~\eqref{Ol-reln}, one can deduce the following commutation formula:
\begin{align*}
L_{-j,-i} L_{-a,a}=&\left\{
\begin{array}{ll}
L_{-a,a}L_{-j,-i},&(a<i \text{ or }a>j),\\
q^{-1}L_{-i,i}L_{-j,-i}+q^{-1}(q-q^{-1})L_{-j,i} L_{-i,-i},&(a=i),\\
qL_{-j,j}L_{-j,-i}-q(q-q^{-1})L_{-j,i} L_{-j,-j}, &(a=j),\\
L_{-a,a}L_{-j,-i}+(q-q^{-1})(L_{-j,a}L_{-a,-i}-L_{-j,-a}L_{-a,i}),&(i<a<j),
\end{array}
\right.\\
L_{-j,i} L_{-a,a}=&\left\{
\begin{array}{ll}
-L_{-a,a}L_{-j,i},&(a<i \text{ or }a>j),\\
-q^{-1}L_{-i,i}L_{-j,i}-q^{-1}(q-q^{-1})L_{-j,-i} L_{-i,-i},&(a=i),\\
-qL_{-j,j}L_{-j,i}-q(q-q^{-1})L_{-j,-i} L_{-j,-j}, &(a=j),\\
-L_{-a,a}L_{-j,i}-(q-q^{-1})(L_{-j,a}L_{-a,i}+L_{-j,-a}L_{-a,-i}),&(i<a<j).
\end{array}
\right.
\end{align*}
This together with Lemma~\ref{XL} and \eqref{q-generator}
proves the lemma as before.
\end{proof}
\end{lem}

\section{ Quantum Commutation formulas of higher order}

We now derive the commutation formulas for higher order quantum root vectors. We only need to consider the cases corresponding to Cases 1, 2, and 4 in \S6. This is because $\overline{X}_{i,j}^2\in
U_q(\ev\mfq)$ (see Lemma~\ref{q-ppoo}) for $1\leq i\neq j\leq n$. We need some preparation.

For $m\geq 1$, let
$$
[m]!=[m] [m-1]\cdots[1], \quad \text{ where } [m]=\ds\frac{q^m-q^{-m}}{q-q^{-1}}.
$$
We also use the convention $[0]=[0]!=1.$ For $c\in \Z,t\geq 1$,  set
$$
\begin{bmatrix}
c\\m
\end{bmatrix}
=\frac{[c][c-1]\cdots[c-m+1]}{[m]!},\quad
\begin{bmatrix}
c\\0
\end{bmatrix}
=1.
$$
Generally, for an element $Z$ in an associative $\Q(q)$-algebra $\mc R$ and $m\in\N$, let
$$
Z^{(m)}=\frac{Z^m}{[m]!}.
$$
If $Z$ is invertible, define, for $t\geq 1$ and $c\in\Z$,

\begin{equation}\label{Kt}
\aligned
\begin{bmatrix}
Z;c\\ t
\end{bmatrix}
&=\prod^t_{s=1}\frac{Zq^{c-s+1}-Z^{-1}q^{-c+s-1}}{q^s-q^{-s}},
\quad \text{and }\;\begin{bmatrix}
Z;c\\ 0
\end{bmatrix}=1,
\endaligned
\end{equation}
 %and  write $\ds\begin{bmatrix}Z\\~t\end{bmatrix}=\begin{bmatrix}Z;0\\\quad t\end{bmatrix}.$

\begin{lem}[{cf. \cite[Lemma 1.6]{Lu}}]\label{basic-lemma1}
Let $\mc R$ be an associative algebra over $\Q(q)$ and let $X,Y,Z\in \mc R$.
Then the following holds for any positive integers $m,s$.
\begin{enumerate}

\item If $X,Y,Z$ satisfy
$
XY=YX+Z,\quad XZ=q^{-2}ZX,\quad ZY=q^{-2}YZ,
$
then
\begin{align*}%\label{basic-formula1}
X^{(m)}Y^{(s)}
=\sum_{t=0}^{\min(m,s)}q^{-(m+s)t+\frac{t(3t+1)}{2}}
Y^{(s-t)}Z^{(t)}X^{(m-t)}.
\end{align*}

%\item If $X,Y,Z$ satisfy
%$
%XY=YX+Z,\quad XZ=q^{2}ZX,\quad ZY=q^{2}YZ,
%$
%then
%\begin{align*}%\label{basic-formula2}
%X^{(m)}Y^{(s)}
%=\sum_{t=0}^{\min(m,s)}q^{(m+s)t-\frac{t(3t+1)}{2}}
%Y^{(s-t)}Z^{(t)}X^{(m-t)}.
%\end{align*}

\item
If $X,Y,Z$ satisfy
$
XY=qYX+Z,\quad XZ=q^{-1}ZX,\quad ZY=q^{-1}YZ,
$
then
\begin{align*}%\label{basic-formula3}
X^{(m)}Y^{(s)}
=\sum_{t=0}^{\min(m,s)}q^{(m-t)(s-t)}
Y^{(s-t)}Z^{(t)}X^{(m-t)}.
\end{align*}

\end{enumerate}
\end{lem}
\begin{proof}

Suppose that $X,Y,Z$ satisfy the relations $
XY=YX+Z,\quad XZ=q^{-2}ZX,\quad ZY=q^{-2}YZ,
$ and $m,s$ are positive integers.
If $m\geq1$, then  we have
\begin{equation}\label{*}
\aligned
X^mY&=YX^m+\sum^{m-1}_{t=0}X^tZX^{m-1-t}=YX^m+\sum^{m-1}_{t=0}q^{-2t}ZX^tX^{m-1-t}\\
&=YX^m+\frac{1-q^{-2m}}{1-q^{-2}}ZX^{m-1}=YX^m+q^{-m+1}[m]ZX^{m-1}
\endaligned
\end{equation}
and hence
$
X^{(m)}Y=YX^{(m)}+q^{-m+1}Z X^{(m-1)}.
$
This proves the lemma in the case $s=1$.
By induction on $s$,
\begin{align*}
X^{(m)}Y^{(s+1)}
=&X^{(m)}Y^{(s)}\frac{Y}{[s+1]}=\frac{1}{[s+1]}\sum_{t=0}^{\min(m,s)}q^{-(m+s)t+\frac{t(3t+1)}{2}}
Y^{(s-t)}Z^{(t)}X^{(m-t)}Y\\
=&\frac{1}{[s+1]}\sum_{t=0}^{\min(m,s)}q^{-(m+s)t+\frac{t(3t+1)}{2}}
Y^{(s-t)}Z^{(t)}\Big(YX^{(m-t)}+q^{-m+t+1}ZX^{(m-t-1)}\Big)\\
=&\frac{1}{[s+1]}\sum_{t=0}^{\min(m,s)}q^{-(m+s)t+\frac{t(3t+1)}{2}}
Y^{(s-t)}Yq^{-2t}Z^{(t)}X^{(m-t)}\\
&+\frac{1}{[s+1]}\sum_{t=0}^{\min(m,s)}q^{-(m+s)t+\frac{t(3t+1)}{2}}
Y^{(s-t)}Z^{(t+1)}q^{-m+t+1}[t+1]X^{(m-t-1)}\\
=&\frac{1}{[s+1]}\sum_{t=0}^{\min(m,s)}q^{-(m+s)t+\frac{t(3t+1)}{2}}
Y^{(s+1-t)}q^{-2t}[s+1-t]Z^{(t)}X^{(m-t)}\\
&+\frac{1}{[s+1]}\sum_{t=1}^{\min(m,s)+1}q^{-(m+s)(t-1)+\frac{(t-1)(3t-2)}{2}}
Y^{(s+1-t)}Z^{(t)}q^{-m+t}[t]X^{(m-t)}\\
=&\sum_{t=0}^{\min(m,s+1)}q^{-(m+s+1)t+\frac{t(3t+1)}{2}}
Y^{(s+1-t)}Z^{(t)}X^{(m-t)},
\end{align*}
where the last equality is due to the fact that $q^{-t}[s+1-t]+q^{s+1-t}[t]=[s+1]$.
This proves part (1) of the lemma.
Similarly, it is easy to prove part (2).
\end{proof}

By a proof similar to \eqref{*}, %Lemma~\ref{basic-lemma1}(2),
we have a special case of Lemma~\ref{basic-lemma1}(2) with the condition $ZY=q^{-1}YZ$ dropped.

\begin{cor}\label{basic-cor}
Let $\mc R$ be an associative algebra over $\Q(q)$ and let $X,Y,Z\in \mc R$.
Suppose $m$ is a positive integer.
%\begin{enumerate}
%\item
If $X,Y,Z$ satisfy
$
XY=qYX-Z,~ XZ=q^{-1}ZX.
$
Then
\begin{align*}%\label{basic-5}
X^{(m)}Y
=q^mYX^{(m)}-ZX^{(m-1)}.
\end{align*}

%\item If $X,Y,Z$ satisfy
%$
%XY=q^{-1}YX-Z,~ XZ=qZX.
%$
%Then
%$$
%X^{(m)}Y
%=q^{-m}YX^{(m)}-ZX^{(m-1)}.
%$$
%\end{enumerate}
\end{cor}

\begin{lem}\label{basic-lemma2}
Let $\mc R$ be an associative algebra over $\Q(q)$.
Suppose that the elements $X,Y,I,J,H\in \mc R$ satisfy
$$%\label{basic-reln7}
XY=YX+H-I,~ XH=q^2HX-J,~ XI=q^{-2}IX+q^{-2}J,~ XJ=JX.
$$
Then the following holds for any positive integer $m$:
$$
X^{(m)}Y
=YX^{(m)}+q^{m-1}HX^{(m-1)}
-q^{-m+1}IX^{(m-1)}
-q^{-1}JX^{(m-2)}.
$$
\end{lem}
\begin{proof}
Let $m$ be a positive integer.
Observe that the following holds for $0\leq t\leq m-1$:
\begin{align*}
X^tH&=q^{2t}HX^{t}-\sum^{t-1}_{s=0}X^s\cdot J\cdot(q^2X)^{t-1-s}=q^{2t}HX^{t}-[t]q^{t-1}JX^{t-1},\\
X^tI&=q^{-2t}IX^{t}+\sum^{t-1}_{s=0}X^s\cdot q^{-2}J\cdot (q^{-2}X)^{t-1-s}=q^{-2t}IX^{t}+[t]q^{-t-1}JX^{t-1}.
\end{align*}
Then one can deduce that
\begin{align*}
X^mY
=&YX^m+\sum^{m-1}_{t=0}X^t(H-I)X^{m-1-t}
=YX^m+\sum^{m-1}_{t=0}X^tHX^{m-1-t}-\sum^{m-1}_{t=0}X^tIX^{m-1-t}\\
%=&YX^m+\sum^{m-1}_{t=0}\Big(q^{2t}HX^{t}-\sum^{t-1}_{s=0}X^s\cdot J\cdot(q^2X)^{t-1-s}\Big)X^{m-1-t}\\
%-&\sum^{m-1}_{t=0}\Big(q^{-2t}IX^{t}+\sum^{t-1}_{s=0}X^s\cdot q^{-2}J\cdot (q^{-2}X)^{t-1-s}\Big)X^{m-1-t}\\
=&YX^m+\sum^{m-1}_{t=0}\Big(q^{2t}HX^{m-1}-[t]q^{t-1}JX^{m-2}-q^{-2t}IX^{m-1}-[t]q^{-t-1}JX^{m-2}\Big)\\
%-&\sum^{m-1}_{t=0}\Big(q^{-2t}IX^{m-1}+[t]q^{-t-1}JX^{m-2}\Big)\\
%=&YX^m+q^{m-1}[m]HX^{m-1}-q^{-m+1}[m]IX^{m-1}-\sum^{m-1}_{t=0}[t](q^t+q^{-t})q^{-1}JX^{m-2}\\
=&YX^m+q^{m-1}[m]HX^{m-1}-q^{-m+1}[m]IX^{m-1}
-q^{-1}[m][m-1]JX^{m-2}.
\end{align*}
Hence the lemma is verified.
\end{proof}

 We now apply the formulas in the previous lemmas to derive the commutation formulas of higher order. First, we deal with the even-even  case for commuting $\Xijem\Xkles$.
\begin{prop}\label{q-div-ppee}
Assume that $i,j,k,l\in\{1,2,\ldots,n\}$ satisfy $i<j$ and $k<l$.
Then the following holds for positive integers $m,s$.
\begin{enumerate}
\item If $i<j<k<l$ or $i<k<l<j$, then
\begin{align*}
\Xijem\Xkles=\Xkles\Xijem.
\end{align*}

\item If $i<k<j=l\text{ or }i=k<j<l$, then
\begin{align*}
\Xijem\Xkles=q^{-ms}\Xkles\Xijem.
\end{align*}

\item If $i<k=j<l$, then
\begin{align*}
\Xijem\Xkles
=&q^{ms}\Xkles\Xijem+\sum_{t=1}^{\min(m,s)}q^{(m-t)(s-t)}
X_{k,l}^{(s-t)}X^{(t)}_{i,l} X_{i,j}^{(m-t)}.
\end{align*}

\item If $i<k<j<l$, then
\begin{align*}
\Xijem\Xkles
=&\Xkles\Xijem\notag\\
&+\sum_{t=1}^{\min(m,s)}(-1)^t[t]!(q-q^{-1})^tq^{-(m+s)t+\frac{t(3t+1)}{2}}
X_{k,l}^{(s-t)}X^{(t)}_{k,j}\Xilet X_{i,j}^{(m-t)}.
\end{align*}
\end{enumerate}
\end{prop}
\begin{proof}
Let $i,j,k,l\in\{1,2,\ldots,n\}$.
Clearly, parts (1) and (2) follow from
the first two formulas in~Lemma~\ref{q-ppee}, respectively.

Assume that $i<k=j<l$ and let
$
X=\Xije,~ Y=\Xkle,~ Z=\Xile.
$
Then by Lemma~\ref{q-ppee}, we have
$
XY=qYX+Z,~ XZ=q^{-1}ZX,~ ZY=q^{-1}YZ.
$
Thus by Lemma~\ref{basic-lemma1}(2), one can deduce that part (3) is proved.

Finally, if $i<k<j<l$, then by Lemma~\ref{q-ppee}
the elements $X=\Xije,~ Y=\Xkle,~ Z=-(q-q^{-1})\Xkje\Xile$
satisfy $XY=YX+Z, XZ=q^{-2}ZX,~ ZY=q^{-2}YZ$ and hence part (4) of the proposition follows
from Lemma~\ref{basic-lemma1}(1).
\end{proof}

\begin{rem} \label{nnee}
As before,
we can obtain another set of formulas for the commutation of two even divided powers $X^{(m)}_{i,j}$ and $X^{(s)}_{k,l}$
by solving for $X^{(s)}_{k,l}X^{(m)}_{i,j}$ in Proposition~\ref{q-div-ppee} and then interchanging $(i,j)$ and $(k,l)$.
These newly obtained formulas together with Proposition~\ref{q-div-ppee} exhaust
the possibilities for a commutation of two even divided powers $X^{(m)}_{i,j}$ and $X^{(s)}_{k,l}$
associated to positive root vectors.
%Moreover, by applying the anti-automorphism $\Omega$,
%one can obtain a complete set of commutation formulas
%for two even divided powers $X^{(m)}_{i,j}$ and $X^{(s)}_{k,l}$ associated to negative root vectors.
\end{rem}

\begin{prop}\label{q-div-pnee}
Assume that $i,j,k,l\in\{1,2,\ldots,n\}$ satisfy $i<j$ and $k>l$.
Let $m,s$ be positive integers.
Then
\begin{enumerate}
\item If $i=l$ and $j=k$, then
\begin{align*}
\Xijem\Xkles
=&\Xkles\Xijem+\sum^{\min(m,s)}_{t=1}X_{k,l}^{(s-t)}
\begin{bmatrix}
K_iK_j^{-1};2t-m-s\\\quad t
\end{bmatrix}
 X^{(m-t)}_{i,j}.
\end{align*}

\item If $i<j\leq l<k$ or $i<l<k<j$, then
\begin{align*}
\Xijem\Xkles=\Xkles\Xijem.
\end{align*}

\item If $i=l<j<k$, then
\begin{align*}
\Xijem\Xkles
=&\Xkles\Xijem+\sum^{\min(m,s)}_{t=1}(-1)^{t}q^{(m+s-t-1)t}
X^{(s-t)}_{k,l}K_i^{-t}K_j^tX^{(t)}_{k,j}X^{(m-t)}_{i,j}.
\end{align*}

\item If $i<l<j=k$, then
\begin{align*}
\Xijem\Xkles
=&\Xkles\Xijem+\sum^{\min(m,s)}_{t=1}q^{(m+s-2t)t}
X^{(s-t)}_{k,l}K_l^{-t}K_j^{t}X^{(t)}_{i,l}X^{(m-t)}_{i,j}.
\end{align*}

\item If $i<l<j<k$, then
\begin{align*}
\Xijem\Xkles
=&\Xkles\Xijem+\sum^{\min(m,s)}_{t=1}q^{(m+s)t-\frac{t(3t+1)}{2}}(q-q^{-1})^t[t]!
X^{(s-t)}_{k,l}K_l^{-t}K_j^tX^{(t)}_{k,j}X^{(t)}_{i,l}X^{(m-t)}_{i,j}.
\end{align*}

\end{enumerate}
\end{prop}
\begin{proof}
It is easy to see that part (2) follows from Lemma~\ref{q-pnee}.

If $i=l<j<k$,
then by ~\eqref{KX}, Lemmas~\ref{q-pnee} and \ref{q-ppee},
the elements $X=\Xije, Y=\Xkle, Z=-K_i^{-1}K_j\Xkje $ satisfy $XY=YX+Z,XZ=q^2ZX$.
Furthermore, by Lemma~\ref{q-ppee}, we have $X_{l,k}X_{j,k}=q^{-1}X_{j,k}X_{l,k}$
and hence by applying $\Omega$ in \eqref{Omega} we obtain $X_{k,j}X_{k,l}=qX_{k,l}X_{k,j}$.
This leads to $ZY=q^2YZ$ by \eqref{KX}.
Meanwhile again by \eqref{KX} we have $\Xkje K_i^{-1}K_j=qK_i^{-1}K_j\Xkje$,
which implies $Z^t=(-1)^tq^{\frac{t(t-1)}{2}}K_i^{-t}K_j^tX^t_{k,j}$.
Therefore, by Lemma~\ref{basic-lemma1}(1) with $q^{-1}$ being replaced by $q$, part (3) holds.
Similarly, the formulas in parts (4) and (5) can be checked by Lemma~\ref{q-pnee}, \eqref{KX} and ~\eqref{q-ppee-extra}
and Lemma~\ref{basic-lemma1}(1).

%If $i<k<j=l$, then by Lemma~\ref{q-pnee}, \eqref{KX} and ~\eqref{q-ppee-extra},
%the elements $X=\Xije, Y=\Xlke, Z=K_k^{-1}K_j\Xike $ satisfy the relation~\eqref{basic-reln2}.
%Hence by~\eqref{basic-formula2}, part (4) is proved.

%Similarly, if $i<k<j<l$, then again by Lemma~\ref{q-pnee}, \eqref{KX} and~\eqref{q-ppee-extra}, one can check that
%the elements $X=\Xije, Y=\Xlke, Z=(q-q^{-1})K_k^{-1}K_j\Xlje\Xike $ satisfy the relation~\eqref{basic-reln2}.
%Hence by~\eqref{basic-formula2}, part (5) is verified since $X_{l,j}X_{i,k}=X_{i,k}X_{l,j}$
%in this case by Lemma~\eqref{q-pnee}.

Now it remains to prove part (1).
Assume $l=i,k=j$.
By~\eqref{KX} and~\eqref{Kt}, we have
$$
\Xije\cdot
\begin{bmatrix}
K_iK_j^{-1};c\\~t
\end{bmatrix}
=\begin{bmatrix}
K_iK_j^{-1};c-2\\ t
\end{bmatrix}\cdot\Xije,\quad
\Xjie\cdot\begin{bmatrix}
K_iK_j^{-1};c\\~t
\end{bmatrix}
=\begin{bmatrix}
K_iK_j^{-1};c+2\\ t
\end{bmatrix}\cdot\Xjie.
$$
Then by Lemma~\ref{q-pnee}, part (1) of the proposition can be directly checked by induction on $s$, which is similar to the classical case.
\end{proof}

\begin{rem} \label{pnee}
By applying the anti-automorphism to the formulas in Proposition~~\ref{q-div-pnee}, we can obtain
another set of commutation formulas for the divided powers $\Xijem$ and $\Xkles$.
This together with Proposition~\ref{q-div-pnee} exhaust the possibilities for a commutation relation between $\Xijem$ and $\Xkles$.
\end{rem}

Next, we derive the commutation formulas for $\Xijem\Xklo$.
\begin{prop}\label{q-div-ppeo}
Assume that $i,j,k,l\in\{1,2,\ldots,n\}$ satisfy $i<j$ and $k<l$.
Let $m$ be a positive integer.
Then
\begin{enumerate}
\item If $i=k,j=l$, or $i<j<k<l$ or $k<l<i<j$ or $k<i<j<l$,
then
$$
\Xijem\Xklo=\Xklo\Xijem.
$$

\item If $i<k<l<j$, then
\begin{align}
\Xijem\Xklo
=&\Xklo\Xijem+q^{m-1}(q-q^{-1})\Xkjo\Xile X^{(m-1)}_{i,j}\notag\\
&-q^{-m+1}(q-q^{-1})\Xkje\Xilo X^{(m-1)}_{i,j}\notag\\
&-q^{-1}(q-q^{-1})^2\Xkje\Xijo\Xile X^{(m-2)}_{i,j}.\notag
\end{align}

\item If $i=k<j<l$, then
\begin{align*}%\label{q-div-ppeo4}
\Xijem\Xklo=q^{-m}\Xklo\Xijem.
\end{align*}

\item If $i<k=j<l$, then
\begin{align*}%\label{q-div-ppeo5}
\Xijem\Xklo=q^{m}\Xklo\Xijem+\Xilo X^{(m-1)}_{i,j}.
\end{align*}

\item If $i<k<j=l$, then
\begin{align*}%\label{q-div-ppeo6}
\Xijem\Xklo=q^m\Xklo\Xijem-(q-q^{-1})\Xkje \Xilo X^{(m-1)}_{i,j}.
\end{align*}

\item If $i<k<j<l$,
then
\begin{align*}%\label{q-div-ppeo7}
\Xijem\Xklo=\Xklo\Xijem-q^{-m+1}(q-q^{-1})\Xkje\Xilo X^{(m-1)}_{i,j}.
\end{align*}

\item If $k=i<l<j$, then
\begin{align*}%\label{q-div-ppeo8}
\Xijem\Xklo=q^{-m}\Xklo\Xijem+q^{-1}(q-q^{-1})\Xkjo \Xile X^{(m-1)}_{i,j}.
\end{align*}

\item If $k<i=l<j$,
then
\begin{align*}%\label{q-div-ppeo9}
\Xijem\Xklo=q^{-m}\Xklo\Xijem-q^{-1}\Xkjo  X^{(m-1)}_{i,j}.
\end{align*}

\item If $k<i<l=j$,
then
\begin{align*}%\label{q-div-ppeo10}
\Xijem\Xklo=q^m\Xklo\Xijem.
\end{align*}

\item If $k<i<l<j$, then
\begin{align*}%\label{q-div-ppeo11}
\Xijem\Xklo=\Xklo\Xijem+q^{m-1}(q-q^{-1})\Xkjo \Xile X^{(m-1)}_{i,j}.
\end{align*}

\end{enumerate}
\end{prop}
\begin{proof}
Clearly, the equalities in parts (1), (3) and (9) can be checked directly
by taking advantage of the corresponding formulas in Lemma~\ref{q-ppeo}, respectively.

If $i<k<l<j$, we set
\begin{align*}
X=&\Xije,~ Y=\Xklo,~ H=(q-q^{-1})\Xkjo\Xile,~ I=(q-q^{-1})\Xkje\Xilo,\\
J=&(q-q^{-1})^2\Xkje\Xijo\Xile.
\end{align*}
Then firstly, by the formula in the first case in Lemma~\ref{q-ppeo}(2) we have $XY=YX+H-I$.
By the formula in the fourth case in Lemma~\ref{q-ppeo}(2) and \eqref{q-ppee-extra} we see that
$
XH=q^2HX-J.
$
By the formula in the sixth case in Lemma~\ref{q-ppeo}(2) and Lemma~\ref{q-ppee}, one can check that
$
XI=q^{-2}IX+q^{-2}J.
$
Furthermore, by Lemma~\ref{q-ppee}, \eqref{q-ppee-extra}~ and~ Lemma~\ref{q-ppeo}(1)
we obtain
$
XJ=JX.
$
Putting together, part (2) is proved by using Lemma~\ref{basic-lemma2}.

If $i<k=j<l$, then by the formulas in the second and third cases in Lemma~\ref{q-ppeo} the elements
$
X=\Xije,\quad Y=\Xklo,\quad Z=-\Xilo
$
satisfy $XY=qYX-Z, XZ=q^{-1}ZX$.
Hence part (4) is verified by Corollary~\ref{basic-cor}.

If $i<k<j=l$, then  the elements
$
X=\Xije,~ Y=\Xklo,~ Z=(q-q^{-1})\Xkje\Xilo
$
satisfy $XY=qYX-Z,~ XZ=q^{-1}ZX$ by the formula in the fourth case in Lemmas~\ref{q-ppeo}(2), ~\ref{q-ppeo}(1) and ~\ref{q-ppee}.
Therefore part (5) follows from Corollary~\ref{basic-cor}.

Similarly, the remaining parts can be proved case-by case using Lemmas~\ref{q-ppeo}, \ref{q-ppee} and \ref{basic-lemma1}, and Corollary~\ref{basic-cor}.
To save some space, we omit the detail.

%If $i<k<j<l$, then by~\eqref{q-ppeo7}, \eqref{q-ppeo4} and~Lemma~\ref{q-ppee}(2) the elements
%$$
%X=\Xije,~ Y=\Xklo,~ Z=-(q-q^{-1})\Xilo\Xkje
%$$
%satisfy the relation~\eqref{basic-reln1} and hence \eqref{q-div-ppeo7} follows from~\eqref{basic-formula1}.

%If $k=i<l<j$, then by~\eqref{q-ppeo8}, \eqref{q-ppeo1} and \eqref{q-ppee-extra} the elements
%$$
%X=\Xije,\quad Y=\Xklo,\quad Z=-q^{-1}(q-q^{-1})\Xkjo\Xile
%$$
%satisfy the relation~\eqref{basic-reln6}, which leads to \eqref{q-div-ppeo8} by~\eqref{basic-6}.

%If $k<i=l<j$, then the relation~\eqref{basic-reln6} is achieved if we set
%$$
%X=\Xije,\quad Y=\Xklo,\quad Z=q^{-1}\Xkjo
%$$
%by~\eqref{q-ppeo9} and \eqref{q-ppeo10}.
%This means that \eqref{q-div-ppeo9} holds by~\eqref{basic-6}.

%Finally if $k<i<l<j$, then by~\eqref{q-ppeo11}, \eqref{q-ppeo10} and~\eqref{q-ppee-extra} the elements
%$$
%X=\Xije,\quad Y=\Xklo,\quad Z=(q-q^{-1})\Xkjo\Xile
%$$
%satisfy the relation~\eqref{basic-reln2} and hence \eqref{q-div-ppeo11} follows from~\eqref{basic-formula2}.
\end{proof}

%\begin{rem}\label{nneo}
%As before, by applying the anti-automorphism $\Omega$ to the formulas in Proposition~\ref{q-div-ppeo},
%we will obtain a complete set of commutation formulas involving a even divided power $X^{(m)}_{j,i}$
%and a odd negative root vector $\oX_{l,k}$.
%\end{rem}

\begin{prop}\label{q-div-pneo}
Assume that $i,j,k,l\in\{1,2,\ldots,n\}$ satisfy $i<j$ and $k>l$.
Let $m$ be a positive integer.
Then
\begin{enumerate}
\item If $i<j\leq l<k$ or $l<k\leq i<j$ or $l<i<j<k$,
then
\begin{align*}%\label{q-div-pneo2}
\Xijem \Xklo
=\Xklo\Xijem.
\end{align*}

\item If $i=l$ and $j=k$, then
\begin{align}
\Xijem \Xjio
=&\Xjio\Xijem-q^{m-1}K_{\bj}K_i^{-1}X^{(m-1)}_{i,j}\notag\\
&+q^{-m+1}K_j^{-1}K_{\bi}X^{(m-1)}_{i,j}
-\Xijo K_j^{-1}K_i^{-1}X^{(m-2)}_{i,j}.\notag
\end{align}

\item If $i<l<k<j$,
then
\begin{align*}
\Xijem \Xklo
=&\Xklo\Xijem+q^{m}(q-q^{-1})K_l^{-1}K_k^{-1}\Xkjo\Xile X^{(m-1)}_{i,j}\notag\\
&-q^{-m+2}(q-q^{-1})K_l^{-1}K_k^{-1}\Xkje\Xilo X^{(m-1)}_{i,j}\notag\\
&-(q-q^{-1})^2K_l^{-1}K_k^{-1}\Xkje\Xijo\Xile X^{(m-2)}_{i,j}.%\label{q-div-pneo3}
\end{align*}

\item If $i=l<j<k$, then
\begin{align*}%\label{q-div-pneo4}
\Xijem \Xklo
=&\Xklo\Xijem-q^{m-1}K_i^{-1}K_j\Xkjo  X^{(m-1)}_{i,j}.
\end{align*}

\item If $i<l<j=k$, then
\begin{align*}
\Xijem \Xklo
=&\Xklo\Xijem+q^{m-1}(q-q^{-1})K_l^{-1}K_{\bj}\Xile X^{(m-1)}_{i,j}\notag\\
&+q^{-m+1}K_l^{-1}K_j^{-1}\Xilo X^{(m-1)}_{i,j}\notag\\
&+q^{-1}(q-q^{-1})K_l^{-1}K_j^{-1}\Xijo \Xile X^{(m-2)}_{i,j}.%\label{q-div-pneo5}
\end{align*}

\item If $i<l<j<k$,
then
\begin{align*}%\label{q-div-pneo6}
\Xijem \Xklo
=\Xklo\Xijem+q^{m-1}(q-q^{-1})K_l^{-1}K_j\Xkjo\Xile X^{(m-1)}_{i,j}
\end{align*}

\item If $l=i<k<j$, then
\begin{align*}
\Xijem \Xklo
=&\Xklo\Xijem-q^{m-1}\Xkjo K_i^{-1}K_k^{-1} X^{(m-1)}_{i,j}\notag\\
&-q^{-m+1}(q-q^{-1})\Xkje K_{\bi}K_k^{-1} X^{(m-1)}_{i,j}\notag\\
&+q^{-1}(q-q^{-1})\Xkje\Xijo K_i^{-1}K_k^{-1} X^{(m-2)}_{i,j}.%\label{q-div-pneo7}
\end{align*}

\item If $l<i<k=j$, then
\begin{align*}%\label{q-div-pneo8}
\Xijem \Xklo
=\Xklo\Xijem+q^{-m+1}\Xilo K_iK_j^{-1} X^{(m-1)}_{i,j}.
\end{align*}

\item If $l<i<k<j$, then
\begin{align*}%\label{q-div-pneo9}
\Xijem \Xklo
=\Xklo\Xijem-q^{-m+1}(q-q^{-1})K_iK_k^{-1}\Xkje \Xilo X^{(m-1)}_{i,j}.
\end{align*}

\end{enumerate}
\end{prop}
\begin{proof}
Clearly, part (1)
follows directly from Lemma~\ref{q-pneo}(1).

If $i=l,j=k$, we set
\begin{align*}
X=&\Xije,~ Y=\Xjio,~ H=-K_{\bj}K_i^{-1},~ I=-K_j^{-1}K_{\bi},\\
J=&q\Xijo K_j^{-1}K_i^{-1}=K_j^{-1}\Xijo K_i^{-1}.
\end{align*}
Then we have $XY=YX+H-I, XH=q^2HX-J, XI=q^{-2}IX+q^{-2}J$, $XJ=JX$ by ~\eqref{KX} and Lemmas~\ref{q-pneo}(1), \ref{q-ppeo}(1) and~\ref{q-XKbar}.
Hence part (2) follows from Lemma~\ref{basic-lemma2}.

If  $i<l<k<j$,
we set
\begin{align*}
X=&\Xije,\quad Y=\Xklo,\quad H=q(q-q^{-1})K_l^{-1}K_k^{-1}\Xkjo\Xile,\\
I=&q(q-q^{-1})K_l^{-1}K_k^{-1}\Xkje\Xilo,\quad
J=q(q-q^{-1})^2K_l^{-1}K_k^{-1}\Xkje\Xijo\Xile.
\end{align*}
Then clearly by Lemma~\ref{q-pneo}(2), we have
$XY=YX+H-I. $
By the formula in the fourth case in Lemma~\ref{q-ppeo}(2) and \eqref{q-ppee-extra}, one can check $XH=q^2HX-J$.
Meanwhile by the formulas in the sixth case in Lemma~\ref{q-ppeo}(2) and~Lemma~\ref{q-ppee}, we see $XI=q^{-2}IX+q^{-2}J$.
Finally, by Lemma~\ref{q-ppeo}(1), \eqref{q-ppee-extra} and Lemma~\ref{q-ppee}, $XJ=JX$ holds.
Therefore part (3) follows from Lemma~\ref{basic-lemma2}.

If $i=l<j<k$, then by the formula in the third case in Lemma~\ref{q-pneo}(2), Lemma~\ref{q-pneo}(1) and~\eqref{KX}
the following elements
$
X=\Xije,~Y=\Xklo,~Z=-K_i^{-1}K_j\Xkjo
$
satisfy $XY=YX+Z,~ XZ=q^2ZX$ and hence part (4) can be verified by a proof similar to \eqref{*} with $q^{-1}$ being replaced by $q$.

If $i<l<j=k$,
we set
\begin{align*}
X=&\Xije,\quad Y=\Xklo,\quad H=(q-q^{-1})K_l^{-1}K_{\bj}\Xile ,\quad I=-K_l^{-1}K_j^{-1}\Xilo\\
J=&-(q-q^{-1})K_l^{-1}K_j^{-1}\Xijo \Xile.
\end{align*}
Then by the formula in fourth case in Lemma~\ref{q-pneo}(2) we see $XY=YX+H-I.$
By~\eqref{q-ppee-extra}, \eqref{KX} and Lemma~\ref{q-XKbar}, we get $XH=q^2HX-J$ and
by the formula in the sixth case in Lemma~\ref{q-ppeo}(2) and \eqref{KX}, we obtain $XI=q^{-2}IX+q^{-2}J$.
Finally, by Lemma~\ref{q-ppeo}(1), \eqref{KX} and \eqref{q-ppee-extra} one can check that $XJ=JX$ holds.
Putting together, we obtain that part (5) follows from Lemma~\ref{basic-lemma2}.

If $i<l<j<k$, then by the formula in the fifth case in Lemma~\ref{q-pneo}(2), Lemma~\ref{q-pneo}(1),
and the equalities \eqref{KX} and~\eqref{q-ppee-extra} the  elements
$
X=\Xije,~ Y=\Xklo,~ Z=(q-q^{-1})K_l^{-1}K_j\Xkjo\Xile
$
satisfy  $XY=YX+Z,~ XZ=q^2ZX$ and hence part (6) follows from a proof similar to \eqref{*} with $q^{-1}$ being replaced by $q$.

If $l=i<k<j$, we let
\begin{align*}
X=&\Xije,\quad Y=\Xklo,\quad H=-\Xkjo K_i^{-1}K_k^{-1},\quad I=(q-q^{-1})\Xkje K_{\bi}K_k^{-1}\\
J=&-(q-q^{-1}) \Xkje\Xijo K_i^{-1}K_k^{-1}.
\end{align*}
By the formula in the sixth case in Lemma~\ref{q-pneo}(2), we have $XY=YX+H-I.$
Meanwhile, by the formula in the fourth case in Lemma~\ref{q-ppeo}(2) and~\eqref{KX} one can check $XH=q^2HX-J$ and
by Lemma~\ref{q-ppee} and Lemma~\ref{q-XKbar} we obtain $XI=q^{-2}IX+q^{-2}J$.
Finally, by Lemma~\ref{q-ppeo}(1), Lemma~\ref{q-ppee} and \eqref{KX}, we see $XJ=JX$.
Hence part (7) follows from Lemma~\ref{basic-lemma2}.

If $l<i<k=j$, then by the formula in the seventh case in Lemma~\ref{q-pneo}(2), ~\eqref{KX} and Lemma~\ref{q-pneo}(1) we find that the elements
$
X=\Xije,~ Y=\Xklo,~ Z=\Xilo K_iK_j^{-1}
$
satisfy $XY=YX+Z,~ XZ=q^{-2}ZX$. Then by a proof similar to \eqref{*}, part (8) holds.
Similarly, the last part can be verified by Lemma~\ref{q-pneo}(2), Lemma~\ref{q-pneo}(1), \eqref{KX} and Lemma~\ref{q-ppee}.

%Finally, if $l<i<k<j$, then by the formula in the eighth case in Lemma~\ref{q-pneo}(2), Lemma~\ref{q-pneo}(1), \eqref{KX} and Lemma~\ref{q-ppee} we find that the elements
%$
%X=\Xije,~ Y=\Xklo,~ Z=-(q-q^{-1})K_iK_k^{-1}\Xkje\Xilo
%$
%satisfy $XY=YX+Z,~ XZ=q^{-2}ZX$ and hence the last part follows from a proof similar to \eqref{*}.

\end{proof}

Finally, it remains to derive the commutation formulas between $\Xijem$ and $K_{\bar a}$.
\begin{prop}\label{q-div-XKbar}The following holds for  $i,j,a\in\{1,2,\ldots, n\}$ satisfying $i<j$:
\begin{enumerate}
\item If $a<i$ or $a>j$, then
$$
\Xijem K_{\bar a}=K_{\bar a}\Xijem.
$$

\item If $a=i$, then
$$
\Xijem K_{\bar i}=q^{-m}K_{\bar i}\Xijem-K_i^{-1}\Xijo X_{i,j}^{(m-1)}.
$$

\item If $a=j$, then
$$
\Xijem K_{\bar j}=q^{m}K_{\bar j}\Xijem+K_j^{-1}\Xijo X_{i,j}^{(m-1)}.
$$

\item If $i<a<j$, then
\begin{align*}
\Xijem K_{\bar a}=
&K_{\bar a}\Xijem+q^{m}(q-q^{-1})\oX_{a,j}X_{i,a}K_a^{-1}X_{i,j}^{(m-1)}-q^{-m+2}(q-q^{-1})X_{a,j}X_{i,a}K_a^{-1}X_{i,j}^{(m-1)}\\
&-(q-q^{-1})^2X_{a,j}\oX_{i,j}X_{i,a}K_a^{-1}X_{i,j}^{(m-2)}.
\end{align*}
\end{enumerate}
\end{prop}
\begin{proof}
As before, the proposition can be verified case-by-case using Lemma~\ref{q-XKbar}, Lemma~\ref{q-ppee}, \eqref{q-ppee-extra},
Lemma~\ref{q-ppeo}, Corollary~\ref{basic-cor} and  Lemma~\ref{basic-lemma2}.
We will check the case when $i<a<j$.
In this case, we set
\begin{align*}
X=&\Xije, ~ Y=K_{\bar a},~ H=q(q-q^{-1})\oX_{a,j}X_{i,a}K_a^{-1},~ I=q(q-q^{-1})X_{a,j}\oX_{i,a}K_a^{-1},\\
J=&q(q-q^{-1})^2X_{a,j}\oX_{i,j}X_{i,a}K_a^{-1}.
\end{align*}
Then by Lemma~\ref{q-XKbar}, we have $XY=YX+H-I$. Moreover, by the formula in the fourth case in Lemma~\ref{q-ppeo}(2) and \eqref{q-ppee-extra} we have
$XH=q^2HX-J$. Meanwhile, by the formula in the sixth case in Lemma~\ref{q-ppeo}(2) and Lemma~\ref{q-ppee}, one can deduce that
$XI=q^{-2}IX+q^{-2}J$. Finally, by Lemma~\ref{q-ppeo}(1), Lemma~\ref{q-ppee} and \eqref{q-ppee-extra} we obtain $XJ=JX$.
Putting together, part (4) is proved by Lemma~\ref{basic-lemma2}.
Similarly, the other parts can be verified.
\end{proof}

\begin{rem}\label{npeo} As before, by applying the anti-automorphism $\Omega$ to the formulas in Propositions~ \ref{q-div-ppeo},~\ref{q-div-pneo}~and~\ref{q-div-XKbar},
we will obtain a complete set of commutation formulas involving an even divided power $X^{(m)}_{i,j}$
and  an odd root vector $\oX_{k,l}$ or the element $K_{\bar{a}}$ for $i>j, k\neq l, 1\leq a\leq n$.
\end{rem}

\section{Lusztig type form for $\qUq$}
We now present two applications of the commutation formulas.
The first application is the existence of an integral form (or Lusztig form) for
$\qUq$. Recall from \eqref{Kt} the notation $\begin{bmatrix}
K_i\\t
\end{bmatrix}=\begin{bmatrix}
K_i;0\\t
\end{bmatrix}$.

Let $\mcZ=\Z[q,q^{-1}]$. Define $\qUZ$ to be the $\mcZ$-subsuperalgebra of $\qUq$ generated by
\begin{align*}
K_i^{\pm1},~
\begin{bmatrix}
K_i\\t
\end{bmatrix},~ E_j^{(m)},~ F^{(m)}_j,~ K_{\bi},~ E_{\bj},~ F_{\bj}\quad (1\leq i\leq n, 1\leq j\leq n-1,t,m\in\N).
\end{align*}
By Proposition~\ref{q-div-ppee}(3),  we obtain
$$
X^{(m)}_{i,j}=X^{(m)}_{i,k}X^{(m)}_{k,j}-q^{m^2}X^{(m)}_{k,j}X^{(m)}_{i,k}-\sum^{m-1}_{t=1}q^{(m-t)^2}X^{(m-t)}_{k,j}X^{(t)}_{i,j}X^{(m-t)}_{i,k}
$$
for $m\geq 1$ and $1\leq i<j\leq n$ such that $|j-i|>1$, where $k$ is strictly between $i$ and $j$.
Hence by induction on $|j-i|$ and $m$, we obtain $X^{(m)}_{i,j}\in U_{q,\mcZ}$ for $1\leq i<j\leq n$ as the classical case (cf. \cite[Proposition 2.17]{Lu}).
Then using \eqref{Omega-root}, one can deduce that $X^{(m)}_{i,j}\in U_{q,\mcZ}$ for $1\leq j<i\leq n$.
Then we will also consider the set of generators involving $K_i^{\pm1},~
\begin{bmatrix}
K_i\\t
\end{bmatrix} (1\leq i\leq n, t\in\N)$ and the set (cf. Remark~\ref{degree}):
\begin{equation}
\mc G_q=\{ \Xijem,~ \Xijo,~ K_{\bi}\mid 1\leq i\neq j\leq n,m\in\N\}
\end{equation}
Let $U^0_{q,\mcZ}$ be the  $\mcZ$-subsuperalgebra of $\qUq$ generated by $K_i^{\pm1},~
\begin{bmatrix}
K_i\\t
\end{bmatrix},
K_{\bi}
$ with $1\leq i\leq n, t\in\N$.
Denote by $U^+_{q,\mcZ}$ (resp. $U^-_{q,\mcZ}$) the $\mcZ$-subsuperalgebra of $\qUq$ generated by
$E^{(m)}_j, E_{\bj}$  (resp. by $F^{(m)}_j, F_{\bj}$) with $1\leq j\leq n-1, m\in\N$.

\begin{lem}\label{Kbarsq}
For $1\leq i\leq n$, we have
$$
K_{\bi}^2=q^{-1}K_i \begin{bmatrix}
K_i\\1
\end{bmatrix}-q^{-1}(q-q^{-1}) \begin{bmatrix}
K_i\\2
\end{bmatrix}.
$$
\end{lem}
\begin{proof}
By~\eqref{Kt}, we have
$$
\begin{bmatrix}
K_i\\1
\end{bmatrix}=\frac{K_i-K_i^{-1}}{q-q^{-1}},\quad
\begin{bmatrix}
K_i\\2
\end{bmatrix}
=\frac{K_i-K_i^{-1}}{q-q^{-1}}\cdot\frac{K_iq^{-1}-K_i^{-1}q}{q^2-q^{-2}}.
$$
Then a direct calculation shows that
$$
q^{-1}K_i \begin{bmatrix}
K_i\\1
\end{bmatrix}-q^{-1}(q-q^{-1})\begin{bmatrix}
K_i\\2
\end{bmatrix}
=\frac{K_i^2-K_i^{-2}}{q^2-q^{-2}}.
$$
Hence the lemma is proved using the relation (QQ1) in Proposition~\ref{presentqUq}.
\end{proof}

Observe that by Lemma~\ref{q-ppoo} and applying the anti-automorphism~\eqref{Omega}, we have
\begin{equation}\label{Xoddsq}
\aligned
\oX^2_{i,j}&=-\frac{q-q^{-1}}{q+q^{-1}}X_{i,j}^2=-(q-q^{-1})X_{i,j}^{(2)},\\
\oX^2_{j,i}&=\frac{q-q^{-1}}{q+q^{-1}}X_{j,i}^2=(q-q^{-1})X_{j,i}^{(2)},
\endaligned
\end{equation}
for $1\leq i<j\leq n$.
%and hence it suffices to consider the elements $E_{\bj}, F_{\bj}$
%with power at most one whenever involved for $1\leq j\leq n-1$.

Similar to \eqref{eAfA}, we introduce the following elements: for $A\in M_n(\N|\Z_2)$,
\begin{align}\label{EAFA}
E_{A^+}=\prod_{1\leq i<j\leq n}\big(X_{i,j}^{(a^0_{ij})}\oX^{a^1_{ij}}_{{i,j}}\big)\,\text{ and }\,
F_{A^-}=\prod_{1\leq i<j\leq n}\big(X_{j,i}^{(a^0_{ji})}\oX^{a^1_{ji}}_{j,i}\big),
\end{align}
where $A_0=(a^0_{ij}), A_1=(a^1_{ij})$, and the order in products is the same as those in \eqref{eA+} and \eqref{fA-}, respectively.

Recall that for $\sigma\in\Z^n$ and $A\in M_n(\N|\Z_2)$, we have $K_{\sigma}=K_1^{\sigma_1}\cdots K_n^{\sigma_n}$ and $\ov{K}_{A^0_1}=K_{\bar{1}}^{a^1_{11}}\cdots K_{\bar{n}}^{a^1_{nn}}$, where $A_1=(a^1_{ij})$.
For $\bsb=(b_1,\ldots,b_n)\in\N^n$, define
$$
\begin{bmatrix}\bsK\\ \bsb\end{bmatrix}=\prod^n_{i=1}\begin{bmatrix}K_i\\ b_i\end{bmatrix}.
$$

\begin{prop}\label{q-intPBW} \begin{enumerate}
\item There is a $\mcZ$-linear isomorphism
$
\qUZ\cong U^-_{q,\mcZ}\otimes U^0_{q,\mcZ}\otimes U^+_{q,\mcZ}.
$

\item The following set
$$
\Bigg\{\fkm^q_{A,\tau}:=F_{A^-} \big(\prod^n_{i=1}K^{\tau_i}_i\big)\begin{bmatrix}\bsK\\ A^0_0\end{bmatrix}\ov{K}_{A^0_1}E_{A^+}\mid
A\in M_n(\N|\Z_2), \tau=(\tau_1,\ldots,\tau_n)\in\Z_2^n, \Bigg\}
$$
is a $\mcZ$-basis for $U_{q,\mcZ}$.
\end{enumerate}
\end{prop}
\begin{proof}
By an argument parallel to the proof of Proposition~\ref{intPBW}(1), one can check that
part (1) of the proposition holds.

By Lemma~\ref{PBW-X}, the set $\{E_{A^+}\mid A\in M_n(\N|\Z_2), 0=A^-=A^0\}$ is linearly independent.
Meanwhile by Lemma~\ref{q-ppoo}, \eqref{Xoddsq}, and Propositions~\ref{q-div-ppee} and ~\ref{q-div-ppeo}, we obtain that
$U^+_{q,\mcZ}$ is spanned by the set $\{E_{A^+}\mid A\in M_n(\N|\Z_2), 0=A^-=A^0\}$ using a proof similar to \cite[Proposition 1.13]{Lu}.
Hence $\{E_{A^+}\mid A\in M_n(\N|\Z_2), 0=A^-=A^0\}$ a $\mcZ$-basis for $U^+_{q,\mcZ}$.
Similarly, one can show that the set $\{F_{A^-}\mid A\in M_n(\N|\Z_2), 0=A^+=A^0\}$ is a $\mcZ$-basis for $U^-_{q,\mcZ}$.
Finally  by \eqref{qUqev} the subalgebra $U^0_{q,\mcZ}(\ev{\mfq})$ of $U^0_{q,\mcZ}$ generated by $K_i^{\pm}, \begin{bmatrix}K_i\\ t\end{bmatrix}$ with $1\leq i\leq n$ and $t\in\N$ can be identified with the corresponding algebra $U^0_{q,\mcZ}(\mathfrak{gl}(n))$
associated to $\mathfrak{gl}(n)$. Then it follows from the classical case that $U^0_{q,\mcZ}(\ev{\mfq})$
has a $\mcZ$-basis given by $\{\big(\prod^n_{i=1}K^{\tau_i}_i\big)\begin{bmatrix}\bsK\\ \bsb\end{bmatrix}\mid \bsb\in\N^n, \tau_i\in\Z_2, 1\leq i\leq n\}$.
This together with Lemma~\ref{Kbarsq} and the relation (QQ1) in Proposition~\ref{presentqUq} we obtain that
$U^0_{q,\mcZ}$ is spanned by the set
$\{\big(\prod^n_{i=1}K^{\tau_i}_i\big)\begin{bmatrix}\bsK\\ \bsb\end{bmatrix}\ov{K}_{D}\mid \bsb\in\N^n, D\in\Z_2^n, \tau_i\in\Z_2, 1\leq i\leq n\}$.
Again by Lemma~\ref{PBW-X}, this set is linearly independent and hence it is a $\mcZ$-basis for $U^0_{q,\mcZ}$.
Putting all together, part (2)  follows from part (1).
\end{proof}

\begin{rem}\label{twisted degree}
Observe the commutation formula in Lemma~\ref{q-pnoo}(1). If we use a degree function similarly defined as in \eqref{degree-mA}, then $[E_\bi,F_\bi]$ is a linear combination of monomials of degree
not strictly less. In order to obtain the analog of Lemma~\ref{express-m}
in quantum case, we introduce the following {\it twisted degree} function:
\begin{align}
{\rm deg}'(X^{(s)}_{i,j})=2s|j-i|,\; {\rm deg}'(\oX_{i,j})=2|j-i|,\;{\rm deg}'(K_{\bar i})=1,\text{ and }\deg'(K_i)=0 \label{q-degree}
\end{align}
for $1\leq i\neq j\leq n,  s\in\N$.
Then we have
\begin{equation}\label{q-degree-content}
{\rm deg}'(\fkm^q_{A,\tau})={\rm deg}'(A):=(a^1_{11}+\cdots+a^1_{nn})+\sum_{1\leq i\neq j\leq n}2(a_{ij}+a_{ji})|j-i|,
\end{equation}
where $\tau\in\Z_2^n$ and $A=(A_0,A_1)\in M_n(\N|\Z_2)$ with $A_0=(a^0_{ij}),A_1=(a^1_{ij})$ and $a_{ij}=a^0_{ij}+a^1_{ij}$.
Let $\mf M_q\subseteq U_q$ be the set of monomial $\fkm_q$ in
$X^{(s)}_{i,j},\oX_{i,j},K_{\bar i},$ $ K_i^{\pm1}, \begin{bmatrix}K_i\\ s\end{bmatrix}$ for $1\leq i\neq j\leq n,  s\in\N$.
Then similar to non-quantum case, the degree of a non-zero monomial $\fkm_q\in\mf M_q$ is defined accordingly.
With $\deg'$, we see that commuting the product of two generators $x,y\in \mc G_q$ from different triangular parts yields a linear combination of certain monomials with smaller twisted degree. See Lemma~\ref{q-pnoo}, Propositions~\ref{q-div-pnee} and \ref{q-div-pneo} for $x,y$ in different $\pm$-part, noting
the following holds for $i<j,k>l$ satisfying $i\leq l<j$ or $l\leq i<k$:
\begin{align*}
2<2|j-i|+2|k-l|,~ &\text{ if }l=i,k=j,\\
2|k-j|+2|l-i|+1<2|j-i|+2|k-l|,~ &\text{ if }l\neq i\text{ or }k\neq j,
\end{align*}
and see
Propositions \ref{q-XKbar} and  \ref{q-div-XKbar} for $x,y$ in different 0-part and $\pm$-part.
\end{rem}

The next application is a spanning set of a certain quotient superalgebra of $\qUZ$.
Similar to the non-quantum situation, let $I_q$ be the ideal of the quantum superalgebra $\qUq$ given by
\begin{align}\label{ideal-Iq}
I_q=\langle K_1\cdots K_n-q^r,  (K_i-1)(K_i-q)\cdots (K_i-q^{r}), K_{\bi}(K_i-q)\cdots (K_i-q^{r})\,\forall i\rangle.
\end{align}
Define
\begin{equation}\label{Pq}
\aligned
\qPnr=U_q/I_q,\quad U_q(n,r)_{\mcZ}=U_{q,\mcZ}/I_q\cap U_{q,\mcZ},\\
U^0_q(n,r)=U^0_q/I_q\cap U_q^0,\quad U^0_q(n,r)_{\mcZ}=U^0_{q,\mcZ}/I_q\cap U^0_{q,\mcZ}.
\endaligned
\end{equation}
Clearly, $U_q(n,r)$ is a homomorphic image of $U_q(\mathfrak{q}(n))$.

{\it As before,
 we also denote by abuse of notation the image of $K_i,E_j,F_j$ etc. in $U_q(n,r)$
by the same letters.} We will write, for $\la\in\La(n,r)$,
$$
1_\la=\begin{bmatrix}\bsK\\\la\end{bmatrix}\in U^0_q(n,r)_\mcZ.
$$

\begin{prop}\label{q-property}
The following holds in $U^0_q(n,r)_\mcZ$:
\begin{enumerate}
\item The set $\{1_\la\mid \la\in\La(n,r)\}$ is a set of pairwise orthogonal central idempotents in $U^0_q(n,r)_\mcZ$ satisfying $1=\sum_{\la\in\La(n,r)}1_\la.$

\item $\begin{bmatrix}\bsK\\ \bsb\end{bmatrix}=0$ for $\bsb\in\N^n$ with $|\bsb|>r$.

\item For $1\leq i\leq n, \la\in\La(n,r)$ and $\bsb\in\N^n$, we have $K_i^{\pm}1_\la=q^{\pm\la_i}1_\la$,
\begin{align*}
 \begin{bmatrix}K_i;c\\t\end{bmatrix}1_\la=\begin{bmatrix}\la_i+c\\t\end{bmatrix}1_\la,~
\begin{bmatrix}\bsK\\ \bsb\end{bmatrix}1_\la=\begin{bmatrix}\la\\ \bsb\end{bmatrix}1_\la, \text{ and }
\begin{bmatrix}\bsK\\ \bsb\end{bmatrix}=\sum_{\la\in\La(n,r)}\begin{bmatrix}\la\\ \bsb\end{bmatrix}1_\la,
\end{align*}
where $\ds\begin{bmatrix}\la\\ \bsb\end{bmatrix}=\prod^n_{i=1}\begin{bmatrix}\la_i\\ b_i\end{bmatrix}$.
\item Suppose $1\leq i\leq n$ and $\la\in\La(n,r)$. If $\la_i=0$, then
$
K_{\bar i}1_\la=0.
$

\item The $\mcZ$-algebra $U^0_q(n,r)_\mcZ$ is spanned by
$$
\{1_\la \ov{K}_D\mid \la\in\La(n,r), D\in\Z_2^n, D_i\leq \la_i, 1\leq i\leq n\}.
$$
\end{enumerate}
\end{prop}
\begin{proof}
By the identification~\eqref{qUqev},  since the elements $K_1K_2\cdots K_n-q^r$ and $(K_i-1)(K_i-q)\cdots (K_i-q^r), ~1\leq i\leq n$
are contained in the ideal $I_q$, we observe that there is a natural algebra homomorphism from the algebra $\mathbf{T}^0$
introduced in \cite[(8.1)]{DG} to $ U^0_q(n,r)$, which sends the generator $K_i$ in \cite{DG} to our generator $K_i$ for $1\leq i\leq n$.
Hence parts (1)-(3) follow directly from \cite[Propositions 8.2-8.3]{DG}.

To prove part (4), firstly we observe that the following holds
\begin{align*}
(K_i-q)\cdots(K_i-q^r)1_\la=(q^{\la_i}-q)\cdots(q^{\la_i}-q^r)1_\la
\end{align*}
for $1\leq i\leq n$ by part (3) of the proposition.
Now assume $\la_i=0$. Then we have
$$
(K_i-q)\cdots(K_i-q^r)1_\la=(1-q)\cdots(1-q^r)1_\la.
$$
Hence
$$
(1-q)\cdots(1-q^r)K_{\bi}1_\la=K_{\bi}(K_i-q)\cdots(K_i-q^r)1_\la=0,
$$
which is due to the fact  $K_{\bi}(K_i-q)\cdots(K_i-q^r)\in I_q\cap U_{q,\mc Z}^0$.
Therefore $K_{\bi}1_\la=0$.

Finally, by part (3), the relation~(QQ1) and Lemma~\ref{Kbarsq}, we know the algebra $U^0_q(n,r)_\mcZ$ is spanned by the set
$\{\ov{K}_{D}1_\la=1_\la \ov{K}_{D}\mid\la\in\La(n,r), D\in\Z_2^n\}$.
By part (4), we have $\ov{K}_{D}1_\la=0$ if there exists $1\leq i\leq n$ such that $D_i=1$ and $\la_i=0$.
Hence part (5) holds.
\end{proof}

\begin{prop}\label{q-root-idem}
Suppose $1\leq i\leq n,\al\in\Phi$ and $\la\in\La(n,r)$.
Then we have in $U_q(n,r)_\mcZ$:
\begin{align*}
X_\al1_\la&=\left\{
\begin{array}{ll}
1_{\la+\al}X_\al,&\text{ if }\la+\al\in\La(n,r),\\
0,&\text{ otherwise},
\end{array}
\right.
\quad
\oX_{\al}1_\la=\left\{
\begin{array}{ll}
1_{\la+\al}\oX_{\al},&\text{ if }\la+\al\in\La(n,r),\\
0,&\text{ otherwise},
\end{array}
\right.\\
1_\la X_\al&=\left\{
\begin{array}{ll}
X_\al1_{\la-\al},&\text{ if }\la-\al\in\La(n,r),\\
0,&\text{ otherwise},
\end{array}
\right.
\quad
1_\la \oX_{\al}=\left\{
\begin{array}{ll}
\oX_{\al}1_{\la-\al},&\text{ if }\la-\al\in\La(n,r),\\
0,&\text{ otherwise},
\end{array}
\right.\\
K_{\bar i}1_\la&=1_\la K_{\bar i}.
\end{align*}
\end{prop}
\begin{proof}
The last equality follows from the relation (QQ1).
The proof of the remaining formulas
is parallel to that of Proposition~\ref{root-idem}.
We skip the detail here.
\end{proof}

As in the non-quantum case, set
$$
u^q_C=F_{C^-}\ov{K}_{C^0_1}E_{C^+}\in \qUZ\,\text{ and }\, \fku^q_{(C,\la)}=F_{C^-}1_\la\ov{K}_{C^0_1}E_{C^+}\in U_q(n,r)_\mcZ
$$ for $C\in M_n(\N|\Z_2)'$ and $\la\in\La(n,r)$.
Then, by \eqref{q-degree-content} we have
$
{\rm deg}'(u^q_C)={\rm deg}'(C).
$
By Proposition~\ref{q-root-idem}, we have in $U_q(n,r)_\mcZ$
\begin{equation}\label{q-uclan}
\fku^q_{(C,\la)}=1_{\la'}u^q_C=u^q_C1_{\la''} \text{ if } \fku^q_{(C,\la)}\neq 0,
\end{equation}
where $\la'=\la+\ro(C^-)-\co(C^-)$ and $\la''=\la-\ro(C^+)+\co(C^+)$.

\begin{lem}\label{express-mq} $(1)$
Let $0\neq\fkm_q\in\mf M_q$. Then, in $U_q(n,r)_\mcZ$,  $\fkm_q$ can be written as a linear combination of $\fku^q_{(C,\la)}$
for $C\in M_n(\N|\Z_2)', \la\in\La(n,r)$ such that ${\rm deg}'(C)\leq {\rm deg}'(\fkm_q)$.

$(2)$ Suppose $C\in M_n(\N|\Z_2)'$. Then  we have in $U_q(n,r)_\mcZ$
$$
F_{C^-}\ov{K}_{C^0_1}E_{C^+}=\pm E_{C^+}\ov{K}_{C^0_1}F_{C^-}+\sum \eta^{C}_{(G,\la)} \fku^q_{(G,\la)},
$$
for some $\eta^{C}_{(G,\la)}\in\mc Z$, the summation is over $G\in M_n(\N|\Z_2)'$ and $\la\in\La(n,r)$
such that ${\rm deg}'(G)<{\rm deg}'(C)$.
\end{lem}
\begin{proof}
As observed in Remark \ref{twisted degree}, commuting two elements $x,y\in\mc G_q$ from different triangular parts strictly decreases the twisted degree.
Thus, with Proposition~\ref{q-intPBW} at hand, the arguments in the proof of Lemma~\ref{express-m} are applicable here.
We leave the detail to the reader.
\end{proof}

Set
$$
\sBq=\{\fku^q_{(C,\la)}\mid
C\in M_n(\N|\Z_2)',
\la\in\La(n,r), \chi(C)\preceq \la\}.
$$
Then similar to the non-quantum case, the following holds:
\begin{equation}\label{Bq}
\sBq=\{\fku^q_A:=F_{A^-}1_{\chi(A)}\ov{K}_{A^0_1}E_{A^+}\mid
A\in M_n(\N|\Z_2)_r\}.
\end{equation}

\begin{prop}\label{Yq-span}
The algebra $U_q(n,r)_\mcZ$ is spanned by the set $\sBq$.
\end{prop}
\begin{proof}
With Lemma~\ref{express-mq} at hand, it is easy to see
the arguments in the proof of Proposition~\ref{Y-span}
are also applicable in the quantum case.
We leave the detail to the reader.
\end{proof}

\section{A presentation for the quantum queer Schur superalgebras $\qQnr$}
By base change to the superspace $V$ of \S2, we now assume that $V$ is a superspace over the field $\Q(q)$.
Thus, $V$ has basis $\{v_1,\ldots,v_n, v_{-1},...,v_{-n}\}$.
Following \cite{Ol}, we set
\begin{equation}\label{eq:operatorS}
\aligned
\Theta&=  \sum_{1\leq a\leq n}(E_{-a,a}-E_{a,-a}),\\
T&=  \sum_{i,j\in I(n|n)}\text{sgn}(j)E_{i,j}\otimes E_{j,i},\\
S&= \sum_{i\leq j\in I(n|n)}S_{i,j}\otimes E_{i,j}\in\End_{\Q(q)}
(V^{\otimes 2}),
\endaligned
\end{equation}
where $S_{i,j}$ for $i\leq j$ are defined as follows:
\begin{equation}\label{S}
\aligned
S_{a,a}&=\sum^n_{b=1}q^{\delta_{ab}}(E_{b,b}+E_{-b,-b})=1+(q-1)(E_{a,a}+E_{-a,-a}),\quad 1\leq a\leq n,\\
S_{-a,-a}&=\sum^n_{b=1}q^{-\delta_{ab}}(E_{b,b}+E_{-b,-b})=1+(q^{-1}-1)(E_{a,a}+E_{-a,-a}),\quad1\leq a\leq n, \\
S_{b,a}&=(q-q^{-1})(E_{a,b}+E_{-a,-b}),\quad 1\leq b<a\leq n, \\
S_{-b,-a}&=-(q -q^{-1})(E_{a,b}+E_{-a,-b}),\quad 1\leq a<b\leq n, \\
S_{-b,a}&=-(q -q^{-1}) (E_{-a,b}+E_{a,-b}),\quad 1\leq a, b\leq n.
\endaligned
\end{equation}

To endomorphisms $Y\in\End_{\Q(q)}(V)$ and
$Z=\sum_{t}H_{t}\otimes I_{t}\in\End_{\Q(q)}(V)^{\otimes
2}=\End_{\Q(q)}(V^{\otimes
2})$, we associate the following elements in $\End_{\Q(q)}(V^{\otimes
r})$:
\begin{align*}
Y_{(k)}&={\rm id}^{\otimes k-1}\otimes Y\otimes {\rm id}^{\otimes r-k},\qquad 1\leq k\leq r,\\
Z_{(j,k)}&=\sum_{t}(H_{t})_{(j)}(I_{t})_{(k)}, \qquad 1\leq j\neq k\leq r.
\end{align*}

Define $\Phi_1:\qUq\rightarrow \End_{\Q(q)} (V)$ by
\begin{equation}  \label{PhiS}
  \Phi_1(L_{i,j})= S_{i,j}, \qquad \text{ for } i, j\in I(n|n), i\leq j.
\end{equation}
It is known \cite[\S 4]{Ol} that $\Phi_1$ is an algebra
homomorphism and hence defines a representation $(\Phi_1,
V)$ of $\qUq$.
Then using the comultiplication~\eqref{comult}, we obtain a representation $(\Phi_r, V^{\otimes r})$,
which can be viewed as a deformation of the
representation $(\phi_r, V^{\otimes r})$ of $U(\mfq)$.
Define the {\em quantum queer Schur superalgebra $\qQnr$} (or {\em $q$-Schur superalgebra of queer type}) by
\begin{equation}\label{Qqnr}
\qQnr=\Phi_r(\qUq),
\end{equation}
that is, the image of the homomorphism $\Phi_r$.
As before, the superalgebra $\qQnr$ can be viewed as a quotient of $\qUq$.
%Let $\mc{Q}_{\mcZ}(n,r)$ (resp. $\mc{Q}^0_{\mcZ}(n,r)$) be the image of $U_{\mcZ}$ (resp. $U^0_{\mcZ}$)
%under the map $\Phi_r$.

Similar to the classical case, there exists another explanation of the superalgebra $\qQnr$
via an analog of Jimbo-Schur duality for $\qUq$ as follows.
The Hecke-Clifford algebra $\HC$ is the
associative superalgebra over $\Q(q)$ with
even generators $T_1,\ldots,T_{r-1}$ and odd generators
$c_1,\ldots,c_r$ subject to \eqref{Cl} and the following relations:
\begin{align*}
(T_i-q)(T_i+q^{-1})=0,\;\quad& \qquad T_ic_j =c_jT_i\; (j\neq i,i+1),\\
T_iT_{i'}=T_{i'}T_i \;(|i-i'|>1), &  \qquad T_ic_i=c_{i+1}T_i,\\
T_iT_{i+1}T_i =T_{i+1}T_iT_{i+1},\quad\;\; & \qquad T_ic_{i+1}=c_iT_i-(q-q^{-1})(c_i-c_{i+1}),
\end{align*}
for all $1\leq i,i'\leq r-1$ and $1\leq j\leq r$.

Let $\ov{S}=TS\in\End_{\Q(q)}(V^{\otimes 2})$.
Then $\bar{S}_{(j,j+1)}\in\End_{\Q(q)}(V^{\otimes r})$.
It follows from \cite[Theorems 5.2-5.3]{Ol} that there exists a representation
$(\Psi_r, V^{\otimes r})$ of the Hecke-Clifford algebra $\HC$ defined by
$$
\Psi_r(T_j)=\bar{S}_{(j,j+1)}\quad\text{ and }\quad \Psi_r(c_k)=\Theta_{(k)},
$$
where $1\leq j\leq r-1$ and $1\leq k\leq r$.
Moreover, $\Phi_r(\qUq)=\End_{\HC}(V^{\otimes r})$ and hence, by \eqref{Qqnr}, the following holds
\begin{equation*}
\qQnr=\End_{\HC}(V^{\otimes r}).
\end{equation*}
Let $V_\Z$ (resp. $V_{\mc Z}$) be the free $\Z$-supermodule (resp. $\mc Z$-supermodule)
with basis $v_1,\ldots,v_n,$ $v_{-1},\ldots,v_{-n}$.
Let $\mathcal{H}^c_{r,\mcZ}$ (resp. $\mf H^c_{r,\Z}$) be the corresponding Hecke-Clifford algebra (resp. Sergeev superalgebra)
defined over $\mcZ$ (resp. $\Z$), and let
\begin{equation}\label{QsZ}
{\mc Q}_q(n,r)_{\mcZ}=\End_{\mathcal{H}^c_{r,\mcZ}}(V_{\mc Z}^{\otimes r}).
\end{equation}
Then the specialization of $\mathcal{H}^c_{r,\mcZ}$ at $q=1$ coincides the Sergeev superalgebra $\mf H^c_{r,\Z}$. Clearly, $\Q(q)\otimes_{\mcZ}\End_{\mathcal{H}^c_{r,\mcZ}}(V_{\mc Z}^{\otimes r})\subseteq\qQnr$.
Meanwhile by \cite[Theorem 4.5]{BK} we have $\Qnr=\Q\otimes_{\Z}\End_{\mf H^c_{r,\Z}}(V_\Z^{\otimes r})$.
As $\End_{\mathcal{H}^c_{r,\mcZ}}(V_{\mc Z}^{\otimes r})$ specializes to $\End_{\mf H^c_{r,\Z}}(V_\Z^{\otimes r})$ at $q=1$, we
have $\dim\Q(q)\otimes_{\mcZ}\End_{\mathcal{H}^c_{r,\mcZ}}(V_{\mc Z}^{\otimes r})\geq \dim\Q\otimes_{\Z}\End_{\mf H^c_{r,\Z}}(V_\Z^{\otimes r})$ and hence
\begin{equation}\label{dimQqnr}
\dim\qQnr\geq \dim\Qnr.
\end{equation}

\begin{lem}\label{Phir-cartan}%{Phi1-L-ba}
\begin{enumerate}\item Let $a,j\in I(n|n)$ and $1\leq b\leq n$. If $-b<a$ and $|j|\neq b$, then
$$
\big(\Phi_1(L_{-b,a})\big)(v_j)=0.
$$
\item Suppose $\udj\in I(n|n)^r$ and ${\rm wt}(\udj)=\mu$.
We have,  for $1\leq i\leq n$,
\begin{enumerate}
\item[(a)] $\Phi_r(K_i)(v_{\udj})=q^{\mu_i} v_{\udj}$;
\item[(b)] $\Phi_r(K_{\bar{i}})(v_{\udj})=0$ if $\mu_i=0$.
\end{enumerate}
\end{enumerate}
\end{lem}
\begin{proof}
Part (1) follows from definitions \eqref{S}, \eqref{PhiS} and \eqref{Eij}, since $|j|\neq b$ implies
\begin{align*}
\big(\Phi_1(L_{-b,a})\big)(v_j)=-(q-q^{-1})(E_{a,-b}+E_{-a,b})(v_j)=0.
\end{align*}

Similarly, by definitions \eqref{PhiS}, \eqref{S}, \eqref{Eij}, and \eqref{q-generator}, we obtain
\begin{align*}%\label{Phi1}
\Phi_1(K_i)(v_j)=\Phi_1(L_{i,i})(v_j)=\big(1+(q-1)(E_{i,i}+E_{-i,-i})\big)(v_j)=q^{\delta_{i,|j|}}v_j
\end{align*}
for $1\leq i\leq n$ and $j\in I(n|n)$.
Hence, the comultiplication \eqref{q-comult} on $U_q(\mfq)$ (and noting $K_i=L_{i,i}$) gives
\begin{align*}
\Phi_r(K_i)(v_{\udj})
=&(\Phi_1(K_i)\otimes \Phi_1(K_i)\otimes\cdots\otimes \Phi_1(K_i)) (v_{\udj})\\
=&\Big(\prod^r_{k=1}q^{\delta_{i,|j_k|}}\Big)v_{\udj}=q^{\mu_i} v_{\udj}\quad(\text{see \eqref{wt}}),
\end{align*}
proving part (2a).

Next, suppose $1\leq b\leq n$ and $\mu_b=0$. Then, by \eqref{wt}, $|j_k|\neq b$ for each $1\leq k\leq r$.
Suppose $\Phi_r(K_{\bar b})(v_{\udj})\neq 0$. Then $\Phi_r(L_{-b,b})(v_{\udj})=-(q-q^{-1})\Phi_r(K_{\bar b})(v_{\udj})\neq 0$
by \eqref{q-generator}. On the other hand,
by the comultiplication \eqref{comult} and the definition of $\Phi_r$,
%\begin{align*}
%\Delta(L_{-b,b})=\sum^{b}_{k_1=-b}\sum^{b}_{k_2=k_1}\cdots\sum^{b}_{k_{r-1}=k_{r-2}}
%L_{-b,k_{1}}\otimes L_{k_{1},k_{2}}\otimes\cdots\otimes L_{k_{r-2},k_{r-1}}\otimes L_{k_{r-1},b}
%\end{align*}
%and hence
\begin{equation*}\label{Phir-L-bb}
\aligned
\Phi_r(L_{-b,b})(v_{\udj})
=\bigg(\sum^{b}_{k_1=-b}\sum^{b}_{k_2=k_1}\cdots\sum^{b}_{k_{r-1}=k_{r-2}}
\Phi_1(&L_{-b,k_{1}})\otimes \Phi_1(L_{k_1,k_2})\otimes \cdots\otimes\\
& \Phi_r(L_{k_{r-2},k_{r-1}})\otimes\Phi_1(L_{k_{r-1},b})\bigg)(v_{\udj}).
\endaligned
\end{equation*}
Thus, there exist $-b\leq k_{1}\leq k_{2}\leq\cdots\leq k_{r-1}\leq b$
such that
\begin{equation}\label{last}
\Phi_1(L_{-b,k_1})(v_{j_1})\neq0,\quad\Phi_1(L_{k_1,k_2})(v_{j_2})\neq0,\quad\ldots,\quad\Phi_1(L_{k_{r-1},b})(v_{j_r})\neq0.
\end{equation}
Since $|j_k|\neq b$, repeatedly applying part (1) to the first, second, ... and second last inequality forces $-b=k_1$, $-b=k_2$, ... and $-b=k_{r-1}$.
Then by part (1) again, $\Phi_1(L_{k_{r-1},b})(v_{j_r})=\Phi_1(L_{-b,b})(v_{j_r})= 0$ (since $|j_r|\neq b$),
which is in contradiction to the last inequality in \eqref{last}.
%Hence, the last inequality in \eqref{last} becomes,  by part (1) again, $\Phi_1(L_{-b,b})(v_{j_r})=\Phi_1(L_{-b,b})(v_{j_r})= 0$ (since $|j_r|\neq b$), a contradiction.
Hence, part (2) is verified.
\end{proof}

We are now ready to establish the quantum version
of Theorem \ref{presentQr}. Recall the ideal $I_q$ defined in \eqref{ideal-Iq}.
%show the ideal vanishes on the tensor space $V^{\otimes r}$.

\begin{theorem}\label{q-surjective}
The homomorphism $\Phi_r: \qUq\rightarrow \End_{\Q(q)}(V^{\otimes r})$ satisfies $I_q\subseteq\ker \Phi_r$ and induces an algebra isomorphism
\begin{equation}\label{ovPhir}
\ov{\Phi}_r: \qPnr\overset\sim\longrightarrow \qQnr.
\end{equation}
In particular, the Schur superalgebra $\qQnr$ is the associative superalgebra
generated by even generators
$ K^{\pm1}_i,E_j, F_j,$
and odd generators $K_{\bar i}, E_{\bar j}, F_{\bar j},$
with $1\leq i\leq n$ and $1\leq j\leq n-1$
subject to the relations {\rm (QQ1)-(QQ6)} together with the following extra relations:
\begin{itemize}

  \item[(QQ7)] $K_1\cdots K_n=q^r$;

   \item[(QQ8)] $(K_i-1)(K_i-q)\cdots(K_i-q^r)=0$, where $1\leq i\leq n$.\footnote{In Theorem \ref{presentQr}, the counterpart of (QQ8) is the relation $h_{i}(h_i-1)\cdots(h_i-r)=0$. This was not displayed, as it can easily be derived from (QS8).}
 \item[(QQ9)] $K_{\bi}(K_i-q)\cdots(K_i-q^r)=0$, where $1\leq i\leq n$;

\end{itemize}

\end{theorem}
\begin{proof}
Fix an arbitrary $\udj\in I(n|n)^r$.
Suppose ${\rm wt}(\udj)=\mu\in\La(n,r)$.
Then by Lemma~\ref{Phir-cartan}(1) we have
\begin{align*}
\Phi_r(K_1\cdots K_n)(v_{\udj})=&q^{\mu_1+\cdots+\mu_n}v_{\udj}=q^rv_{\udj},\\
\Phi_r((K_i-1)(K_i-q)\cdots (K_i-q^r))(v_{\udj})=&(q^{\mu_i}-1)(q^{\mu_i}-q)\cdots (q^{\mu_i}-q^r)v_{\udj}=0
\end{align*}
since $0\leq \mu_i\leq r$ for $1\leq i\leq n$ and $\mu_1+\cdots+\mu_n=r$.
This means
\begin{align}\label{ker-Phir-ev}
\Phi_r(K_1\cdots K_n-q^r)=0,\quad \Phi_r((K_i-1)(K_i-q)\cdots (K_i-q^r))=0
\end{align}
for $1\leq i\leq n$.
Meanwhile still by Lemma~\ref{Phir-cartan}(1) we obtain
\begin{align}\label{Phir-KioddKi}
\Phi_r(K_{\bi}(K_i-q)\cdots (K_i-q^r))(v_{\udj})=&(q^{\mu_i}-q)\cdots (q^{\mu_i}-q^r)\Phi_r(K_{\bi})(v_{\udj}).
\end{align}
Similar to non-quantum case, if $\mu_i=0$, then $\Phi_r(K_{\bi})(v_{\udj})=0$ due to Lemma~\ref{Phir-cartan}(2).
Otherwise we have $1\leq\mu_i\leq r$. Then $(q^{\mu_i}-q)\cdots (q^{\mu_i}-q^r)=0$. Putting together,
by \eqref{Phir-KioddKi} we obtain
$\Phi_r(K_{\bi}(K_i-q)\cdots (K_i-q^r))(v_{\udj})=0$.
Therefore, we have proved
\begin{align}\label{ker-Phir-odd}
\Phi_r(K_{\bi}(K_i-q)\cdots (K_i-q^r))=0.
\end{align}
In summary, by~\eqref{ideal-Iq}, \eqref{ker-Phir-ev} and \eqref{ker-Phir-odd},
the ideal $I_q$ is contained in the kernel of the homomorphism $\Phi_r$. Hence, $\Phi_r$ induces a surjective homomorphism
\begin{equation}\label{ovPhir}
\ov{\Phi}_r: \qPnr\twoheadrightarrow \qQnr.
\end{equation}
Now a dimensional comparison gives the required isomorphism since, by \eqref{dimQqnr} and Proposition~\ref{Yq-span},
$$
\dim \qPnr\leq |\sBq|=|\sB|=\dim  \Qnr\leq \dim \qQnr\leq \dim \qPnr,
$$
The last assertion follows from the definition of $\qPnr$ in~\eqref{Pq}.
\end{proof}

Propositions \ref{Yq-span} and \ref{q-property}(5) can now be strengthened as follows.
%By Theorem~\ref{presentQqr}, the following holds.
\begin{cor}
The set $\sBq$ defined in \eqref{Bq}
forms a $\mcZ$-basis for $U_q(n,r)_{\mcZ}$.
In particular, the set $\{1_\la\ov{K}_D\mid \la\in\La(n,r),D\in\Z_2^n, D_i\leq \la_i, 1\leq i\leq n\}$
is a $\mcZ$-basis for $U^0_q(n,r)_\mcZ$.
Moreover,  $\dim_{\Q(q)} \qQnr=\dim_\Q\Qnr$ and $\dim_{\Q(q)} \mc{Q}_q^0(n,r)=\dim_\Q\mc{Q}^0(n,r)$ are given as in Corollary~\ref{dim formulas}.
\end{cor}

\begin{rem} Since $\Phi_r(\qUZ)\subseteq {\mc Q}_q(n,r)_{\mcZ}$ by \eqref{Qqnr} and \eqref{QsZ}, the restriction of $\bar\Phi_r$ in \eqref{ovPhir} induces a superalgebra monomorphism
$\bar\Phi_{r,\mcZ}:U_q(n,r)_\mcZ\to{\mc Q}_q(n,r)_{\mcZ}$. It is natural to conjecture that $\bar\Phi_{r,\mcZ}$ is surjective.
\end{rem}

Similar to the non-quantum case, by an argument parallel to the proof of \cite[Theorem 3.4]{DG}
we have the following.
\begin{thm}
The quantum queer Schur superalgebra $\qQnr$ is the unitary associative superalgebra generated by
the even elements $1_\la, E_j, F_j $ and odd elements
$K_{\bi}, E_{\bj}, F_{\bj}$ for $\la\in\La(n,r),1\leq i\leq n, 1\leq j\leq n-1$
subject to the relations:
\begin{enumerate}
\item[(QQ1${}^\prime$)]
 $1_\la1_\mu=\delta_{\la,\mu}1_\la,\quad \sum_{\la\in\La(n,r)}1_\la=1$,
 $K_{\bi}1_\la=1_\la K_{\bi},$

\noindent $K_{\bar i}K_{\bar j}+K_{\bar j}K_{\bar i}=\ds\delta_{ij}\sum_{\la\in\La(n,r)}\frac{2(q^{2\la_i}-q^{-2\la_i})}{q^2-q^{-2}}1_\la$,

 \noindent $K_{\bi}1_\la=0$ if $\la_i=0$;

 \vspace{0.1in}

 \item[(QQ2${}^\prime$)]
$
E_j1_\la=\left\{
\begin{array}{ll}
1_{\la+\al_j}E_j,&\text{ if }\la+\al_j\in\La(n,r),\\
0,&\text{ otherwise},
\end{array}
\right.
$
$
E_{\bj}1_\la=\left\{
\begin{array}{ll}
1_{\la+\al_j}E_{\bj},&\text{ if }\la+\al_j\in\La(n,r),\\
0,&\text{ otherwise},
\end{array}
\right.
$
$
F_j1_\la=\left\{
\begin{array}{ll}
1_{\la-\al_j}F_j,&\text{ if }\la-\al_j\in\La(n,r),\\
0,&\text{ otherwise},
\end{array}
\right.
$
$
F_{\bj}1_\la=\left\{
\begin{array}{ll}
1_{\la-\al_j}F_{\bj},&\text{ if }\la-\al_j\in\La(n,r),\\
0,&\text{ otherwise},
\end{array}
\right.
$
$
1_\la E_j=\left\{
\begin{array}{ll}
E_j1_{\la-\al_j},&\text{ if }\la-\al_j\in\La(n,r),\\
0,&\text{ otherwise},
\end{array}
\right.
$
$
1_\la E_{\bj}=\left\{
\begin{array}{ll}
E_{\bj}1_{\la-\al_j},&\text{ if }\la-\al_j\in\La(n,r),\\
0,&\text{ otherwise},
\end{array}
\right.
$
$
1_\la F_j=\left\{
\begin{array}{ll}
F_j1_{\la+\al_j},&\text{ if }\la+\al_j\in\La(n,r),\\
0,&\text{ otherwise},
\end{array}
\right.
$
$
1_\la F_{\bj}=\left\{
\begin{array}{ll}
F_{\bj}1_{\la+\al_j},&\text{ if }\la+\al_j\in\La(n,r),\\
0,&\text{ otherwise};
\end{array}
\right.
$

 \vspace{0.1in}

 \item[(QQ3${}^\prime$)]
 $K_{\bi}E_i-qE_iK_{\bi}=\sum_{\la\in\La(n,r)}E_{\bi}q^{-\la_i}1_\la,\quad qK_{\bi}E_{i-1}-E_{i-1}K_{\bi}=-\sum_{\la\in\La(n,r)}q^{-\la_i}1_\la E_{\ov{i-1}},$

 \noindent $K_{\bi}F_i-qF_iK_{\bi}=-\sum_{\la\in\La(n,r)}q^{\la_i}F_{\bi}1_\la,\quad qK_{\bi}F_{i-1}-F_{i-1}K_{\bi}=\sum_{\la\in\La(n,r)}q^{\la_i}1_\la F_{\ov{i-1}},$

\noindent  $K_{\bi}E_{\bi}+qE_{\bi}K_{\bi}=\sum_{\la\in\La(n,r)}q^{-\la_i}E_i1_\la,\quad qK_{\bi}E_{\ov{i-1}}+E_{\ov{i-1}}K_{\bi}=\sum_{\la\in\La(n,r)}q^{-\la_i}1_\la E_{i-1},$

\noindent  $K_{\bi}F_{\bi}+qF_{\bi}K_{\bi}=\sum_{\la\in\La(n,r)}q^{\la_i}F_i1_\la,\quad qK_{\bi}F_{\ov{i-1}}+F_{\ov{i-1}}K_{\bi}=\sum_{\la\in\La(n,r)}q^{\la_i}1_\la F_{i-1},$

\noindent  $K_{\bi}E_j-E_jK_{\bi}=K_{\bi}F_j-F_jK_{\bi}=K_{\bi}E_{\bj}+E_{\bj}K_{\bi}=K_{\bi}F_{\bj}+F_{\bj}K_{\bi}=0$
for $j\neq i,i-1$;

 \vspace{0.1in}

 \item[(QQ4${}^\prime$)]
 $E_iF_j-F_jE_i=\delta_{ij}\sum_{\la\in\La(n,r)}[\la_i-\la_{i+1}]1_\la$,

\noindent  $E_{\bi}F_{\bj}+F_{\bj}E_{\bi}=\delta_{ij}\sum_{\la\in\La(n,r)}[\la_i+\la_{i+1}]1_\la+\delta_{ij}(q-q^{-1})K_{\bi}K_{\ov{i+1}}$,

 %$qE_{i+1}F_i-F_iE_{i+1}=E_iF_{i+1}-qF_{i+1}E_i=E_iF_j-F_jE_i=0,$
 %for $|i-j|>1$;

 %$qE_{\ov{i+1}}F_{\bi}+F_{\bi}E_{\ov{i+1}}=E_{\bi}F_{\ov{i+1}}+qF_{\ov{i+1}}E_{\bi}=E_{\bi}F_{\bj}+F_{\bj}E_{\bi}=0,$
 %for $|i-j|>1$;

\noindent  $E_iF_{\bj}-F_{\bj}E_i=\delta_{ij}\sum_{\la\in\La(n,r)}(q^{-\la_{i+1}}K_{\bi}-q^{-\la_i}K_{\ov{i+1}})1_\la$,

% $qE_{i+1}F_{\bi}-F_{\bi}E_{i+1}=E_iF_{\ov{i+1}}-qF_{\ov{i+1}}E_i=E_iF_{\bj}-F_{\bj}E_i=0$
% for $|i-j|>1$;

\noindent  $E_{\bi}F_j-F_jE_{\bi}=\delta_{ij}\sum_{\la\in\La(n,r)}(q^{\la_{i+1}}K_{\bi}-q^{\la_i}K_{\ov{i+1}})1_\la$,

 %$qE_{\ov{i+1}}F_i-F_iE_{\ov{i+1}}=E_{\bi}F_{i+1}-qF_{i+1}E_{\bi}=E_{\bi}F_j-F_jE_{\bi}=0$
% for $|i-j|>1$;
\end{enumerate}
and the relations {\rm (QQ5)-(QQ6)}.
\end{thm}

%{\bf Acknowledgement.} The second author would like to thank Weiqiang Wang, Alexander Kleshchev, and Jun Hu for some helpful discussions.

\end{document}